\documentclass[a4paper, 12pt, oneside, notitlepage]{amsart}
\usepackage[margin=3cm]{geometry}
\usepackage{amsmath,amssymb,amsthm,graphicx,mathrsfs,bbm,url,stmaryrd,amscd,cancel}
\usepackage{amsthm}
\usepackage{wrapfig}
\usepackage{enumitem}
\usepackage{mathtools}
\usepackage[utf8]{inputenc}
 \usepackage[T1]{fontenc}
\usepackage[usenames,dvipsnames]{color}
\usepackage[colorlinks=true,linkcolor=Red,citecolor=Green]{hyperref}
\usepackage[super]{nth}
\usepackage[open, openlevel=2, depth=3, atend]{bookmark}
\hypersetup{pdfstartview=XYZ}
\usepackage[font=footnotesize]{caption}
\usepackage{a4wide}
\usepackage{tikz}
\usepackage{tikz-cd}

\numberwithin{equation}{section}

\makeatletter
\newcommand\footnoteref[1]{\protected@xdef\@thefnmark{\ref{#1}}\@footnotemark}
\makeatother

\hypersetup{ colorlinks   = true, 
urlcolor  = blue, 
linkcolor    = Red, 
citecolor   = Green 
}

\def\I{{\mathrm I}}
\def\II{{\mathrm{II}}}
\def\III{{\mathrm{III}}}

\def\SS{{\mathbb S}}

\def\beq{\begin{equation}}
\def\eeq{\end{equation}}

\def\PSL{\mathrm{PSL}_2(\mathbb{R})}

\def\PSLc{\mathrm{PSL}_2(\mathbb{C})}

\def\det{\mathrm{det}}

\def\dS{\mathrm{dS}}

\catcode`\@=11
\def\eqalign#1{\null\,\vcenter{\openup1\jot \m@th %
\ialign{\strut\hfil$\displaystyle{##}\quad$&$\displaystyle{{}##}$\hfil\crcr#1\crcr}}\,}
\def\triplealign#1{\null\,\vcenter{\openup1\jot \m@th %
\ialign{\strut\hfil$\displaystyle{##}\quad$&$\displaystyle{{}##}$\hfil&$\displaystyle{{}##}$\hfil\crcr#1\crcr}}\,}
\def\multiline#1{\null\,\vcenter{\openup1\jot \m@th %
\ialign{\strut$\displaystyle{##}$\hfil&$\displaystyle{{}##}$\hfil\crcr#1\crcr}}\,}
\catcode`\@=12

\def\restriction#1#2{\mathchoice
              {\setbox1\hbox{${\displaystyle #1}_{\scriptstyle #2}$}
              \end{enumerate}
\restrictionaux{#1}{#2}}
              {\setbox1\hbox{${\textstyle #1}_{\scriptstyle #2}$}
              \restrictionaux{#1}{#2}}
              {\setbox1\hbox{${\scriptstyle #1}_{\scriptscriptstyle #2}$}
              \restrictionaux{#1}{#2}}
              {\setbox1\hbox{${\scriptscriptstyle #1}_{\scriptscriptstyle #2}$}
              \restrictionaux{#1}{#2}}}
\def\restrictionaux#1#2{{#1\,\smash{\vrule height .8\ht1 depth .85\dp1}}_{\,#2}}

\newcommand{\Z}{{\mathbb Z}}
\newcommand{\R}{{\mathbb R}}

\newcommand{\C}{{\mathbb C}}

\newcommand{\HH}{{\mathbb H}}

\newcommand{\BB}{\overline{\mathbb{B}}}
\newcommand{\Ss}{\mathbb{S}}

\newcommand{\sph}{{\mathrm{sph}}}
\newcommand{\cA}{{\mathcal A}}
\newcommand{\cC}{{\mathcal C}}
\newcommand{\cD}{{\mathcal D}}

\newcommand{\cF}{{\mathcal F}}

\newcommand{\cH}{{\mathcal H}}

\newcommand{\cV}{{\mathcal V}}
\newcommand{\cW}{{\mathcal W}}

\newcommand{\cT}{{\mathcal T}}

\newcommand{\cU}{{\mathcal U}}
\newcommand{\cK}{{\mathcal K}}
\newcommand{\cO}{{\mathcal O}}
\newcommand{\cR}{{\mathcal R}}

\newcommand{\eps}{\varepsilon}
\newcommand{\fhi}{\varphi}
\newcommand{\vide}{\emptyset}

\def\dans{\mathop{\subset}}

\newcommand{\sect}{\mathrm{sect}}
\newcommand{\Ric}{\mathrm{Ric}}
\newcommand{\dd}{\mathrm{d}}

\newcommand{\MD}{\mathrm{MD}}

\newcommand{\opEnd}{\mathrm{End}}

\newcommand{\vol}{\mathrm{vol}}

\newcommand{\jetX}{\mathcal{J}^2X}
\newcommand{\jete}{\mathcal{J}^2e}
\newcommand{\tr}{\mathrm{tr}}
\newcommand{\opK}{{\text{{\rm K}}}}

\newcommand{\moins}{\setminus}

\newcommand{\Hess}{\mathrm{Hess}}
\newcommand{\opSpec}{{\text{{\rm Spec}}}}
\newcommand{\opc}{{\text{{\rm c}}}}
\newcommand{\Id}{\mathrm{Id}}

\newcommand{\OpIm}{\mathrm{Im}}

\newcommand{\Symm}{\mathrm{Sym}}
\newcommand{\Zero}{\mathrm{Zero}}
\newcommand{\End}{\mathrm{End}}
\newcommand{\Area}{\mathrm{Area}}

\newcommand{\Op}{\mathrm{Op}}

\newcommand{\opD}{{\text{{\rm D}}}}
\newcommand{\opI}{\text{{\rm I}}}
\newcommand{\opII}{\text{{\rm II}}}
\newcommand{\opSym}{\mathrm{Sym}}

\newcommand{\closedtwoball}{\overline{\mathbb{B}}\vphantom{\mathbb{B}}^2}
\newcommand{\closedthreeball}{\overline{\mathbb{B}}\vphantom{\mathbb{B}}^3}
\newcommand{\twoball}{\mathbb{B}^2}
\newcommand{\opO}{\mathcal{O}}
\newtheorem{theorem}{Theorem}[section]
\newtheorem{remark}[theorem]{Remark}

\newtheorem{lemma}[theorem]{Lemma}
\newtheorem{defi}[theorem]{Definition}
\newtheorem{example}[theorem]{Example}
\newtheorem{prop}[theorem]{Proposition}
\newtheorem{corollary}[theorem]{Corollary}

\newtheorem{conjecture}{Conjecture}

\catcode`\@=11
\def\triplealign#1{\null\,\vcenter{\openup1\jot \m@th %
\ialign{\strut\hfil$\displaystyle{##}\quad$&$\displaystyle{{}##}$\hfil&$\displaystyle{{}##}$\hfil\crcr#1\crcr}}\,}
\def\multiline#1{\null\,\vcenter{\openup1\jot \m@th %
\ialign{\strut$\displaystyle{##}$\hfil&$\displaystyle{{}##}$\hfil\crcr#1\crcr}}\,}
\catcode`\@=12

\newif\ifshowcomments
\showcommentstrue

\begin{document}
	
\title[Boundary area rigidity in dimension three]{Foliated Plateau problems, surface Radon transforms, and boundary area rigidity in dimension three}

\author[Alvarez]{S\'ebastien Alvarez}
\address{CMAT, Facultad de Ciencias, Universidad de la Rep\'ublica, \& IRL-IFUMI (CNRS)\\
	Igua 4225 esq. Mataojo. Montevideo, Uruguay.\\}
\email{salvarez@cmat.edu.uy}

\author[Lefeuvre]{Thibault Lefeuvre}
\address{Université Paris-Saclay, CNRS, Laboratoire de mathématiques d’Orsay, 91405, Orsay, France.}
\email{thibault.lefeuvre1@universite-paris-saclay.fr}

\author[Lowe]{Ben Lowe}
\address{University of Chicago, Department of Mathematics, Chicago IL 60637, USA}
\email{loweb24@gmail.com}

\author[Smith]{Graham A. Smith}
\address{Departamento de Matemática, PUC-Rio, Rua Marquês de São Vicente, 225, Rio de Janeiro, 22451-900, RJ, Brazil.}
\email{grahamandrewsmith@gmail.com}

\date{\today}

\begin{abstract}
The present paper studies integral geometry problems on three-dimensional Riemannian balls, where integration is performed over minimal surfaces or, more generally, $\Phi$-surfaces defined by an elliptic curvature functional $\Phi$. The space of all $\Phi$-surfaces spanned by round circles on the boundary is a three-dimensional manifold, which we call the \emph{space of circles}. We show that, when the metric is \emph{$\Phi$-simple} - a notion which extends to this setting the notion of simple metrics in the geodesic case -, the Gauss lifts of the $\Phi$-surfaces define a foliation of the unit tangent bundle that should be viewed as a two-dimensional analogue of the standard geodesic foliation. This is achieved by solving a \emph{foliated Plateau problem} on the ball. We then analyze the associated \emph{surface Radon transform} corresponding to integration along the surfaces and show that it has a finite-dimensional kernel; we also prove that it is injective for an open and dense set of metrics. In the special case of a foliation by minimal surfaces, we apply these results to solve the following boundary area rigidity problem: does the collection of areas of the minimal surfaces determine the metric up to isometry?
\end{abstract}

\maketitle

\section{Introduction}

\subsection{Background}
In its simplest form, the present paper introduces and studies an analogue of the classical \emph{boundary distance rigidity problem}, in which lengths of geodesics are replaced by \emph{areas of minimal surfaces}. We recall that the boundary rigidity problem asks whether \emph{a compact Riemannian manifold with boundary is uniquely determined, up to isometry, by the distances it induces between boundary points}. This problem was introduced by Michel \cite{Michel-81}, who conjectured that simple manifolds are boundary rigid. This conjecture was proven in the $2$-dimensional case by Pestov--Uhlmann \cite{Pestov-Uhlmann-05}, following earlier results of Croke \cite{Croke-90} and Otal \cite{Otal-90-2}, as well as in several higher-dimensional model cases, including the hemisphere \cite{Michel-81}, Euclidean domains \cite{Gromov-83}, and hyperbolic domains \cite{Croke-04}, the latter relying on the Besson--Courtois--Gallot theorem \cite{Besson-Courtois-Gallot-95}. Further progress includes results of Stefanov--Uhlmann \cite{Stefanov-Uhlmann-05}, Burago--Ivanov \cite{Burago-Ivanov-10, Burago-Ivanov-13}, and Stefanov--Uhlmann--Vasy \cite{Stefanov-Uhlmann-Vasy-21}, which extend the theory beyond the simple setting.

\subsubsection{Boundary area rigidity}

We consider the following analogue of the boundary distance rigidity problem: \emph{is a compact Riemannian manifold with boundary uniquely determined, up to isometry, by the areas of the embedded minimal surfaces that it contains?} This problem was first studied by Alexakis--Balehowsky--Nachman \cite{Alexakis-Balehowsky-Nachman-20} who, motivated in part by its connections to the AdS/CFT correspondence, solved it for $3$-dimensional balls under certain geometric hypotheses, such as closeness to the Euclidean metric, or suitable ``thinness'' or ``straightness'' conditions. It was then studied by C{\^a}rstea--Liimatainen--Tzou \cite{Carstea-Liimatainen-Tzou-24} in the $3$-dimensional case, using connections with Calder\'on-type inverse problems. More recently, Busch--Liimatainen--Salo--Tzou \cite{Busch-Liimatainen-Salo-Tzou-25} studied the higher-dimensional case, obtaining a local rigidity result, with respect to conformal variations of the metric which vanish to infinite order on the boundary, for analytic metrics under suitable ``ampleness'' conditions.

The families of minimal surfaces used in \cite{Alexakis-Balehowsky-Nachman-20} are infinite-dimensional, and in \cite{Carstea-Liimatainen-Tzou-24} and \cite{Busch-Liimatainen-Salo-Tzou-25}, although the authors reduce to finite-dimensional families, these families are non-explicit and of non-optimal dimension (at least $6$). This contrasts markedly with the geodesic case, where the space of geodesics is typically identified with $\partial X\times\partial X\setminus\Delta$, which has dimension $2\dim(X)-2$. We refer the reader to \S\ref{ssection:champions} for a detailed comparison of our results with the work cited above.

\subsubsection{Content of the present paper}

In the present paper, restricting to $3$-dimensional ambient spaces, we prove local and global rigidity results, with respect to conformal variations of the metric, for the restriction of the area spectrum to certain explicit $3$-dimensional families of minimal surfaces. These families will be of optimal dimension for uniquely determining the metric.

Whenever the metric is \emph{area simple} - a notion which is the analogue of the simplicity condition for metrics in the geodesic case\footnote{A Riemannian manifold is \emph{simple} if its boundary is strictly convex, and it has no conjugate points and trapped geodesics. See, for example, \cite[Section 3.8]{Paternain-Salo-Uhlmann-23}.} (see Definition \ref{definition:area-simple} below) - we achieve this by constructing foliations of unit tangent bundles of Riemannian balls by families of (lifts of) minimal discs (Theorem \ref{th.Foliated_Plateau_posta}). We use these foliations to construct surface Radon transforms, and we show that the normal operators of these Radon transforms are of Fredholm type (Theorem \ref{theorem:intro-normal}). Local rigidity of the boundary area spectrum (Theorem \ref{theorem:main}) then follows by injectivity of the Radon transforms, which we establish for an open and dense set of area simple metrics containing all real analytic metrics. We also show global rigidity when the conformal variation is analytic (Theorem \ref{theorem:global}).

In fact, we study foliated Plateau problems for surfaces satisfying far more general elliptic curvature conditions. Such surfaces, which we will call $(\Phi,k)$-surfaces, are described in \S\ref{ssection:foliated-plateau} and \S\ref{ssection:phiksurfaces}. This framework includes minimal surfaces, constant mean curvature surfaces, and constant extrinsic curvature surfaces as special cases. We study the foliations and Radon transforms that such Plateau problems define, and we establish general conditions under which the normal operators of these Radon transforms are of Fredholm type.

Finally, we emphasize that all our results are restricted to the $3$-dimensional case, in contrast to \cite{Busch-Liimatainen-Salo-Tzou-25}, for example, which holds for any dimension $\geq 3$. The extension of our results to higher dimensions remains an open problem, and we discuss possible approaches and obstructions in \S\ref{sssection:dimensions}.

\subsection{Main rigidity results}  \label{ssection:area}

Throughout this paper, $X$ will be a \emph{real analytic} $3$-manifold with boundary, diffeomorphic to the closed unit ball $\closedthreeball$, and $\bar g$ will denote a \emph{smooth} Riemannian metric over $X$.

\subsubsection{The boundary area spectrum and the space of circles}

\label{sssection:intro-circles}

For any oriented simple closed curve $c$ in $\partial X$, we say that a properly embedded minimal surface with boundary $S\subseteq X$ \emph{solves} the \emph{Plateau problem} with boundary $c$ whenever $\partial S=c$ and the orientation of $\partial S$ coincides with that of $c$ (we keep track of orientation in order to address more general problems later). Given such a curve $c$, we define
\begin{equation}
\mathcal{A}_{\bar g}(c) := \inf_{\partial S = c}\mathrm{Area}_{\bar{g}}(S)\ ,
\end{equation}
where the infimum is taken over all properly embedded oriented minimal surfaces $S$ with boundary $c$. In the cases of interest to us here, the minimizer will always be a disc.

We call the functional $\mathcal{A}_{\bar g}$ the \emph{boundary area spectrum} of $\bar g$, and we will be interested in its restrictions to $3$-dimensional families of curves defined as follows. Let $\mathbb{S}^2$ denote the unit sphere in $\mathbb{R}^3$, and let $\alpha : \partial X \to \mathbb{S}^2$ be a smooth diffeomorphism. We define an (oriented) \emph{round circle} in $\mathbb{S}^2$ to be an intersection $c:=P\cap\mathbb{S}^2$, for some oriented affine plane $P\subseteq\mathbb{R}^3$, and we define an (oriented) \emph{round circle} in $\partial X$ to be the pre-image $\alpha^{-1}(c)$ of any such curve. We denote the space of all such oriented circles in $\partial X$ by $\cC_\alpha$, and we note that $\cC_\alpha$ is a $3$-dimensional manifold diffeomorphic to $\mathbb{S}^2\times(-1,1)$ (see \S\ref{ssection:space-of-circles}). We henceforth consider the restriction of $\mathcal{A}_{\bar g}$ to $\cC_\alpha$.

\subsubsection{Area simple metrics}

We study rigidity, with respect to conformal variations of the metric, of the functional $\mathcal{A}_{\bar g}:\cC_\alpha\rightarrow(0,\infty)$. As in the case of length spectral rigidity, this will require some sort of uniqueness and non-degeneracy conditions for solutions of the Plateau problem. For this reason, we say that an embedded minimal surface $S$ is \emph{strictly stable} whenever all Dirichlet eigenvalues of its Jacobi operator are strictly positive (see \S\ref{sssection:phi-surfaces} and \S\ref{ss:Stability} for further details).

\begin{defi}[Area simple metrics] \label{definition:area-simple} We will say that $\bar g$ is \emph{area simple} (with respect to $\cC_\alpha$) whenever $\partial X$ is mean convex, and every curve in $\cC_\alpha$ bounds a \emph{unique} properly embedded minimal disc $\mathrm{D}_{\bar g}(c)$ in $X$ which, furthermore, is \emph{strictly stable}.
\end{defi}

This is an open condition on the metric $\bar g$ with respect to the $C^N$ topology for $N \gg 1$ sufficiently large. We will see presently that, when this condition holds, the disc $\mathrm{D}_{\bar g}(c)$ is in fact the unique minimal surface, of any topological type, bounded by $c$ and, in particular, $\mathrm{D}_{\bar g}(c)$ minimizes area amongst all surfaces bounded by $c$ (Lemma \ref{lemma:area-minimizer}). This concept of area simplicity extends to higher dimensions the notion of \emph{simple} metrics used in the study of boundary distance rigidity. We note that area simplicity depends not only on the metric but also on the family $\cC_\alpha$ of curves.

\begin{example}
All metrics sufficiently close to the Euclidean metric on the unit ball are area simple (for the family of round circles on $\Ss^2$). In our companion paper \cite{Alvarez-Lefeuvre-Lowe-Smith}, we describe an explicit open set of metrics for which area simplicity can be verified. Indeed, for every Riemannian metric $\bar g$ over $X$, we define
\begin{equation}\label{eq.Kdegbar}
\opK(\bar g):=\sup_{c\in\cC_\alpha}\int_c |\kappa_{c}|\dd c\ ,
\end{equation}
where, for every curve $c$, $\kappa_c$ denotes its geodesic curvature with respect to $\bar g$, and $\dd c$ denotes its arc-length measure. Note that this quantity is invariant under rescaling of $\bar g$ by a constant factor. We define $\alpha_\star := (4\sqrt{17}-8)/13 \simeq 0.65...$, and, for $\lambda>0$, we consider the conditions: \textbf{\hypertarget{AA1}{(A1)}} $\partial X$ is mean convex; \textbf{\hypertarget{AA2}{(A2)}} $\opK(\bar g)<4\pi$; \textbf{\hypertarget{AA3}{(A3)}} $-(1+\alpha_\star) \lambda^2 <\sect_{\bar g}<-\lambda^2$; and \textbf{\hypertarget{AA4}{(A4)}} $\|\nabla\Ric_{\bar g}\|< \alpha_\star \lambda^{3}$. In \cite{Alvarez-Lefeuvre-Lowe-Smith}, we show that, under Hypotheses $\hyperlink{AA1}{\rm(A1-A4)}$, $\bar g$ is area simple with respect to $\cC_\alpha$. The reader may verify that these hypotheses hold, in particular, for the unit ball in $3$-dimensional hyperbolic space, where $\cC_\alpha$ is the canonical space of round circles (see \S\ref{ssection:example} for further details).
\end{example}

\subsubsection{Local rigidity of the boundary area spectrum for $C^{14}$ conformal factors}

The first main result of the present paper is the following:

\begin{theorem}
\label{theorem:main}
There exists $N \gg 1$ such that, for all analytic $\alpha$, there exists a subset $\cU_\alpha$ of the set of area simple metrics, open and dense in the $C^N$ topology, which contains all real analytic area simple metrics, and which has the following property:

For all $\bar g \in \cU_\alpha$, there exists $\eps := \eps(\bar g) >0$ such that, if $f\in C^{14}(X)$ satisfies $\|f\|_{C^{14}(X)}<\eps$ and
\begin{equation*}
\mathcal{A}_{e^{2f}\bar g}(c)=\mathcal{A}_{\bar g}(c)\ , \qquad \forall c \in \cC \ ,
\end{equation*}
then $f=0$.

That is, generic area simple metrics $\bar g$ are locally rigid, with respect to $C^{14}$ conformal perturbations, for the functional $\mathcal{A}_{\bar g}:\cC_\alpha\rightarrow(0,\infty)$.
\end{theorem}

We also obtain global rigidity when $f$ is assumed to be analytic (see Theorem \ref{theorem:global}). We actually prove a stronger result, including a stability estimate for this problem (Theorem \ref{theorem:estimee}). Finally, we expect the closeness assumption between $\bar g$ and $e^{2f}\bar g$ to be unnecessary as soon as both metrics are area simple. This is further supported by Theorem \ref{theorem:global} and is formulated in Conjecture \ref{conjecture:a}.

We emphasize that the family of boundary conditions $\cC_\alpha$ is \emph{$3$-dimensional}. In this sense, the boundary area rigidity problem considered here is \emph{sharp} for microlocal inversion\footnote{The canonical relation of the associated Radon transform is a symplectomorphism.}, much like boundary distance rigidity in dimension $2$, that is, the space $\cC_\alpha$ of oriented round circles has the same dimension as the ambient manifold $X$. This contrasts with higher-dimensional integral geometry problems, which are typically overdetermined. See \S\ref{ssection:champions} for a comparison of our theorem with earlier works.

\subsubsection{Global rigidity for analytic conformal factors} When the conformal factor $f$ is further assumed to be analytic, we can establish global rigidity.

\begin{theorem}
\label{theorem:global}
If $\bar g$ and $e^{2f} \bar g$ are smooth and area simple with respect to $\cC_\alpha$, if $f \in C^\omega(X)$ is real analytic on $X$, and if
\begin{equation*}
\mathcal{A}_{e^{2f}\bar g}(c)=\mathcal{A}_{\bar g}(c)\ , \qquad \forall c \in \cC \ ,
\end{equation*}
then $f=0$.
\end{theorem}

Observe that $\bar g$ need not be analytic in Theorem \ref{theorem:global}. The proof of this theorem is considerably simpler than that of Theorem \ref{theorem:main} and consists in showing that $f$ vanishes to infinite order on the boundary $\partial X$, using small circles. In fact, it will become apparent from the proof that the conclusion follows from an even weaker assumption, namely, it suffices to assume that $\mathcal{A}_{e^{2f}\bar g}$ and $\mathcal{A}_{\bar g}$ coincide to infinite order on the boundary $\partial \overline{\cC_\alpha} = \overline{\cC_\alpha} \setminus \cC_\alpha$ of the (analytic) compactification of $\cC_\alpha$.

\subsection{Foliated Plateau problems} \label{ssection:foliated-plateau}

The results stated above fit into the much more general framework of surface Radon transforms on $(\Phi,k)$-surfaces, which we now introduce. In what follows, $\cC := \cC_\alpha$ denotes a space of circles defined with respect to a smooth diffeomorphism $\alpha : \partial X \to \Ss^2$ as in \S\ref{sssection:intro-circles}. When $\bar g$ is analytic, $\alpha$ is also further assumed to be analytic.

\subsubsection{Plateau problems}

Given a circle $c \subset \partial X$ in the boundary, a \emph{Plateau problem} consists in constructing a geometric disc $\mathrm{D}(c) \subset X$ with prescribed boundary $c$ that satisfies a given curvature equation. Minimal surfaces, for example, correspond to discs of vanishing \emph{mean curvature}\footnote{Half the trace of the Weingarten operator.}, although many other classes of solutions arise (see \S\ref{sssection:k-surfaces} and \S\ref{sssection:phi-surfaces}).

By a \emph{foliated Plateau problem}, we mean a family of such boundary value problems parametrized by a finite-dimensional manifold $\cC$, so that the boundary circle $c$ varies smoothly within $\cC$ (in our setting, $\cC$ is the three-dimensional space of circles). Under some assumptions on the metric $\bar g$, this is expected to provide a natural foliation of the unit tangent bundle $SX$ of $X$ by considering the \emph{Gauss lifts} $\mathrm{L}(c) \subset SX$ of the discs $\mathrm{D}(c)$ defined by $\mathrm{L}(c) = \{(x,n(x)) ~:~ x \in \mathrm{D}(c)\}$, where $n(x)$ is the oriented normal unit vector field to the disc. See \cite{Gromov-91-1,Gromov-91-2}.

The key novelty of the present work is the combination of the analytic framework for Radon transforms satisfying the Bolker condition (see \cite{Guillemin-84, Greenleaf-Uhlmann-89, Greenleaf-Uhlmann-91} among other references on the topic) with the geometric tools introduced in \cite{Alvarez-Lowe-Smith-25-1,Alvarez-Lowe-Smith-25-2} to study solutions to foliated Plateau problems (see also \cite{Alvarez-25} for a survey). In those works, solutions to foliated Plateau problems were used to study the asymptotic behaviour and rigidity of closed surfaces with constant extrinsic curvature in negatively curved three-manifolds, in the spirit of \cite{Calegari-Marques-Neves-22}, who treated the minimal case by different methods.

\subsubsection{More general Plateau problems: $k$-surfaces and CMC surfaces} \label{sssection:k-surfaces}In \S\ref{ssection:area}, we have focused on minimal surfaces. However, there are other Plateau problems that can be treated by the same methods and which exhibit, in some respects, better behaviour. One natural class is that of \emph{$k$-surfaces}, that is, surfaces of prescribed \emph{extrinsic curvature}\footnote{The determinant of the Weingarten/shape operator.} $\kappa_{\mathrm{ext}}=k^2$ where $k>0$ is a constant. They are solutions  to \emph{Monge-Ampère equations} and were extensively studied by Labourie \cite{Labourie-97,Labourie-00,Labourie-05} (see also \cite{Smith-21}). The associated Plateau problem is significantly better behaved than in the minimal case. More precisely, let $\bar{g}$ be a Riemannian metric on $X$ with $\sect_{\bar g}\leq -a^2$, and assume that $\partial X$ satisfies $\kappa_{\mathrm{ext}}>a^2$.  If $k<a$, then every oriented round circle in $\partial X$ bounds a unique $k$-disc with compatible orientation. Moreover, this disc is automatically strictly stable (see \cite{Smith-20}); compare with the result of \cite{Alvarez-Lefeuvre-Lowe-Smith} under the assumptions $\hyperlink{AA1}{\rm(A1-A4)}$. On the other hand, these surfaces do not admit a clear variational characterization. Other types of surfaces defined by curvature equations also exist, such as constant mean curvature (CMC) surfaces; see \cite{Lopez-13} for an introduction, and \S\ref{sss.more_examples} for other examples of curvature functionals.

\subsubsection{Curvature functions and $(\Phi,k)$-surfaces} \label{sssection:phi-surfaces}

Although our main focus is on minimal surfaces, it is convenient to work in a more general and unified setting, which simultaneously encompasses minimal surfaces, constant mean curvature surfaces, and surfaces of prescribed extrinsic curvature. To this end, we adopt a geometric version of the framework of curvature functions introduced by Caffarelli–Nirenberg–Spruck \cite{Caffarelli-Nirenberg-Spruck-88}. It is closely related in spirit to the theory of parametric elliptic functionals developed by White \cite{White-87}, with the important difference that we also allow for non-variational problems (for instance, surfaces of prescribed extrinsic curvature). Let us briefly describe this unified framework.

Let $\jetX \to X$ denote the $2$-jet bundle of oriented immersed surfaces, that is, the set of triples $(x,n,A)$ where $x\in X$, $n\in S_xX$, and $A$ is a self-adjoint linear endomorphism of the plane $n^\perp \subset T_xX$. Any immersion $e:D\to X$ induces a lift $\jete:D\to\jetX$ defined by
$$\jete(p)=(e(p),n_e(p),A_e(p)),$$
where $n_e$ and $A_e$ denote respectively the unit normal oriented vector field to the surface, and the \emph{shape operator} of $e$, given by $A_e(u)=\overline{\nabla}_u n_e$.

A \emph{curvature function} is a pair $(\Phi,\Lambda)$, where $\Lambda\subset\jetX$ is an open subset such that, for each $(x,n)\in SX$, the fibre $\Lambda_{(x,n)} := \{A : (x,n,A)\in\Lambda\}$ is an open cone invariant under conjugation by rotations of $n^\perp$ and under addition of positive definite self-adjoint operators (see \S\ref{ss:CurvatureFunctions} for details). The function $\Phi:\Lambda\to\mathbb{R}$ is $1$-homogeneous, invariant under conjugation by rotations and satisfies $\Phi(\Id)=1$, $\Phi|_{\partial \Lambda}=0$ (see Definition \ref{definition:curvature-functions}). In addition, we say that $(\Phi,\Lambda)$ is \emph{elliptic} if the fibrewise differential of $\Phi$ is positive definite (when properly identified with a symmetric endomorphism, see \eqref{equation:ba}).

\begin{defi}[$\Lambda$-convexity and $(\Phi,k)$-surfaces] \label{definition:phik-surfaces-intro}An immersion $e:D\to X$ is said to be \emph{$\Lambda$-convex} whenever $\jete(D)\subset \Lambda$. Given $k\geq 0$, we say that $e$ is a \emph{$(\Phi,k)$-surface} whenever $e$ is $\Lambda$-convex and
\begin{equation*}
\Phi(\jete(p)) = k \qquad \forall p\in D.
\end{equation*}
\end{defi}

Given a $(\Phi,k)$-surface $e:D\to X$ and $f\in C^\infty(D)$, consider the normal variation
\[
e_{tf}(p) := \exp_{e(p)}(t f(p)\, n_e(p)),
\]
defined for $t$ sufficiently small. The associated $\Phi$-\emph{Jacobi operator} is defined by
\begin{equation}
\label{equation:jacobi}
J^\Phi f := \partial_t\Phi(\jete_{tf})|_{t=0},
\end{equation}
for $f\in C^2(D)$. We will show that $J^\Phi$ is a second-order linear elliptic differential operator (Lemma \ref{lem.Jacobi_elliptic}). We say that the $(\Phi,k)$-surface is \emph{strictly stable} when the $\Phi$-Jacobi operator admits a positive solution $J^\Phi u = 0$, $u > 0$ (see \S\ref{ss:Stability}). This condition should be thought of as the natural higher-dimensional analogue of the absence of conjugate points for geodesics.

\begin{example}[Minimal and $k$-surfaces]Minimal surfaces correspond to the case $\Lambda=\jetX$, $\Phi(A)=\tr(A)/2$, and $k=0$, while $k$-surfaces correspond to $\Lambda=\{A>0\}$, $\Phi(A)=\det(A)^{1/2}$ and $k>0$. \end{example}

\subsubsection{$(\Phi,k)$-simple metrics} The previous paragraph motivates the following notion, which extends the standard notion of a \emph{simple} metric from the geodesic setting to the present one:

\begin{defi}[$(\Phi,k)$-simple metrics]
\label{definition:phik-simplicity}
Let $(X,\bar g)$ be a closed three-dimensional Riemannian ball, and let $\cC$ denote the space of oriented round circles in $\partial X$ (see \S\ref{sssection:intro-circles}). Let $(\Phi,\Lambda)$ be an elliptic curvature function, and let $k>0$.

We say that $\bar g$ is \emph{$(\Phi,k)$-simple} whenever the following conditions hold:
\begin{itemize}
\item $\partial X$ is $\Lambda$-convex and has $\Phi$-curvature strictly greater than $k$;
\item for every $c\in\cC$, there exists a unique oriented $(\Phi,k)$-disc $\opD_{\bar g}(c)$ spanned by $c$;
\item each such disc is strictly stable.
\end{itemize}

In the particular case where $\Phi=\frac{1}{2}\tr, k=0$, we call the metric \emph{area simple} (Definition \ref{definition:area-simple}).
\end{defi}

This is an open condition on the metric $\bar g$ with respect to the $C^N$ topology for $N \gg 1$ sufficiently large. In the minimal case, the boundary condition reduces to the requirement that $\partial X$ be mean convex. The first condition is reminiscent of the strict convexity imposed to the boundary for simple metrics in the geodesic case. By \cite{Alvarez-Lefeuvre-Lowe-Smith}, the metric is area simple whenever it satisfies $\hyperlink{AA1}{\rm(A1-A4)}$. In the case of $k$-surfaces, as mentioned in \S\ref{sssection:k-surfaces}, the following holds: if $\sect_{\bar g}\leq -a^2$, $\partial X$ satisfies $\kappa_{\mathrm{ext}}>a^2$, and $k<a$, then the metric is $(\Phi,k)$-simple, see \cite{Smith-21}.

\subsubsection{The foliated Plateau problem}

We now state the main geometric result underlying our approach, which provides a general solution to the foliated Plateau problem. In the following statement, $SX^\bullet$ is the open subset of the unit tangent bundle $SX$ consisting of unit tangent vectors that are not orthogonal to $\partial X$. By $Y \Subset Y'$, we mean that there exists an open set $U$ such that $Y \subset U \subset Y'$, and refer the reader to Appendix \ref{section:extension} for the notion of proper extensions.

\begin{theorem}[Foliated Plateau problem]\label{th.Foliated_Plateau_posta}
Let $(X,\bar g)$ be a closed three-dimensional Riemannian ball, and let $\cC$ denote the space of oriented round circles in $\partial X$ (see \S\ref{sssection:intro-circles}). Let $(\Phi,\Lambda)$ be an elliptic curvature function, and let $k>0$. Assume that $\bar g$ is $(\Phi,k)$-simple in the sense of Definition \ref{definition:phik-simplicity}. Then:

\begin{itemize}

\item[(i)] There exists a smooth disc-bundle $\pi_R:SX^\bullet\to\cC$ whose fibres are precisely the Gauss lifts of the $(\Phi,k)$-discs $\opD_{\bar g}(c)$, and which moreover satisfies the \emph{Bolker condition}. If, in addition, $\bar g$ is real analytic, then the bundle $\pi_R$ is real analytic.

\item[(ii)] There exists a smooth (resp. analytic) proper Riemannian extension $X_e$ of $X \Subset X_e$, a smooth (resp. analytic) proper extension $\cC_e$ of $\cC \Subset \cC_e$, and an open subset $U \subset SX_e$ satisfying $SX \subset U$, such that $\pi_R : U \to \cC_e$ is a smooth (resp. analytic) disc-bundle extension of the map $\pi_R : SX^\bullet \to \cC$. Its fibres are the Gauss lifts of extensions of the $(\Phi,k)$-discs $\opD_{\bar g}(c)$ to $X_e$, and $\pi_R$ satisfies the Bolker condition.

\item[(iii)] In particular, there exists a compactification $\overline{\cC} \Subset \cC_e$ of $\cC$ such that $\pi_R : SX \to \overline{\cC}$ extends smoothly (resp. analytically). The fibre above a boundary point $c \in \partial \overline{\cC} = \overline{\cC} \setminus \cC$ reduces to the singleton $\{(x,\pm\nu(x))\}$, where $x \in \partial X$ is the only point such that $\mathrm{D}_{\bar g}(c)=\{x\}$ and $\nu$ denotes the outward-pointing unit normal vector field along $\partial X$.
\end{itemize}

In the case where $\Phi=\frac{1}{2}\tr$, the same result also holds for $k=0$.
\end{theorem}

Theorem \ref{th.Foliated_Plateau_posta} is one of the main new technical contributions of the present article and occupies \S\ref{section:geometry}--\ref{section:plateau}. The \emph{Bolker condition} is a geometric nondegeneracy property that we refrain from introducing here; we refer to Theorem \ref{theorem:bolker} below and the surrounding discussions for details. See also \cite{Mazzuchelli-Salo-Tzou-23} where this condition was systematically studied. We emphasize that establishing the Bolker condition in the present setting is far from straightforward and requires a careful analysis of solutions to the $\Phi$-Jacobi equation, $J^\Phi u = 0$ (see \S\ref{section:plateau}).

\subsubsection{Double fibration} The main consequence of Theorem \ref{th.Foliated_Plateau_posta} is the existence of a smooth (resp. real analytic) \emph{double fibration} (in the sense of Guillemin \cite{Guillemin-84})
\begin{equation}
\label{equation:diag_intro}
\begin{tikzcd}
& SX^\bullet \arrow[rd, "\pi_R"] \arrow[ld,"\pi"'] & \\
X & & \cC
\end{tikzcd}
\end{equation}
where $\pi:SX\to X$ is the basepoint projection, and $\pi_R:SX^\bullet\to\cC$ is the projection of $(x,v)$ to the unique circle $c$ such that the $(\Phi,k)$-surface $\mathrm{D}(c)$ (with boundary $c$) passes through $x$ with unit normal vector $v$. This shows that the \emph{Gauss lifts of $(\Phi,k)$-surfaces spanned by oriented round circles $c\in\cC$ form a smooth (resp. analytic) foliation $\cF$ of $SX^\bullet$}. Furthermore, the map $\pi_R:SX^\bullet\to\cC$ defines a smooth (resp. real-analytic) disc-bundle over $\cC$. This foliation may be viewed as a $2$-dimensional analogue of the geodesic flow, while the Bolker condition plays a role analogous to the absence of conjugate points for geodesics.

\subsubsection{Extending the foliation} As stated in Theorem \ref{th.Foliated_Plateau_posta}, item (ii), an important feature of our work is that we can show that the double fibration \eqref{equation:diag_intro} extends smoothly (resp. analytically) to a larger domain $SX_e$ containing $SX$. In the geodesic setting, this is immediate, since geodesics in $X$ are simply the restrictions of geodesics in $X_e$. By contrast, such an extension is far from obvious for a foliation by $(\Phi,k)$-surfaces. Establishing it is one of the main purposes of the asymptotic analysis developed in \S\ref{section:asymptotic}. See Theorem \ref{theorem:extension-foliation} for a complete statement on the extension.

In particular, $\pi_R : SX \to \bar \cC$ extends smoothly (resp. analytically) but the fibre $\pi_R^{-1}(\{c\})$ over a boundary point $c \in \bar \cC \setminus \cC$ reduces to a single point. This phenomenon already occurs on simple manifolds: the space of geodesics is naturally identified with $\partial_-SX$, the set of inward-pointing unit vectors based at a boundary point. Each vector $v$ in the interior of $\partial_-SX$ (i.e., one that is not tangent to $\partial X$) generates a geodesic segment in $X$ whereas for vectors in the boundary of $\partial_-SX$ (i.e., tangent vectors to the boundary $\partial X$), the geodesic reduces to a single point.

\subsection{Analysis of the Radon transform} \label{ssection:radon}

We now assume that the hypotheses of Theorem \ref{th.Foliated_Plateau_posta} are satisfied, that is, the metric $\bar g$ is $(\Phi,k)$-simple.

\subsubsection{Surface Radon transform} The \emph{Radon transform} (of functions) is the operator
\begin{equation}
\label{equation:radon-intro}
R_0 : C(X) \to C(\cC), \qquad R_0f(c) := \int_{\mathrm{D}(c)} f~\dd\mathrm{D}(c),
\end{equation}
where $C(\bullet)$ is the space of continuous functions, and $\dd\mathrm{D}(c)$ denotes the area form on $\mathrm{D}(c)$ induced by the ambient metric $\bar g$. In the case of a foliation by geodesics, this operator is called the \emph{X-ray transform} and was extensively studied in the literature, see \cite{Mukhometov-77, Sharafutdinov-94, Stefanov-Uhlmann-04, Stefanov-Uhlmann-05, Paternain-Salo-Uhlmann-13, Uhlmann-Vasy-16,Paternain-Salo-Uhlmann-23} among other references. More generally, a natural Radon transform $R_2$ on $2$-tensors can be defined, and corresponds to the linearization of the area functional with respect to the metric, see \S\ref{sssection:radon-tensors}.

One of the main results of the present paper is the following:

\begin{theorem}[Kernel/Injectivity of the surface Radon transform]
\label{theorem:quasi-injective}
Assume that $\bar g$ is $(\Phi,k)$-simple. Then:
\begin{enumerate}
\item[(i)] $R_0 : C(X) \to C(\cC)$ has finite-dimensional kernel;
\item[(ii)] If $f \in \ker_{C(X)}R_0$, then $f \in C^\infty(X)$ and $f$ vanishes to infinite order on $\partial X$;
\item[(iii)] There exists $N \gg 1$ such that, for all analytic $\alpha$, there exists a subset $\cU_\alpha$ of the set of area simple metrics, open and dense in the $C^N$ topology, which contains all real analytic area simple metrics, and which has the following property: for all $\bar g \in \cU$, the corresponding Radon transform $R_0 : C(X) \to C(\cC)$ is injective.
\end{enumerate}
\end{theorem}

The set $\cU_\alpha$ in Theorem \ref{theorem:quasi-injective} is the same as the one appearing in Theorem \ref{theorem:main}.

\subsubsection{Normal operator} The proof of Theorem \ref{theorem:quasi-injective} is based on the following key property: we show that the associated \emph{normal operator} $\Pi_0 := R_0^*R_0$ is an \emph{elliptic pseudodifferential operator of order $-2$} in the interior $X^\circ$. The adjoint $R_0^*$ is computed with respect to an arbitrary smooth measure on $\cC$; the pseudodifferential elliptic behavior of the operator is independent of this choice. We refer to Appendix \ref{appendix:pseudo} for a brief reminder on the theory of pseudodifferential operators.

However, this fact alone is not sufficient to prove Theorem \ref{theorem:quasi-injective}. Namely, one needs to understand the boundary behaviour of solutions of $R_0 f = 0$, and show that they are \emph{uniformly smooth} (resp. analytic) up to the boundary. This is where the asymptotic analysis of small discs developed in \S\ref{section:asymptotic} enters in a crucial way: we show that small discs $c \to \partial \cC$ can actually be obtained as restriction of larger discs on a larger domain, which are uniformly smooth as $c$ shrinks to a point. In other words, the normal operator $\Pi_0$ is the restriction to $X$ of a pseudodifferential operator $P$ defined in a neighborhood $X_e$ of $X$.

This is the content of the following result:

\begin{theorem}[Pseudodifferential behaviour of the normal operator]
\label{theorem:intro-normal}
Assume that $\bar g$ is $(\Phi,k)$-simple. Then:
\begin{enumerate}
\item[(i)] The normal operator $\Pi_0 := R_0^*R_0$ is an elliptic pseudodifferential operator on $X^\circ$;
\item[(ii)] There exists an extension $X_e$ of $X \Subset X_e$ and a pseudodifferential operator $P$ on $X_e^\circ$, elliptic on $X$, such that $R_0 f = 0$ implies $P E_0 f = 0$, where $E_0 : L^2(X) \to L^2(X_e)$ is the extension operator by $0$ outside of $X$.
\item[(iii)] When $\bar g$ is analytic, both $\Pi_0$ and $P$ are analytic pseudodifferential operators.
\end{enumerate}
\end{theorem}

Theorem \ref{theorem:main} then follows from standard arguments using Theorem \ref{theorem:intro-normal}. The proof of the analyticity of these operators requires a careful analysis of the Schwartz kernel of $\Pi_0$ on the blown-up space $[X \times X; \Delta]$, where $\Delta \subset X \times X$ denotes the diagonal.

\subsection{Strategy of proof}

The main difficulty lies in establishing Theorem \ref{th.Foliated_Plateau_posta}, and in particular in proving that the Gauss lifts of the discs define a foliation of the unit tangent bundle, that the Bolker condition is satisfied, and that the foliation extends smoothly across the boundary $\partial X$ to a slightly larger manifold. This is the focus of \S\ref{section:geometry}--\ref{section:plateau}. Once these geometric properties are established, it follows from now-standard microlocal arguments, developed in \S\ref{s.Radon}, that the normal operator associated with the corresponding Radon transform is an elliptic pseudodifferential operator in the smooth setting.

However, establishing analyticity when the metric is analytic is considerably more delicate, requiring an explicit computation of the Schwartz kernel near the diagonal on the blown-up space $[X \times X ; \Delta]$, where $\Delta \subset X \times X$ denotes the diagonal. Finally, the application to boundary area rigidity is developed in \S\ref{section:applications}, and follows directly from the ellipticity of the normal operator together with its analyticity in the analytic category.

\subsubsection{Global strategy for Theorem \ref{th.Foliated_Plateau_posta}}

The proof of Theorem \ref{th.Foliated_Plateau_posta} proceeds in three steps:
\begin{itemize}
\item First, we show that the space of \emph{marked discs}
\begin{equation*}
\mathrm{MD} := \{ (x,c) : c \in \cC,\ x \in \mathrm{D}(c)\} \subset X \times \cC
\end{equation*}
is a smooth submanifold.
\item Second, we prove that the \emph{Gauss map}
\begin{equation*}
N : \mathrm{MD} \to SX^\bullet, \qquad SX^\bullet := SX \setminus \{(x,\pm\nu(x)) ~:~ x \in \partial X\},
\end{equation*}
where $\nu$ denotes the outward-pointing unit normal vector field along $\partial X$, defined by $N(x,c)=(x,v)$ with $v$ the oriented unit normal vector to the disc $\mathrm{D}(c)$ at $x$, is a smooth diffeomorphism.
\item Finally, we prove that the small discs $\mathrm{D}(c)$ (obtained as $c$ shrinks to a point in the space of circles $\cC$) are uniformly smooth (resp. analytic) and therefore extend to a slightly larger neighborhood of $X$.
\end{itemize}

\subsubsection{Step 2. Analyzing solutions of the Jacobi equation}

While the first step is a mere application of the implicit function theorem, the second step relies on a careful analysis of solutions to the Jacobi equation $J^\Phi f=0$ on the disc $\mathrm{D}(c_0)$ with prescribed boundary data $\hat{f}\in C^\infty(c_0)$, for all circles $c_0 \in \cC$. Using a normal graph parametrization of nearby circles to $c_0 \subset \partial X$, it can be shown that the tangent space $T_{c_0}\cC$ identifies with the set of functions
\begin{equation*}
\cT_{c_0} = \{\SS^1 \ni \theta \mapsto \gamma-\alpha\cos \theta-\beta\sin\theta ~:~ \alpha,\beta,\gamma \in \R\},
\end{equation*}
where $\SS^1 = \R/2\pi\Z$ and $\theta$ is a variable parametrizing the circle $c_0$ at constant speed with respect to the standard metric (see \S\ref{sss.First_order_variations} for details). A key point is to understand how the zero set of the Jacobi field $f$ depends on the zero set of the boundary datum $\hat{f} \in \cT$. More precisely, for $\hat{f} \in \cT$, we consider $Z_f := \{f=0\} \subset \mathrm{D}(c_0)$, and $\hat{Z}_f \subset S\mathrm{D}(c_0)$, the lift of $Z_f$ to the unit tangent bundle $S\mathrm{D}(c_0)$ by the normal vector to the curve $\{f=0\}$. Letting $P\cT$ denote the projectivization of $\cT$, we show the following key statement: the set $(\hat{Z}_f)_{f \in P\cT}$ covers the unit tangent bundle $S\mathrm{D}(c_0)$ by curves (Theorem \ref{thm:coveringofunitbundle}). This analysis is carried out in \S\ref{ss:TwoDSetting} using techniques from the theory of nodal sets of second-order elliptic operators, together with a graph-theoretic argument; the two-dimensionality of the disc plays an essential role at this step of the argument.

Going back to the three-dimensional ball $(X,\bar g)$, this analysis easily translates to show that the Gauss map is a local diffeomorphism (Theorem \ref{thm:StructureOfdiscSpace}). To prove that it is globally a diffeomorphism onto $SX^\bullet$, it therefore remains only to establish that it is proper. This is where the asymptotic analysis of small discs developed in \S\ref{s.asymptotic_Analysis} enters. We show that, as the boundary circle shrinks to a point, the Gauss lifts of the corresponding minimal discs converge in the Hausdorff topology to the image of the outward-pointing unit normal field $\nu$ along $\partial X$. This asymptotic behavior yields the required properness.

\subsubsection{Step 3. Main difficulty: the geometry of small surfaces} \label{ssection:difficulty} Compared with the geodesic Radon (X-ray) transform on manifolds with strictly convex boundary, the present problem involves a fundamentally new difficulty: understanding the geometry of the family of $(\Phi,k)$-surfaces as the boundary circles they span collapse to a point. The underlying obstruction is that, unlike short geodesics in a manifold with strictly convex boundary, it is not immediate that minimal surfaces can, in general, be realized as restrictions of minimal surfaces in a larger ambient manifold.

In \S\ref{s.asymptotic_Analysis}, we address this issue through what we call the \emph{asymptotic analysis of $(\Phi,k)$-surfaces}. We show that sufficiently small surfaces can be described as graphs over their boundary discs, given by perturbations of a certain paraboloid. This is the content of Theorem \ref{thm:SmallCircles}. \emph{A priori}, without further analysis, as the radius $r$ of the boundary circles tends to zero, the limiting regularity that one obtains is only $C^{1,1}$ (the corresponding graph functions remain uniformly bounded in $C^{1,1}$ as the radius $r$ tends to zero). However, we show that, under an appropriate change of variables — that is, by introducing the right coordinate $s := r^2$ on $\cC$, and compactifying $\cC$ with respect to the boundary defining function $s$ by $\bar \cC := \cC \sqcup \{s=0\}$ — the graphs actually remain uniformly bounded in $C^\infty$, and can thus be extended to a larger ambient space. In the analytic category, we show that they remain uniformly analytic. These arguments require a careful study of the Taylor expansion about $r=0$ of the graphs corresponding to circles of radius $r$ and constitute the purpose of \S\ref{ssection:smooth-coordinates}.

\subsection{Comparison with earlier works} \label{ssection:champions}
\subsubsection{Boundary area problem}

The works most closely related to the present article are \cite{Alexakis-Balehowsky-Nachman-20, Busch-Liimatainen-Salo-Tzou-25} but, as indicated above, they differ from ours in two major ways:

\begin{itemize} \item First, our boundary data are given by the space of circles $\cC$ on $\partial X$, which is only $3$-dimensional, whereas \cite{Alexakis-Balehowsky-Nachman-20} considers infinite-dimensional families, and in \cite{Carstea-Liimatainen-Tzou-24} and \cite{Busch-Liimatainen-Salo-Tzou-25}, the authors reduce to finite-dimensional families of non-optimal dimension (at least $6$). Our space $\cC$ is sharp in the sense that one cannot expect stable injectivity of the corresponding Radon transform from a smaller family of boundary data.

\item Second, Theorem \ref{theorem:main} (resp. Theorem \ref{theorem:global}) makes no vanishing assumption for the jets of the function $f$ at the boundary $\partial X$: any conformal perturbation is allowed, provided it is small in the $C^{14}$ topology (resp. it is analytic).

\end{itemize}

\subsubsection{The marked area spectrum for closed $3$-dimensional manifolds}

The approach developed in the present paper builds on Labourie's insight that $k$-surfaces (and minimal surfaces) may be viewed as a natural higher-dimensional analogue of the geodesic flow in dimension $3$ \cite{Labourie-97,Labourie-00,Labourie-02}. This perspective is closely tied to the work of Kahn--Markovic, who proved that the fundamental group of a closed $3$-dimensional hyperbolic manifold contains many surface subgroups \cite{Kahn-Markovic-12-1}. The surface subgroups they construct are \emph{quasi-Fuchsian} (their limit set in the ideal boundary is a Jordan curve) and abundant (see the topological counting in \cite{Kahn-Markovic-12-2}). Each such conjugacy class admits a unique $k$-surface representative, obtained by solving an asymptotic Plateau problem \cite{Labourie-97,Rosenberg-Spruck-94,Smith-21}. It also admits a minimal surface representative \cite{Anderson-83}, which is unique \emph{provided the limit circle is close enough to being circular} \cite{Uhlenbeck-83}; otherwise, there may exist uncountably many minimal surfaces representing the same conjugacy class \cite{Huang-Lowe-Seppi-26}. This phenomenon is reminiscent of the geometric conditions $\hyperlink{AA1}{\rm(A1-A4)}$ that we impose in the present paper to guarantee uniqueness and stability of solutions to the Plateau problem \cite{Alvarez-Lefeuvre-Lowe-Smith}: in both settings, some form of geometric closeness to a rigid model is what makes the relevant Plateau problem well-posed.

Building on this correspondence, one can introduce a higher-dimensional analogue of the \emph{marked length spectrum}: given the conjugacy class of a quasi-Fuchsian subgroup, one associates the area of its unique $k$-surface representative (or, when it is unique, that of its minimal surface representative). This object was introduced in \cite{Alvarez-Lowe-Smith-25-1,Alvarez-Lowe-Smith-25-2} as the $k$-surface analogue of the construction in \cite{Calegari-Marques-Neves-22}, where it was shown that the marked area spectrum \emph{of the hyperbolic metric} is rigid among metrics with sectional curvature bounded either above or below by $-1$.

Our result addresses a rather different rigidity problem. All currently known rigidity results for the marked area spectrum compare a metric with the \emph{extremal} hyperbolic metric $\bar g_0$: they show that if the marked area spectrum of $\bar g$ agrees with that of $\bar g_0$ and $\sect_{\bar g}\leq -1$ (or $\sect_{\bar g}\geq -1$) everywhere, then $\bar g$ is isometric to $\bar g_0$. The extremality of the hyperbolic metric is a fundamental feature of these results and ultimately allows one to exploit Ratner's theory.

By contrast, Theorem~\ref{theorem:main} compares an arbitrary analytic and area simple metric $\bar g$---not necessarily close to any extremal model---with its $C^{14}$ conformal deformations $e^{2f}\bar g$. The analytic method that we use yields a local rigidity statement, valid only for sufficiently small $f \in C^{14}(X)$. When $f$ is analytic, no smallness assumption is required. The proof therefore follows a completely different strategy, combining the geometric framework of foliations by $(\Phi,k)$-surfaces with tools from microlocal analysis, namely Radon transforms and pseudodifferential operators.

We finally note that all of the results discussed above are, for the same reason, inherently three-dimensional: they rely on the resolution of a foliated Plateau problem---the existence of a foliation of the unit tangent bundle by Gauss lifts of $k$-surfaces or minimal surfaces---which is currently only available in dimension $3$ (see \cite{Alvarez-Lowe-Smith-25-1,Labourie-21,Lowe-21}). Extending this foliation result to higher dimensions, in either setting, remains open.

\subsubsection{Zoll metrics on the round sphere}

Our approach also shares similarities with \cite{Ambrozio-Marques-Neves-25, Ambrozio-Guajardo-26}, where smooth Riemannian metrics on the sphere are constructed that admit smooth Zoll families of minimal hypersurfaces. In that setting, the authors study an integration operator over minimal hypersurfaces, whose associated normal operator is shown to be a pseudodifferential operator (see \cite[Section~6]{Ambrozio-Marques-Neves-25}). This is reminiscent of Guillemin's construction of Zoll metrics on the $2$-sphere using the Radon transform along geodesics \cite{Guillemin-76}. The construction of \cite{Ambrozio-Marques-Neves-25, Ambrozio-Guajardo-26} is analogous to our construction, where the normal operator arising from our Radon transform is likewise a pseudodifferential operator (see \S\ref{s.Radon}).

\subsection{Perspectives} This is a non-exhaustive list of questions arising from the present work, that we intend to investigate in the near future.

\subsubsection{Higher dimensions} \label{sssection:dimensions}

A natural direction, which appears substantially challenging, is to replace minimal surfaces by minimal hypersurfaces in higher dimensions. In this case, one considers the \emph{space of hyperspheres} $\cC$ on $\partial X$, again modelled on the intersections of affine hyperplanes in $\R^n$ with $\SS^{n-1}$. For each hypersphere $c\in\cC$, one seeks the unique stable minimal hypersurface spanning $c$. Under suitable assumptions ensuring existence, uniqueness, and stability, one could hope that these hypersurfaces foliate $SX$, just as we prove in the present paper in the three-dimensional case.

The principal obstacle is to extend the analysis of \S\ref{ss:TwoDSetting} to higher dimensions, where the disc is replaced by the $(n-1)$-dimensional ball. In particular, the graph-theoretic argument developed in \S\ref{sssection:geometry-zero-set} to analyse the zero set of solutions of the Jacobi equation $J^\Phi u=0$ would need to be replaced by a higher-dimensional analogue. At present, such a generalization appears far from straightforward, and we leave these questions for future investigation.

\subsubsection{Injectivity of the Radon transform}  \label{ssection:boundary} We expect the Radon transform to be always injective when the metric is $(\Phi,k)$-simple. However, we are only able to prove it on an open dense subset of smooth metrics containing every analytic metric so far. We formulate this as a conjecture.

\begin{conjecture}
\label{conj:r0}
Assume that $\bar g$ is smooth and $(\Phi,k)$-simple. Then $R_0 : C(X) \to C(\cC)$ is injective.
\end{conjecture}

Let us emphasize that, in \cite{Uhlmann-Vasy-16,Stefanov-Uhlmann-Vasy-21}, Stefanov, Uhlmann and Vasy developed a framework for the local study of the X-ray transform of symmetric tensors near the boundary in dimension $\geq 3$. This approach yields local injectivity of the X-ray transform and has led to boundary rigidity results in several settings. However, for dimensional reasons, their method breaks down in dimension $2$. In our setting (Radon transform on minimal surfaces in dimension 3), this local strategy is also \emph{inapplicable} for the very same reason. Indeed, recovering microlocal singularities that are almost tangent to the boundary requires the use of families of “large” discs, i.e. minimal surfaces that are not confined to a small neighborhood of the boundary.

\subsubsection{Conformal case and Santaló's formula} The boundary distance problem in the same conformal class was solved on simple manifolds by Muhometov \cite{Mukhometov-81} (see also \cite{Croke-91} for a more modern proof) without any closeness assumption on the two metrics. The proof relies on a key integral geometry identity, known as \emph{Santaló's formula} \cite{Santalo-52}, combined with the Hölder inequality. Similarly, we expect that Theorem \ref{theorem:main} should hold without any closeness assumption on $\bar g$ and $e^{2f} \bar g$, provided they both satisfy the assumptions of Theorem \ref{th.Foliated_Plateau_posta}.

\begin{conjecture}
\label{conjecture:a}
Assume that $\bar g$ and $e^{2f} \bar g$ are area simple. Then, $\cA_{\bar g} = \cA_{e^{2f} \bar g}$ implies $f \equiv 0$.
\end{conjecture}

The main missing ingredient is an analogue of the Santaló formula in our setting, and such an identity is unlikely to exist in general. The conceptual mechanism behind Santaló's formula is that the geodesic flow on the unit tangent bundle of a Riemannian manifold is a volume-preserving $\R$-action all of whose orbits intersect transversally the boundary, up to a set of measure zero. As a consequence, the Liouville measure can be disintegrated as the product of a boundary measure and the Lebesgue measure in the trivial $\R$-factor.

In our situation, this volume-preserving group action fixing each leaf (in the unit tangent bundle) is absent. One should only expect such a structure to exist in very rigid settings, for instance for the foliation by hyperbolic discs in $\HH^3$, where a well-defined leafwise volume-preserving $\mathrm{PSL}(2,\R)$-action is present. See \S\ref{ssection:example} for instance. This makes the area rigidity problem in the same conformal class substantially more difficult.

\subsubsection{Non-conformal case} \label{ssection:non-conformal} The boundary area rigidity problem for non-conformal metrics is even harder. The main difficulty is to understand the natural kernel of the associated Radon transform $R_2$ on symmetric $2$-tensors and show that $R_2$ is injective on a natural complement. See \eqref{equation:s2} for a definition of the Radon transform $R_2$ on $2$-tensors. By gauge-invariance, observe that this kernel must contain potential tensors but is likely to contain other tensors as well. It was recently characterized in Euclidean domains in \cite{Grebnev-Stefanov-Uhlmann-Zhou-26}. However, the case of general Riemannian manifolds is still open for now, and we intend to study it in a follow-up paper.

\subsubsection{Open manifolds with topology}

In a breakthrough paper \cite{Guillarmou-17-2}, Guillarmou proved that on strictly convex manifolds without conjugate points and with a hyperbolic \emph{trapped set} (the set of unit tangent vectors whose geodesics remain trapped in $X$ for all time), one can define a normal operator associated with the geodesic Radon transform and show that it is elliptic. This result has led to a number of important applications to (marked) boundary rigidity and the lens rigidity problem, see \cite{Lefeuvre-19-1,Lefeuvre-19-2,Cekic-Guillarmou-Lefeuvre-24} for instance.

We expect that the program initiated in the present paper can be extended in a similar direction to manifolds with nontrivial topology and a nonempty trapped set. Under suitable assumptions on the metric, such as $\hyperlink{AA1}{\rm(A1)-(A4)}$ in the minimal case, or negative sectional curvature in the case of $k$-surfaces, we conjecture that for every circle $c\in\cC$ and every relative homotopy class of discs with boundary $c$, there exists a unique $(\Phi,k)$-disc in that class. The resulting family of discs should then induce a foliation of $SX$, analogous to that of Theorem~\ref{th.Foliated_Plateau_posta}, and one should then be able to study the corresponding Radon transform.

\subsubsection{Closed manifolds}
Likewise, the case of \emph{closed manifolds} is of considerable interest. It was shown in \cite{Alvarez-Lowe-Smith-25-1,Alvarez-Lowe-Smith-25-2} that every closed negatively curved $3$-manifold admits a foliation of $SX$ by $k$-surfaces. While most of these surfaces have infinite area, some have finite area and correspond to incompressible surface subgroups $\pi_1(\Sigma_g)\hookrightarrow \pi_1(X)$, whose existence was established in \cite{Kahn-Markovic-12-1}. This is analogous to the geodesic flow, where most orbits have infinite length while some are periodic. One may therefore define a Radon transform by integrating over the closed finite-area leaves and ask whether this transform is injective. This is the analogue of the geodesic X-ray transform along closed geodesics, which is well understood for closed negatively curved manifolds (see, for instance, \cite[Part IV]{Lefeuvre-book} for a survey).

In \cite{Guillarmou-17-1}, Guillarmou associated a natural normal operator to the geodesic X-ray transform over closed geodesics, providing the closed-manifold analogue of the classical normal operator for manifolds with boundary (see, for example, \cite{Stefanov-Uhlmann-05}). We expect that an analogous normal operator should also exist in the setting of $k$-surfaces on negatively-curved $3$-manifolds. The analysis in \cite{Guillarmou-17-1} relies on the deep work of \cite{Faure-Sjostrand-11,Dyatlov-Zworski-16}, which constructs anisotropic Sobolev spaces on which the geodesic flow has discrete spectrum. Extending these techniques to the present setting would require an analogue of those results, whose existence is far from clear. We leave this problem for future work.

\subsection{Organization of the paper}

The paper is organized as follows:

\begin{itemize}
\item In \S\ref{section:geometry}, we review the geometric tools used throughout the paper. In \S\ref{ssection:setting}, basic geometric features (unit tangent bundle, immersed surfaces, etc.) are discussed, while the space of circles is introduced in \S\ref{ssection:space-of-circles}. The specific example of balls in the hyperbolic $3$-space is discussed in \S\ref{ssection:example}.
\item In \S\ref{section:phik}, the notion of $(\Phi,k)$-surfaces is introduced and its main properties are explored. In \S\ref{appendix:asymptotic}, we recall (and reprove) preliminary results on the geometry of graphs. Elliptic curvature functions are then defined in \S\ref{ss:CurvatureFunctions}, before introducing $(\Phi,k)$-surfaces in \S\ref{ssection:phiksurfaces}. The concept of strict stability (the higher dimensional analogue of the absence of conjugate points) is discussed in \S\ref{ss:Stability}. Plateau problems are explained in \S\ref{ssection:plateau-problems}.

\item In \S\ref{section:asymptotic}, we study the asymptotics of small $(\Phi,k)$-discs, which bound small circles $c \in \cC$ as $c$ shrinks to a point. In \S\ref{s.asymptotic_Analysis}, we prove that the discs can be realized as graphs over a boundary disc (Theorem \ref{thm:SmallCircles}); it is then established in \S\ref{ssection:smooth-coordinates} that these graphs are uniformly smooth (resp. analytic) as the circle collapses to a point. An important consequence of \S\ref{s.asymptotic_Analysis} is the fact that the Gauss lifts of small $(\Phi,k)$-discs Hausdorff-converge to the normal unit vector to the boundary. This is a key ingredient in showing the foliation property, and that the Gauss map is proper.

\item In \S\ref{section:plateau}, we prove Theorem \ref{th.Foliated_Plateau_posta} on the existence of a foliation of $SX^\bullet$ by Gauss lifts of $(\Phi,k)$-discs $\mathrm{D}(c)$ with $c \in \cC$. In \S\ref{ss:TwoDSetting}, a two-dimensional infinitesimal version of the foliated Plateau problem is studied. It is then used in \S\ref{ssection:plateau-3d} to show the existence of the foliation and its properties on the three-dimensional manifold $X$. The Bolker condition is proved in \S\ref{ssection:bolker}, and concludes the proof of Theorem \ref{th.Foliated_Plateau_posta}, item (i). The analytic case is treated in \S\ref{ssection:analytic}. The existence of a smooth (resp. analytic) extension of the foliation across the boundary of $X$ is finally established in \S\ref{ssection:extended-foliation}.
\item In \S\ref{s.Radon}, we study the corresponding Radon transform and prove Theorem \ref{theorem:quasi-injective}. We define the Radon transform and derive its first properties in \S\ref{ss.def}. Results on the pseudodifferential behaviour of the corresponding normal operator are collected in \S\ref{ss.results_radon}. In \S\ref{ss.pseudo-interior}, using the Bolker condition for the foliation by $(\Phi,k)$-discs derived in \S\ref{ssection:bolker}, we show that the normal operator is an elliptic pseudodifferential operator in the interior $X^\circ$. To treat the analytic case, we study in \S\ref{ss.kernel} the Schwartz kernel of the normal operator near the diagonal on the blown-up space $[X \times X ; \Delta]$ where $\Delta \subset X \times X$ is the diagonal, and derive an explicit expression for this kernel from which one can deduce analyticity.
\item Finally, applications to boundary area rigidity (Theorems \ref{theorem:main} and \ref{theorem:global}) are treated in \S\ref{section:applications}. \\
\end{itemize}

\noindent \textbf{GenAI disclosure:} GenAI was used to improve writing and presentation. Lemma \ref{lemma:beta} was found with the assistance of GenAI. All remaining mathematical arguments were developed by the authors without GenAI assistance.\\

\noindent \textbf{Acknowledgement:}  This project has received funding from the European Research Council (ERC) under the European Union’s Horizon research and innovation programme (grant agreement No. 101162990) and from IRL-2030 IFUMI, Laboratorio del Plata. S.A. acknowledges financial support from CSIC through the
Grupo I+D 149-348 ``Geometría y Acciones de Grupos''. During the preparation of this
paper, S.A. benefited from a \emph{Poste Rouge} from the CNRS.

\section{Geometric preliminaries} \label{section:geometry}

\subsection{Framework} \label{ssection:setting}

Throughout this paper, $X := X^3$ will be a compact smooth manifold with boundary, diffeomorphic to the closed unit ball in $\R^3$. We denote by $\partial X$ its boundary and by $X^\circ$ its interior. We fix a Riemannian metric $\bar g$ on $X$, and we denote respectively by $\overline{\nabla}$ and by $\overline{R}$ its Levi--Civita covariant derivative and Riemann curvature tensor, adopting here the convention that
\begin{equation}
\label{equation:riemann}
\overline{R}(u,v)w:=\overline{\nabla}_u\overline{\nabla}_v w-\overline{\nabla}_v\overline{\nabla}_u w -\overline{\nabla}_{[u,v]}w\ .
\end{equation}
It will also be convenient at various points to isometrically embed $(X,\bar g)$ into a larger open manifold $(X_e, \bar g_e)$, where here $\bar g_e$ is some smooth extension of $\bar g$ across the boundary.

\subsubsection{The unit sphere bundle}\label{sss_unit_tangent_bundle}
Let $SX$ denote the \emph{unit sphere bundle} of $(X,\bar g)$, and let $\pi:SX\to X$ denote the \emph{base point projection}. Let $\nu$ denote the outward-pointing unit normal vector field along $\partial X$, and define the open subset $SX^\bullet$ of $SX$ by
\begin{equation*}
SX^{\bullet}:=SX\moins\big\{\big(x,\pm \nu(x)\big):x\in\partial X\big\}\ .
\end{equation*}

The tangent bundle of $SX$ canonically splits into
\begin{equation}\label{eq_horizontal_vertical}
T(SX) = \mathcal H \oplus \mathcal V\ ,
\end{equation}
where $\mathcal V := \ker \dd\pi$ denotes the \emph{vertical bundle}, and $\mathcal H := \ker \mathcal K$ denotes the \emph{horizontal bundle} of the Levi--Civita connection. Here $\mathcal K : TTX \to TX$ denotes the \emph{connection map}, which we recall is defined as follows. Given $v \in SX$ and $\xi \in T_v SX$, let $\gamma(t) = \big(\alpha(t), \Gamma(t)\big)$ be a smooth curve in $SX$ such that $\gamma(0)=v$, $\dot{\gamma}(0)=\xi$, and $\pi\circ\Gamma=\alpha$ (that is, $\Gamma$ is a unit vector field over $\alpha$). The connection map is determined by
\begin{equation*}
\cK_v(\xi):=(\overline{\nabla}_{\dot\alpha}\Gamma)(0)\ .
\end{equation*}

The \emph{Sasaki metric} over $SX$ is the Riemannian metric $G$ given by
\begin{equation*}
G(\xi,\zeta):=\bar g\big(\dd\pi(\xi),\dd\pi(\zeta)\big)+\bar g\big(\cK(\xi),\cK(\zeta)\big)\ .
\end{equation*}
The splitting \eqref{eq_horizontal_vertical} is trivially orthogonal with respect to the Sasaki metric. This metric also induces a volume form over $SX$, which we denote by $\dd\mu$, and which is precisely the \emph{Liouville measure} of $SX$ (see \cite[Chapter 1]{Paternain-99}).

\subsubsection{Immersed surfaces and their extrinsic geometry}\label{ss_extrinsic_geom_surfaces} Let $e:D\to X$ be an oriented immersion, let $n_e$ denote its unit normal vector field, let $g_e:=e^*\bar g$ denote the metric induced by $e$ over $D$, and let $\dd D$ denote the area form defined over $D$ by this metric. We define the \emph{Gauss lift} of $e$ to be the immersion $\hat e:D\to SX$ given by
\begin{equation}\label{eq.Gauss_lift}
\hat e(p):=\big(e(p),n_e(p)\big)\in SX\ .
\end{equation}

Let $A_e\in\Gamma(\mathrm{End}(TD))$ denote the \emph{shape operator} of $e$, with sign convention chosen such that, for all $v \in TD$,
\begin{equation*}
\dd e\cdot A_ev:=\overline{\nabla}_v n_e\ .
\end{equation*}
Recall that, at every point of $D$, $A_e$ is symmetric with respect to $g_e$. We denote respectively by $\mathrm{II}_e$ and $\mathrm{III}_e$ the \emph{second} and \emph{third fundamental forms} of $e$, which we recall are defined over $TD$ by
\begin{equation*}
\II_e(v,w):=g_e(A_ev,w)\qquad\text{and}\qquad\III_e(v,w):=g_e(A_e^2v,w)\ .
\end{equation*}
We say that $e$ is \emph{infinitesimally strictly convex} whenever $\II_e$ is everywhere positive definite.

The \emph{mean curvature} $H_e$ and the \emph{extrinsic curvature} $\kappa_{\mathrm{ext}}$ are defined respectively by
\begin{equation*}
H_e:=\frac{1}{2}\tr(A_e)\qquad\text{and}\qquad\kappa_{\mathrm{ext}}:=\det(A_e)\ .
\end{equation*}
Recall that the latter is related to the \emph{intrinsic curvature} of $D$ (viewed as a Riemannian surface) by \emph{Gauss' equation}
\begin{equation}\label{eq.Gauss_eqn}
\kappa_{int}=\kappa_{\mathrm{ext}}+\sect_{\bar g}|_{TD}\ .
\end{equation}
In the sequel, we will be interested in surfaces whose curvature is prescribed by some function defined over the ambient space. These will include minimal surfaces, constant mean curvature surfaces, and constant extrinsic curvature surfaces. It will be convenient to study these problems within a unified framework, which we introduce presently in \S\ref{ss:CurvatureFunctions}.

Finally, we denote by $\jetX$ the $2$-jet bundle of oriented immersed surfaces in $X$. This is the set of triples $(x,n,A)$ where $x\in X$, $n\in S_xX$, and $A$ is a self-adjoint endomorphism of the plane $n^\perp \subset T_xX$. Any immersion $e:D\to X$ naturally lifts to an immersion $\jete:D\to\jetX$, given by
\begin{equation*}
\jete(p)=\big(e(p),n_e(p),A_e(p)\big)\ ,
\end{equation*}
where $n_e$ and $A_e$ are as above.

\subsection{The space of circles} \label{ssection:space-of-circles} We now introduce the \emph{spaces of oriented circles} on the boundary $\partial X$ of $X$. Before addressing the general case, we first review the structure of the space of oriented circles over the round sphere.

\subsubsection{Circles over $\SS^2$} \label{sssection:s2-circles} Let $\SS^2$ denote the unit sphere in $\R^3$ equipped with its canonical metric. We denote by $\cC:=\cC_{\SS^2}$ the \emph{space of oriented round circles} in $\SS^2$, that is oriented circles obtained by intersecting $\SS^2$ with oriented affine $2$-planes in $\R^3$. A natural parametrization of $\cC$ is given as follows. For each $v \in \SS^2$ and $\rho\in (-1,1)$, we denote
\begin{equation*}
c(v,\rho) := (\rho v + v^\perp) \cap \SS^2\ ,
\end{equation*}
where $v^\perp$ denotes the linear $2$-plane orthogonal to $v$. Note that $v^\perp$ is naturally oriented by the $2$-form $\iota_v \omega_{\R^3}$, where $\omega_{\R^3}$ denotes the standard volume form of $\R^3$. The orientation of $c(v,\rho)$ is then chosen to be compatible with $\iota_v \omega_{\R^3}$, that is $c(v,\rho)$ is parametrized in the anticlockwise direction. By construction, $\cC$ is diffeomorphic to $\SS^2 \times (-1,1)$ and is trivially compactified by $\SS^2\times[-1,1]$.

We also view $\cC$ as a symmetric space as follows. Using the Poincaré ball model, we identify $\SS^2$ with the boundary at infinity of hyperbolic $3$-space $\HH^3$. Via this identification, $\PSLc$ acts on $\SS^2$ by M\"obius transforms, and this action is transitive, both on $\SS^2$ and on the space of circles $\cC$. Every round circle $c \subset \SS^2$ bounds a unique totally geodesic $\HH^2$ in $\HH^3$ whose stabilizer in $\PSLc$ is a copy of $\PSL$. The stabilizer of every round circle is thus also a copy of $\PSL$, and in this manner, we identify $\cC$ with the symmetric quotient space $\cC=\PSLc/\PSL$.

The symmetric space $\PSLc/\PSL$ also identifies with $(2+1)$-dimensional de Sitter space $\dS^{2,1}$. In particular, the volume form of $\dS^{2,1}$ then induces a volume form over $\cC$, which also coincides with the volume form induced by the Haar measure of $\PSLc$. This will be discussed further in \S\ref{ssection:example}. We will also investigate further how the spacetime structure of $\dS^{2,1}$ is realized geometrically in $\cC$ in the following section.

\subsubsection{The tangent space of the space of circles}
\label{sss.First_order_variations} We now denote by $c_0 := \mathbb R^2 \cap \SS^2$ the oriented equator with normal $\partial_3 \in (\R^2)^\perp$, where $\partial_3$ here denotes the unit upward pointing vertical vector field on $\R^3$. Let $P$ be an oriented affine plane in $\R^3$, and let $c := P \cap \SS^2$ denote the oriented round circle in $\SS^2$ that it defines. Let $[0,2\pi] \ni \theta \mapsto c(\theta)$ be an oriented constant speed parametrization, and let $\theta \mapsto n_c(\theta)$ denote the oriented unit normal vector field over $c$. Consider a smooth variation of circles
\begin{equation*}
(-\eps,\eps) \to \cC\ , \qquad  t \mapsto c_t := P_t \cap \SS^2\ .
\end{equation*}
We view $(c_t)_{t \in (-\eps,\eps)}$ as a family of normal graphs over $c$. Indeed, there exists a (smooth) family of smooth functions $(F_t)_{t \in (-\eps,\eps)}$, $F_t : [0,2\pi] \to \R$ such that
\begin{equation*}
c_t = \big\{\exp_{c(\theta)}\left(F_t(\theta)n_c(\theta)\right) ~:~ \theta \in [0,2\pi]\big\}\ ,
\end{equation*}
where here $\exp$ denotes the Riemannian exponential map of $\SS^2$. Denoting $\SS^1 := \R/2\pi\Z$, and differentiating with respect to $t$ at $t=0$, we obtain the natural identification
\begin{equation*}
T_{c}\cC \simeq \cT_c := \big\{\theta \mapsto \partial_t F_t(\theta)|_{t=0}\big\} \subset C^\infty(\SS^1)\ .
\end{equation*}
The set $\cT_c$ is given explicitly as follows.

\begin{lemma}
\label{lemma.trigo_positive}
For all $c \in \cC$,
\begin{equation*}
\cT_c = \big\{ f_{\alpha,\beta,\gamma} ~:~ \alpha,\beta,\gamma\in\R \big\}\ ,
\end{equation*}
where, for all $\alpha,\beta,\gamma\in\mathbb{R}$,
\begin{equation*}
f_{\alpha,\beta,\gamma}:\mathbb{S}^1\rightarrow\mathbb{R};\ \theta\mapsto\gamma-\alpha\cos(\theta)-\beta\sin(\theta)\ .
\end{equation*}
\end{lemma}
\begin{proof}
Upon rotating, we may suppose that $c$ is the horizontal circle
\begin{equation*}
c:=\big\{\big(x_1,x_2,\sin(\phi_0)\big)\ ~:~\ x_1^2+x_2^2+\sin^2(\phi_0) = 1 \big\}\ .
\end{equation*}
Given $(\alpha,\beta,\gamma)\in\R^3$, we denote
\begin{equation*}
P_{\alpha,\beta,\gamma} :=\big\{ (x_1,x_2,x_3)\ ~:~\ \alpha x_1 + \beta x_2 + x_3 = \gamma + \sin(\phi_0)\big\}\ ,
\end{equation*}
and we denote $c_{\alpha,\beta,\gamma} := P_{\alpha,\beta,\gamma}\cap\SS^2$. By linearity and symmetry, it suffices to consider variations of the form $(c_{0,0,\gamma})_{\gamma\in(-\eps,\eps)}$ and $(c_{\alpha,0,0})_{\alpha\in (-\eps,\eps)}$.

In the first case, for all $\gamma$, $c_{0,0,\gamma}$ is the normal graph over $c$ of the constant function
\begin{equation*}
\psi_\gamma(\theta) := \arcsin\big(\gamma+\sin(\phi_0)\big) - \phi_0\ .
\end{equation*}
Upon differentiating, we see that $\cT_c$ contains the constant function $f_{0,0,1}$.

In order to address the second case, we parametrize $c$ explicitly by
\begin{equation*}
c(\theta) := \big(\cos(\phi_0)\cos(\theta),\cos(\phi_0)\sin(\theta),\sin(\phi_0)\big)\ .
\end{equation*}
The upward pointing unit normal vector field over $c$ is then given by
\begin{equation*}
n_c(\theta) = \big(-\sin(\phi_0)\cos(\theta),-\sin(\phi_0)\sin(\theta),\cos(\phi_0)\big)\ .
\end{equation*}
Given a smooth function $\psi:[0,2\pi]\rightarrow\R$, its normal graph over $c$ is parametrized by
\begin{equation*}
c_\psi(\theta) := \cos(\psi(\theta))c(\theta) + \sin(\psi(\theta))n_c(\theta)\ ,
\end{equation*}
which lies in the plane $P_{\alpha,0,0}$ provided that $\langle c_\psi(\theta),(\alpha,0,1)\rangle = \sin(\phi_0)$. Using elementary trigonometry, we find that this holds provided that $\psi=\psi_\alpha$, where $\psi_\alpha$ is defined implicitly by
\begin{equation*}
\alpha\cos(\psi_\alpha(\theta) + \phi_0)\cos(\theta) + \sin(\psi_\alpha(\theta)+ \phi_0) = \sin(\phi_0)\ .
\end{equation*}
Differentiating with respect to $\alpha$ yields
\begin{equation*}
\partial_\alpha\psi_\alpha(\theta)|_{\alpha=0} = -\cos(\theta) = f_{1,0,0}(\theta)\ .
\end{equation*}
It follows that $f_{1,0,0}\in\cT_c$, and this completes the proof.
\end{proof}

We denote by $P^+\mathcal{T}_c:=\mathcal{T}_c\setminus\{0\}/\R^+$ the projective space of \emph{oriented} lines in $\cT_c$, and we denote the oriented projective class of any element $f\in\cT_c\setminus\{0\}$ by $[f]$. The number and type of zeroes of any non-trivial element $f\in\cT_c$ trivially only depends on the projective class $[f]$ and is in fact entirely determined by the discriminant
\begin{equation*}
\Delta(f)=\alpha^2+\beta^2-\gamma^2\ .
\end{equation*}
Indeed,
\begin{itemize}
\item if $\Delta(f) < 0$, then $f$ has no zeros, and the corresponding variation is a one-parameter family of disjoint circles;
\item if $\Delta(f) = 0$, then $f$ has a unique zero, which is a non-degenerate extremum, and the corresponding variation consists of circles sharing a common tangent at a single fixed point; and
\item if $\Delta(f) > 0$, then $f$ has exactly two simple zeros, and the corresponding variation consists of circles rotating about an axis.
\end{itemize}
The polarization of $\Delta$ defines a Minkowski structure over $\mathcal{T}_c$ whose timelike, lightlike, and spacelike directions correspond respectively to $\Delta < 0$, $\Delta = 0$, and $\Delta > 0$. Since this holds for every tangent space of $\cC$, we see that $\Delta$ furnishes $\cC$ with the structure of a Lorentzian spacetime. This is precisely the $(2+1)$-dimensional de Sitter structure described in the preceding section.

\subsubsection{Foliating the space of spacelike perturbations}\label{sss.foliating_space_like}
For all $c$, we view the oriented projective space $P^+\cT_c$ as a $2$-dimensional sphere. We now introduce various natural structures over this sphere which will play a central role in establishing the foliation property in \S\ref{section:plateau}. We first consider the subsets
\begin{equation*}
\cT_c^+:=\{f:\Delta(f)>0\}\ , \qquad \cT_c^0:=\{f:\Delta(f)=0\}\ ,\qquad\cT_c^-:=\{f:\Delta(f)<0\}\ ,
\end{equation*}
and we denote by $P^+\mathcal T_c^+$, $P^+\mathcal T_c^0$, and $P^+\mathcal T_c^-$ their respective projectivizations in $P^+\mathcal T_c$.

Every element $[f]$ of $P^+\cT_c^+$ identifies with a unique trigonometric polynomial $f(\theta)=\gamma-\alpha\cos(\theta)-\beta\sin(\theta)$ satisfying $\alpha^2+\beta^2=1$ and $|\gamma|<1$ which we call its \emph{normalized representative}. In this manner, we identify the set $P^+\cT_c^+$ with the open annulus $\mathbb{S}^1\times (-1,1)$. Similarly, since $P^+\cT_c^0$ coincides with the boundary of $P^+\cT_c^{+}$, it identifies with a disjoint union of two circles $\mathbb{S}^1\times\{\pm 1\}$. Finally, $P^+\cT_c^-$ identifies with a disjoint union of two open discs.

The annulus $P^+\cT_c^+$ is naturally endowed with a foliation $\cW$ by open arcs whose leaves are the level sets of the map associating to each element $[f]$ the first two coordinates $(\alpha,\beta)$ of its normalized representative. This foliation is trivially transverse to the foliation by circles given by level sets of the third coordinate.

Finally, we note that there is a natural continuous map
\begin{equation*}
\zeta:P\cT^0\to\mathbb{S}^1,\,[f]\mapsto\Zero(f)\ .
\end{equation*}

These objects and their properties are all that we will require to prove the main technical result of \S\ref{ss:TwoDSetting}, namely Theorem \ref{thm:coveringofunitbundle}. More precisely the key property of interest to us is the following elementary observation.

\begin{lemma}\label{lem.invariant_-1}
The objects $P^+\cT^+$, $P^+\cT^0$, $\cW$, and $\zeta$ are all invariant under multiplication by $-1$.
\end{lemma}

\subsubsection{The general case} \label{sssection:circles-x}We return now to the setting of \S\ref{ssection:setting} and consider the smooth closed three-dimensional Riemannian ball $(X,\bar{g})$. Let $\alpha : \partial X \to \SS^2$ be a smooth diffeomorphism. If the metric $\bar g$ is further assumed to be analytic, then we suppose also that $\alpha$ is analytic.

\begin{defi}[Space of circles on $\partial X$] \label{definition:space-of-circles}We define the space of circles on the boundary of $X$ by
\begin{equation*}
\cC:=\cC_{(X,\alpha)} := \alpha^{-1}\big(\cC_{\SS^2}\big)\ ,
\end{equation*}
where $\cC_{\SS^2}$ denotes the space of circles on the round sphere described in \S\ref{sssection:s2-circles}.
\end{defi}

This space trivially depends on the choice of $\alpha$. However, a \emph{canonical} choice can also be made as follows. By Riemann's uniformization theorem, there exists a conformal diffeomorphism $\alpha : (\partial X,g_{\partial X}) \to \SS^2$ which is unique up to post-composition by M\"obius maps. Since M\"obius maps preserve $\cC_{\SS^2}$, the family
\begin{equation*}
\cC_X := \alpha^{-1}(\cC_{\SS^2})
\end{equation*}
is independent of the conformal diffeomorphism chosen.

\begin{defi}[Adapted space of circles]
\label{definition:adapted}
We say that the space of circles $\cC:=\cC_X$ is \emph{adapted} (to the metric $\bar g$) if it is constructed as above.
\end{defi}

\subsection{Example: geodesic balls in hyperbolic $3$-space} \label{ssection:example} Consider the Kleinian model of $3$-dimensional hyperbolic space $\HH^3$. This is the open unit ball in $\R^3$ furnished with the metric
\begin{equation*}
g_{\mathrm{klein}} := \frac{1}{1-\|x\|^2}g_{\mathrm{euc}} + \frac{1}{(1-\|x\|^2)^2}\bigg(\sum_{i=1}^3 x_i dx^i\bigg)^2\ .
\end{equation*}
In this model, the geodesics are simply straight lines.

The ideal boundary of $\HH^3$ identifies with the unit sphere $\SS^2 \subset \R^3$. Recall now the space $\cC:=\cC_{\SS^2}$ of circles on $\SS^2$ described above. Since totally geodesic planes in $\HH^3$ are given in this model by intersections of the unit ball with affine planes in $\R^3$, we see that every circle $c\in\cC_{\mathbb{S}^2}$ bounds a unique totally geodesic, and hence minimal, copy of the hyperbolic plane, which we denote by  $\HH_c$. In particular, this yields the \emph{Radon transform on totally geodesic planes}
\begin{equation*}
R_{\HH^3} : C_{\mathrm{comp}}(\HH^3) \longrightarrow C(\cC)\ , \qquad
R_{\HH^3}f(c) := \int_{\HH_c} f ~\dd \HH_c\ ,
\end{equation*}
where $\dd\HH_c$ denotes the area measure of the hyperbolic metric of $\HH_c$. We refer to \cite{Helgason-59,Semyanistyj-61,Kurusa-91} for further details on the Radon transform on locally symmetric spaces, especially in the setting of hyperbolic spaces.

Let $o\in\HH^3$ be a fixed base point, which we identify with the origin in $\R^3$. For all $r>0$, let $X:=B_o(r)$ denote the hyperbolic ball of radius $r>0$, and recall that this coincides with the euclidean ball of radius $\tanh(r)$. Its boundary $\partial X$, furnished with the restriction of the hyperbolic metric, has constant positive sectional curvature equal to $\sinh(r)^{-2}$. Let $U_r\subset\cC_{\SS^2}$ denote the open subset consisting of those circles $c$ for which the totally geodesic plane $\HH_c$ intersects $\partial X$ non-trivially and transversely. The resulting family of curves
\begin{equation*}
\cC := \big\{\HH_c\cap\partial X : c\in U_r\big\}
\end{equation*}
then defines a space of circles on $\partial X$. Since every element of $\cC$ is given by the intersection of $\partial X$ with a Euclidean plane, this family coincides with the space of adapted circles of $\partial X$ (see Definition \ref{definition:adapted}). Likewise, each $c':=\HH_c\cap\partial X$ bounds a unique totally geodesic disc in $X$, namely, the intersection $\HH_{c}\cap X$.

In the minimal case, it follows from \cite{Barbosa-doCarmo-80} that $(X, g_{\mathrm{klein}})$ is area simple (Definition \ref{definition:area-simple}). Still in the minimal case, in \cite{Alvarez-Lefeuvre-Lowe-Smith} we show that area simplicity follows from the more general assumptions \hyperlink{AA1}{\rm(A1-A4)}, which are clearly also satisfied by the unit ball in $\HH^3$.

As a straightforward application, Theorem \ref{theorem:quasi-injective} can be used to recover the following classical result.

\begin{corollary}
\label{cor}
The Radon transform $R_{\HH^3} : C_{\mathrm{comp}}(\HH^3) \to C(\cC)$ on totally geodesic planes in $\HH^3$ is injective.
\end{corollary}

\begin{proof}[Proof of Corollary \ref{cor}]
Choose $f \in C_{\mathrm{comp}}(\HH^3) \cap \ker R_{\HH^3}$, and let $r > 0$ be such that the support of $f$ is contained in $X := B_r(o)$. Let $R_0 : C_{\mathrm{comp}}(X) \to C(\cC)$ denote the corresponding Radon transform over $X$. By Theorem \ref{theorem:quasi-injective}, $R_0$ is injective, and hence $f \equiv 0$, as desired.
\end{proof} 

\section{$(\Phi,k)$-surfaces}
\label{section:phik}
In this section, we introduce $(\Phi,k)$-surfaces and discuss their main properties. We begin in \S\ref{appendix:asymptotic} by describing the geometry of small normal graphs over some fixed surface. In \S\ref{ss:CurvatureFunctions}, we define elliptic curvature functions over the $2$-jet bundle of $X$. In \S\ref{ssection:phiksurfaces}, we define $(\Phi,k)$-surfaces and in \S\ref{ss:Stability} we discuss the key concepts of \emph{stability} and \emph{strict stability}. Finally, the associated Plateau problem is explained in \S\ref{ssection:plateau-problems}.

\subsection{The geometry of graphs}
\label{appendix:asymptotic}
We begin by reviewing some general results on the geometry of graphs that will be used extensively throughout this section.

\subsubsection{Fermi coordinates along the boundary}
Let $\Omega \subset \R^2$ be an open subset, fix $\eps > 0$ and let $\bar g$ be a smooth metric over $\Omega \times (-\eps,\eps) \in \R^3$ of the form
\begin{equation*}
\bar g = g_t + \dd t^2\ ,
\end{equation*}
where $g_t$ denotes the induced metric on the parallel surface $S_t := \Omega \times \{t\}$. Note that this is the form of a general Riemannian metric written in Fermi coordinates around a smoothly immersed surface.

For all $t$, let $\nu = \partial_t$ denote the upward-pointing unit normal vector field over $S_t$, let
\begin{equation*}
\II_t(u,v) := g_t\big(\overline\nabla_u \nu,v\big)
\qquad\text{and}\qquad
A_t u := \overline\nabla_u \nu
\end{equation*}
denote respectively the \emph{second fundamental form} and the \emph{shape operator} of $S_t$, let
\begin{equation*}
\III_t(u,v) := g_t\big(A_tu,A_tv\big)
\end{equation*}
denote its \emph{third fundamental form}, and let
\begin{equation*}
\mathcal R_{\nu,t}(u,v) := \bar g\big(\bar R(\nu,u)\nu,v\big)
\end{equation*}
denote its \emph{normal curvature}.
\begin{lemma}
For all $t$, the derivatives of $g_t$ and $\II_t$ satisfy
\begin{equation}
\label{eq.normal_variation_equations}
\begin{aligned}
\partial_t g_t &= 2\II_t,\ \qquad  \partial_t \II_t &= \mathcal R_{\nu,t}+\III_t\ .
\end{aligned}
\end{equation}
\end{lemma}
\begin{proof}
Let $u$ and $v$ be vector fields that are constant in the $t$-direction, that is, such that
\begin{equation*}
[u,\partial_t] = [v,\partial_t] = 0\ .
\end{equation*}
Differentiating $g_t(u,v)=\bar g_{(x,t)}(u,v)$ with respect to $t$ yields
\begin{equation*}
\partial_t\bar g_{(x,t)}(u,v) = \partial_t g_t (u,v) = \bar g\big(\overline\nabla_{\partial_t} u, v\big) + \bar g\big(u, \overline\nabla_{\partial_t} v\big)\ .
\end{equation*}
Since the Levi--Civita covariant derivative is torsion free,
\begin{equation*}
\bar g\big(\overline\nabla_{\partial_t} u, v\big) = \bar g\big(\overline\nabla_{u} \partial_t, v\big)=\bar g\big(\overline\nabla_{u} \nu, v\big)=\II_t(u,v)\ .
\end{equation*}
Similarly,
\begin{equation*}
\bar g\big(u,\overline\nabla_{\partial_t} v\big) = \II_t(u,v)\ ,
\end{equation*}
and it follows that $\partial_t g_t = 2\II_t$, as desired.

Differentiating a second time with respect to $t$ yields
\begin{equation*}
\partial_t \II_t(u,v) = \bar g\big(\overline \nabla_{\partial_t} \overline \nabla_u \partial_t, v\big) + \bar g\big(\overline \nabla_u \partial_t, \overline \nabla_{\partial_t} v\big)\ .
\end{equation*}
Applying to the second term the torsion-free property of the Levi--Civita covariant derivative yields
\begin{equation*}
\bar g\big(\bar\nabla_u\partial_t,\bar\nabla_{\partial_t}v\big) = g_t\big(A_t u,A_t v\big) = g_t\big(A_t^2 u,v\big) = \III_t(u,v)\ .
\end{equation*}
Since $\nabla_{\partial_t} \partial_t = 0$ and $[u,\partial_t]=0$, the first term becomes
\begin{equation*}
\bar g\big(\overline \nabla_{\partial_t} \overline \nabla_u \partial_t, v\big) =\bar g\big(\bar\cR(\partial_t,u)\partial_t, v\big) = \cR_{\nu,t}(u,v)\ ,
\end{equation*}
and the second identity follows.
\end{proof}

It follows by \eqref{eq.normal_variation_equations} that, at $t=0$,
\begin{equation}
\label{eq.gt_IIt_expansions}
\begin{aligned}
g_t &= g+2t\II+\cO(t^2)\ \text{and}\\
\II_t &= \II+t(\mathcal R_\nu+\III)+\cO(t^2)\ ,
\end{aligned}
\end{equation}
where here $g=g_0$, $\II=\II_0$, $A=A_0$, $\III=\III_0$, and $\mathcal R_\nu=\mathcal R_{\nu,0}$.

\subsubsection{Graphs}
\label{sss:graphs}
Consider a smooth function $f:\Omega\to(-\varepsilon,\varepsilon)$ and let
\begin{equation*}
\hat f(x) := \big(x,f(x)\big)
\end{equation*}
denote the standard parametrization of its graph, and let $n_{\hat f}$ denote its upward-pointing \emph{unit normal vector field} with respect to the metric $\bar g$. Given a vector field $u$ over $\Omega$, we denote by
\begin{equation*}
u_f := u+\dd f(u)\partial_t = \dd \hat{f}(u)
\end{equation*}
the corresponding tangent vector of the graph. We denote respectively by
\begin{equation*}
\opI_{\hat f}\ ,\qquad\opII_{\hat f}\qquad\text{and}\qquad A_{\hat f}
\end{equation*}
the \emph{first} and \emph{second fundamental forms} and \emph{shape operator} of the immersion $\hat{f}$, that is, for vectors $u$ and $v$,
\begin{equation*}
\opI_{\hat f}(u,v) := \bar g\big(\dd \hat f(u),\dd \hat f(v)\big)\qquad\text{and}\qquad \opII_{\hat f}(u,v) := \bar g\big(\overline\nabla_{\dd \hat f(u)}n_f,\dd \hat f(v)\big)\ ,
\end{equation*}
and, for all vectors $u$ and $v$,
\begin{equation*}
\opI_{\hat f}\big(A_{\hat f}u,v\big)=\opII_{\hat f}(u,v)\ .
\end{equation*}
That is, by mild abuse of notation,
\begin{equation}
\label{eq.shape_operator_graph_definition}
A_{\hat f}=\opI_{\hat f}^{-1}\cdot\opII_{\hat f}\ ,
\end{equation}
where the product and inverse are taken by identifying symmetric bilinear forms with matrices via the Euclidian metric in the usual manner.

It is important to emphasize, especially for the asymptotic analysis that will be carried out in \S\ref{section:asymptotic}, that $n_{\hat f}$, $\I_{\hat f}$, $\II_{\hat f}$, and $A_{\hat f}$ are understood here as non-linear partial differential operators acting on the function $f$ (or, more correctly, on the immersion $\hat{f}$). We make this precise as follows. For all $k$, we denote by $\mathcal{J}^k\R^2$ the space of $k$-jets of smooth functions over $\R^2$. In particular,
\begin{equation*}
\begin{aligned}
\mathcal{J}^1\R^2 &= \R\times\R^2\ ,\ \text{and}\\
\mathcal{J}^2\R^2 &= \R\times\R^2\times\opSym(\R^2)\ ,
\end{aligned}
\end{equation*}
where $\opSym(\R^2)$ here denotes the space of symmetric bilinear forms over $\R^2$. For any smooth function, for any $k$, and for any $p$, we likewise denote by $J^kf(p)$ the $k$-jet of $f$ at $p$, that is
\begin{equation*}
J^1f(p) := \big(f(p),\dd f(p)\big)\qquad\text{and}\qquad J^2f(p) := \big(f(p),\dd f(p),\dd^2 f(p)\big)\ .
\end{equation*}
Using the explicit formulae for $n_{\hat f}$ and $A_{\hat f}$, we see that there exist smooth functions $\mathbf{n}:\Omega\times\mathcal{J}^1\R^2\rightarrow\R^3$ and $\mathbf{A}:\Omega\times\mathcal{J}^2\R^2\rightarrow\opEnd(\R^2)$ such that, for any smooth function $f$, and for all $p$,
\begin{equation*}
n_{\hat f}(p) = \mathbf{n}\big(p,J^1f(p)\big)\qquad\text{and}\qquad A_{\hat f}(p) = \mathbf{A}\big(p,J^2f(p)\big)\ .
\end{equation*}
In particular, fixing the base point $p \in \Omega$, the functions $\mathbf{n}\big(p,\bullet \big)$ and $\mathbf{A}\big(p,\bullet\big)$ are respectively non-linear functions on the vector spaces $\mathcal{J}^1_p\R^2$ and $\mathcal{J}^2_p\R^2$ which thus have well-defined Taylor expansions about the zero section. We will make considerable use of this perspective in \S\ref{section:asymptotic}.

\begin{lemma}
\label{lemma:FormulaeForIAndII}
The first and second fundamental forms and shape operator of the graph of $\hat f$ satisfy
\begin{equation}
\label{eq.graph_I_II_expansion}
\begin{aligned}
\opI_{\hat f} &= g+2f\II+\cO\big((J^1f)^2\big)\ ,\\
\opII_{\hat f} &= \II-\Hess_g f + f(\mathcal R_\nu+\III) + \cO\big((J^1f)\cdot (J^2f)\big)\ ,\ \text{and}\\
A_{\hat f} & = A-(\Hess_g f)^\sharp + f(\mathcal R_\nu^\sharp-A^2) + \cO\big((J^1f)\cdot (J^2f)\big)\ ,
\end{aligned}
\end{equation}
where here we denote $(\Hess_g f)^\sharp:=g^{-1}\Hess_g f$ and $\mathcal R_\nu^\sharp:=g^{-1}\mathcal R_\nu$.
\end{lemma}
\begin{remark}We underline that we have used here the notion introduced in Appendix \ref{ss:Notation}. The above equalities thus hold pointwise at every $p \in \Omega$, with remainders that are smooth in all arguments including $p$, and contain factors of $J^1f$ and $J^2f$. More precisely, there exists, for example, a smooth function $R : \mathcal{J}^1\Omega \to \opEnd(\R^2)$ such that, for all $p$,
\begin{equation*}
\opI_{\hat f}(p) = g(p)+2f(p)\II(p)+ R(p, J^1f)\ ,
\end{equation*}
and, furthermore, for all $p$ and for all $\xi$,
\begin{equation*}
R(p,\xi) = \sum_{i,j} \xi_i \xi_j R_{ij}(p,\xi)
\end{equation*}
for some smooth functions $R_{ij} : \mathcal{J}^1\Omega \to \opEnd(\R^2)$. The remaining two equations in \eqref{eq.graph_I_II_expansion} may also be expanded in a similar manner.
\end{remark}

The proof of Lemma \ref{lemma:FormulaeForIAndII} will require the following standard formula for the Hessian operator over immersed hypersurfaces. For any smooth function $\phi \in C^\infty\big(\Omega \times (-\eps,\eps)\big)$,
\begin{equation}
\label{equation:hessian-utile}
\overline\Hess~\phi=\Hess_t~\phi +\dd \phi(\nu)\cdot\II_t\ ,
\end{equation}
where $\Hess_t$ denotes the Hessian of $g_t$ on the slice $S_t=\Omega \times \{t\}$. Indeed, denoting the Levi-Civita connection of $g_t$ by $\nabla^t$, we recall that
\begin{equation*}
\bar\nabla_u v = \nabla^t_u v  - \II_t(u,v)\cdot\partial_t\ .
\end{equation*}
Hence
\begin{equation*}
\begin{split}
\overline\Hess~\phi(u,v) &= \dd_u\dd_v\phi - \dd \phi(\overline\nabla_u v)\\
&= \dd_u\dd_v\phi-\dd \phi(\nabla^t_u v) + \dd \phi\big(\II_t(u,v)\cdot\partial_t\big) \\
&= \Hess_t~\phi(u,v) + \dd \phi(\nu)\cdot\II_t(u,v)\ ,
\end{split}
\end{equation*}
from which \eqref{equation:hessian-utile} follows. We now prove Lemma \ref{lemma:FormulaeForIAndII}.
\begin{proof}
Denote
\begin{equation*}
F(x,t) := t-f(x)\ ,
\end{equation*}
so that the graph of $f$ is $F^{-1}(0)$. Along this graph, the upward-pointing unit normal is
\begin{equation*}
n_{\hat f}=\frac{\overline\nabla F}{|\overline\nabla F|}=\frac{\partial_t-\nabla_{g_f}f}{\sqrt{1+|\nabla_{g_f}f|_{g_f}^2}}\ ,
\end{equation*}
where $g_f$ here denotes $g_t$ evaluated at $t=f(x)$. \\

\emph{The first fundamental form.} The first fundamental form is given by
\begin{equation*}
\opI_{\hat f}(u,v)=\bar g(u^f,v^f)=g_f(u,v)+\dd f(u)\cdot\dd f(v)\ .
\end{equation*}
Applying $g_f=g+2f\II+O(f^2)$ yields
\begin{equation*}
\opI_{\hat f}=g+2f\II+\cO\big(f^2+\dd f^2\big)\ ,
\end{equation*}
and the first identity follows.\\

\emph{The second fundamental form.} Note first that, by \eqref{equation:hessian-utile},
\begin{equation*}
\overline\Hess~ F(u^f,v^f) =\overline g(\overline\nabla_{u^f}\overline\nabla F,v^f) = |\overline\nabla F|\cdot\overline g(\overline\nabla_{u^f}n_{\hat f},v^f) = |\overline\nabla F|\cdot\opII_{\hat f}(u,v)\ .
\end{equation*}
Hence
\begin{equation}
\label{eq.II_graph_HessF}
\opII_{\hat f}(u,v) =|\overline\nabla F|^{-1}\cdot\overline\Hess~ F(u^f,v^f)\ .
\end{equation}
We now compute $\overline\Hess~ F(u^f,v^f)$. For each fixed $t$, we denote respectively by $\nabla^t$ and $\Hess_t$ the Levi--Civita covariant derivative and Hessian operator of $(S_t,g_t)$. The difference between the ambient covariant derivative and the product covariant derivative $\nabla^t\oplus d_t$ is given by
\begin{equation*}
\overline\nabla_uv=\nabla^t_uv-\II_t(u,v)\cdot\partial_t\ ,
\qquad\qquad\overline\nabla_u\partial_t = A_tu\ ,\qquad\text{and}
\qquad\qquad\overline\nabla_{\partial_t}\partial_t=0\ .
\end{equation*}
For $u,v\in T\Omega$, we now claim that
\begin{equation}
\label{eq.HessianOfFHat_intrinsic}
\begin{aligned}
\overline\Hess~ F(u^f,v^f)&=-\Hess_t f(u,v)+\II_t(u,v) \\
&\qquad\qquad+\dd f(u)\cdot\II_t(v,\nabla^t f)+\dd f(v)\cdot\II_t(u,\nabla^t f)\ .
\end{aligned}
\end{equation}
Indeed,
\begin{equation*}
\begin{split}
\overline\Hess~ F(u^f,v^f) & = \overline\Hess~ F(u,v) + \dd f(u)\cdot\dd f(v)\cdot\overline\Hess~ F(\partial_t,\partial_t) \\
&\qquad\qquad + \dd f(u)\cdot\overline\Hess~ F(\partial_t,v) + \dd f(v)\cdot\overline\Hess~ F(u,\partial_t)\ .
\end{split}
\end{equation*}
Since $F(x,t)=t-f(x)$,
\begin{equation*}
\overline\Hess~ F(\partial_t,\partial_t)=0\ .
\end{equation*}
Next,
\begin{equation*}
\overline\Hess~ F(\partial_t,v) = \dd_v\cdot(\partial_t F) - \dd F(\overline\nabla_v \partial_t) = 0-\bar g(\overline\nabla  F, A_t v) = \bar g(\nabla^t f, A_t v) = \II_t(v,\nabla^t f)\ .
\end{equation*}
Finally, by \eqref{equation:hessian-utile},
\begin{equation*}
\overline\Hess~ F(u,v) = \Hess_t~F(u,v) + \dd F(\partial_t)\cdot\II_t(u,v) = -\Hess_t f(u,v) + \II_t(u,v)\ .
\end{equation*}
Combining the previous identities yields \eqref{eq.HessianOfFHat_intrinsic}.

The last two terms in \eqref{eq.HessianOfFHat_intrinsic} are $\cO(\dd f^2)$. Likewise,
\begin{equation}
\label{equation:lalaventilo}
|\overline\nabla F|^{-1}=\bigl(1+|\nabla_{g_f}f|_{g_f}^2\bigr)^{-1/2}
=1+\cO\big(\dd f^2\big)\ .
\end{equation}
Since $\nabla^t-\nabla=\cO(t)$, setting $t=f(x)$ yields
\begin{equation}
\label{equation:lalaventilo2}
\Hess_t f = \Hess_g f+O(f\cdot \dd f)\ .
\end{equation}
and after setting $t=f(x)$, this gives an error $\cO\big(f,\dd f\big)$.

Equations \eqref{eq.HessianOfFHat_intrinsic}, \eqref{equation:lalaventilo} and \eqref{equation:lalaventilo2} together yield
\begin{equation*}
\II_{\hat f}(X,Y)=\big(1+\cO(\dd f^2)\big)\cdot\big(-\Hess_g f(X,Y) + \II_t(X,Y) + \cO\big((f,\dd f)^2\big)\big)\ .
\end{equation*}
Thus, by \eqref{eq.gt_IIt_expansions} with $t=f(x)$,
\begin{equation*}
\opII_{\hat f}=\II-\Hess_g f+f(\mathcal R_\nu+\III)+\cO\big((J^1f)\cdot(J^2f)\big)\ ,
\end{equation*}
as desired. \\

\emph{The shape operator.} Finally,
\begin{equation*}
\I_{\hat f}^{-1} = \big(g + 2f\II +\cO((J^1f)^2)\big)^{-1} = g^{-1} - 2f g^{-1}\II g^{-1} + \cO\big((J^1f)^2\big)\ .
\end{equation*}
Thus, by \eqref{eq.shape_operator_graph_definition},
\begin{equation*}
\begin{split}
A_{\hat f} & = \big(g^{-1} - 2f g^{-1}\II g^{-1} + \cO\big((J^1f)^2\big)\big)\cdot\\
&\qquad\qquad\big(\II-\Hess_g f+f(\mathcal R_\nu+\III)+\cO\big((J^1f)\cdot(J^2f)\big)\big)  \\
& = g^{-1}\II - g^{-1} \Hess_g f + f g^{-1}(\mathcal R_\nu+\III) - 2f g^{-1}\II g^{-1}\II + \cO\big((J^1 f)\cdot(J^2 f)\big) \\
& = A - \Hess_g^\sharp f + f(\mathcal R_\nu^\sharp + A^2)-2fA^2 +\cO\big((J^1 f)\cdot(J^2 f)\big) \\
& = A -\Hess_g^\sharp f + f(\mathcal R_\nu^\sharp-A^2)+\cO\big((J^1 f)\cdot(J^2 f)\big)\ ,
\end{split}
\end{equation*}
and this completes the proof.
\end{proof}

\subsection{Elliptic curvature functions}
\label{ss:CurvatureFunctions}

In order to address general Plateau problems, we will use the geometric framework of curvature functions introduced by Caffarelli–Niren\-berg–Spruck \cite{Caffarelli-Nirenberg-Spruck-88} (see also \cite{White-87}). Although this formalism naturally extends to higher dimensions, we will only be concerned in this paper with the $(2+1)$-dimensional case.

\subsubsection{Curvature functions of symmetric endomorphisms}\label{sss.Curvature_functions_def}  We furnish $\mathbb{R}^2$ with its natural Euclidean scalar product. We denote by $\Symm(\mathbb{R}^2)\dans\End(\R^2)$ the vector space of self adjoint linear operators on $\mathbb{R}^2$ which we identify with the space of real, symmetric $2\times 2$ matrices in the usual manner. The natural scalar product of this space is given by
\begin{equation}
\label{equation:trala}
\langle A,B\rangle := \tr(AB)\ .
\end{equation}
The orthogonal group $\mathrm{O}(2)$ acts isometrically on $\Symm(\mathbb{R}^2)$ by conjugation.

Let $\Lambda^+\subseteq\Symm(\mathbb{R}^2)$ denote the open cone of positive definite symmetric bilinear forms. Let $\Lambda\subseteq\Symm(\mathbb{R}^2)$ be an open subset such that
\begin{enumerate}[label=(\roman*)]
\item $\Lambda$ is convex;
\item $\Lambda$ is $\mathrm{O}(2)$-invariant;
\item for all $A\in\Lambda$ and all $t>0$, $tA\in\Lambda$; and
\item for all $A\in\Lambda$ and all $B\in\Lambda^+$, $A+B\in\Lambda$.
\end{enumerate}

Observe that these conditions imply that $\Lambda^+\subseteq\Lambda$. Moreover, if $\Lambda\neq\Symm(\mathbb{R}^2)$, then convexity implies that $\Lambda$ is contained in an open half-space, and therefore
\begin{equation}\label{eq.interior_disjoints_lambda}
\Lambda \cap(-\Lambda)=\vide.
\end{equation}

\begin{defi}[Curvature functions]
\label{definition:curvature-functions}
We say that a function $\Phi\in C^0(\overline{\Lambda})\cap C^\infty(\Lambda)$ is a \emph{curvature function} whenever
\begin{enumerate}[label=\emph{(\roman*)}]
\item $\Phi$ is $\mathrm{O}(2)$-invariant (which ensures that curvature is well-defined over immersed hypersurfaces);
\item $\Phi(\Id)=1$ (which ensures that the unit sphere has curvature $1$);
\item $\Phi$ vanishes on $\partial\Lambda$ (which ensures that the condition $\Phi(A)>0$ identifies with the convexity condition $A\in\Lambda$); and
\item $\Phi$ is $1$-homogeneous (which ensures that curvature scales correctly under dilatations).
\end{enumerate}
\end{defi}

\begin{remark}
Note that, by $\mathrm{O}(2)$-invariance, $\Phi$ is a function only of the two eigenvalues of $A$.
\end{remark}

Using the above scalar product \eqref{equation:trala}, the differential $\dd\Phi$ of $\Phi$ at any point $A\in\Lambda$ identifies with the symmetric bilinear form $B_A\in\Symm(\R^2)$ given by
\begin{equation}
\label{equation:ba}
\dd\Phi(A)C=\langle B_A,C\rangle = \tr(B_AC)\ .
\end{equation}

\begin{defi}[Elliptic curvature function]
We say that a curvature function $\Phi$ is \emph{elliptic} whenever, in addition
\begin{enumerate}[label=\emph{(\roman*)}]
\setcounter{enumi}{4}
\item $B_A=\dd \Phi(A)$ is positive-definite for all $A \in \Lambda$.
\end{enumerate}
\end{defi}

In the special case where $\Lambda=\Symm(\mathbb{R}^2)$, the only possible curvature function is
\begin{equation}
\label{eqn:LinearCase}
\Phi(A)=\frac12\tr(A)\ .
\end{equation}
Indeed, since $0\in\Lambda$, it follows by smoothness and homogeneity that $\Phi$ is linear. $\mathrm{O}(2)$-invariance and normalization then yield \eqref{eqn:LinearCase}. Note that, in this case $B_A = 1/2 \cdot \mathbf{1}$, and in particular, is independent of $A$.

Within this formalism, elliptic curvature problems are naturally associated with proper convex cones $\Lambda\subsetneq\Symm(\mathbb{R}^2)$ satisfying \eqref{eq.interior_disjoints_lambda}, with the unique exception of the case $(\Phi,k)=\big(\frac{1}{2}\tr,0\big)$. Note that, in the case $(\Phi,k)=(\frac{1}{2}\tr,k)$, with $k>0$, the natural ellipticity domain is the open half-space $\Lambda=\{\tr>0\}$.

\subsubsection{Normalization of curvature along positive directions}\label{sss.normalizaition} The following lemma allows us to reach a prescribed curvature level along any positive direction.

\begin{lemma}
\label{lemma:ChoiceOfC1}
Let $(\Phi,\Lambda)$ be an elliptic curvature function such that $\Lambda\subsetneq\Symm(\R^2)$. Let $k>0$ and $A\in \Lambda$ be such that $\Phi(A)>k$. For all $B\in \Lambda^+$, there exist $\lambda^\pm:=\lambda^\pm(A,B)\in\R$ such that
\begin{enumerate}[label=\emph{(\roman*)}]
\item $0<\lambda^+<\lambda^-$\ ;
\item $A-\lambda^+B \in \Lambda$ and $\lambda^-B-A \in \Lambda$\ ;
\item $\Phi(A-\lambda^+B)=\Phi(\lambda^-B-A)=k$\ .
\end{enumerate}
Moreover, $\lambda^\pm$ are unique and depend smoothly on $(A,B)$.
\end{lemma}

In the special case where $(\Phi,k)=(\frac{1}{2}\tr,0)$, the above construction degenerates in the sense that $\lambda^-$ and $\lambda^+$ coincide. The resulting analogue of Lemma \ref{lemma:ChoiceOfC1} is trivial, but worth stating independently.
\begin{lemma}\label{lemma:ChoiceOfC2}
Let $A\in\End(\mathbb{R}^2)$ be such that $\tr(A)>0$, and choose $B\in\Lambda^+$. There exists a unique $\lambda:=\lambda(A,B)$ such that $\lambda>0$ and $\tr(A-\lambda B)=0$. Furthermore, $\lambda$ depends smoothly on $(A,B)$.
\end{lemma}
\begin{proof}Indeed, $\lambda:=\tr(A)/\tr(B)$.
\end{proof}
\begin{proof}[Proof of Lemma \ref{lemma:ChoiceOfC1}.]
Define $f:\R\rightarrow\End(\mathbb{R}^2)$ by
\begin{equation*}
f(t) := A - Bt\ ,
\end{equation*}
and denote
\begin{equation*}
I^+ := f^{-1}(\Lambda)\qquad\text{and}\qquad I^-=f^{-1}(-\Lambda)\ .
\end{equation*}
Recall from \eqref{eq.interior_disjoints_lambda} that $\Lambda\cap (-\Lambda)=\vide$, and the sets $I^-$ and $I^+$ are thus open and disjoint. In addition, since $f$ is affine and $\Lambda$ and $-\Lambda$ are convex, $I_-$ and $I_+$ are both connected.

Since $B\in \Lambda^+$ and $A+\Lambda^+\subset \Lambda$, $f(t)=A-tB \in \Lambda$ for $t< 0$, and $I^+$ thus contains a neighbourhood of $-\infty$.
Similarly, for $t> 0$ large enough, the matrix $-f(t)=tB-A$ is positive definite, and hence lies in $\Lambda$,
and $I^-$ thus contains a neighbourhood of $+\infty$. Recall now that, for any two symmetric, positive definite matrices $S_1$, $S_2$,
\begin{equation*}
\tr(S_1S_2)>0\ .
\end{equation*}
Thus, by ellipticity, for all $t\in I^+$,
\begin{equation*}
\frac{d}{dt}\Phi(f(t)) = - \dd\Phi(f(t))B < 0\ ,
\end{equation*}
and, for all $t\in I_-$,
\begin{equation*}
\frac{d}{dt}\Phi(-f(t)) = \dd\Phi(-f(t))B > 0\ .
\end{equation*}
The function $t\mapsto\Phi(f(t))$ is thus strictly decreasing over $I^+$, whilst the function $t \mapsto \Phi(-f(t))$ is strictly increasing over $I^-$. Moreover, these functions extend continuously to the boundaries of their domains and vanish when $f(t)$ (resp. $-f(t)$) reaches $\partial \Lambda$. Finally, by homogeneity, they tend to $+\infty$ as $t \to \pm\infty$.

Existence of $\lambda_\pm$ now follows by the intermediate value theorem, uniqueness follows from strict monotonicity, and smooth dependence follows from the implicit function theorem.
\end{proof}

\subsubsection{Elliptic curvature functions in Riemannian $3$-manifolds}
\label{sss.curv_3mfds}

We now define elliptic curvature functions in this geometric setting. Note first that the bundle $\jetX$ can also be viewed as a fibre bundle over $SX$, whose fibre at $(x,n)$ is the space of self-adjoint endomorphisms of $n^\perp$.

\begin{defi}[Elliptic curvature function]
\label{definition:elliptic-curvature-function}
An \emph{elliptic curvature function} is a pair $(\Phi,\Lambda)$, where $\Lambda \subset \jetX$ is an open subset and $\Phi:\Lambda\to\R$ is a smooth function such that, on each fibre, $(\Phi,\Lambda)$ satisfies the conditions introduced above.
\end{defi}

\begin{remark} Note that, by $\mathrm{O}(2)$-invariance over each fibre, we may write any curvature function as a function $\Phi(p,n,\kappa_1,\kappa_2)$, where $\kappa_1$ and $\kappa_2$ denote the eigenvalues of $A$. We will see presently that $\kappa_1$ and $\kappa_2$ correspond to the principal curvatures of the surface.
\end{remark}

\subsection{$(\Phi,k)$-surfaces} \label{ssection:phiksurfaces}The notion of elliptic curvature functions on $3$-manifolds gives rise to the concept of $(\Phi,k)$-surfaces. This concept provides a general framework which includes minimal surfaces, constant mean curvature surfaces, and $k$-surfaces (that is, surfaces of constant extrinsic curvature equal to $k$) as special cases. Given a smooth immersion $e:D\rightarrow X$, we recall that its $2$-jet, denoted by $\jete(p)$, is defined in \S\ref{ssection:setting}.

\begin{defi}[$(\Phi,k)$-surfaces]
A smooth immersion $e:D\to X$ is said to be $\Lambda$-\emph{convex} if $\jete(p)\in \Lambda$ for all $p\in D$.
Given $k\in\R$, we say that $e$ is a \emph{$(\Phi,k)$-surface} if, in addition,
\begin{equation*}
\Phi(\jete(p)) = k \qquad \text{for all } p\in D\ .
\end{equation*}
If moreover $D$ is a disc, we call $e$ a \emph{$(\Phi,k)$-disc}.
\end{defi}

\subsubsection{The $\Phi$-Jacobi operator} Let $e:D\to X$ be a $(\Phi,k)$-surface, and let $n_e:D\to SX$ denote its unit normal vector field. Recall that $(X,\bar g)$ extends to a slightly larger open manifold $Y$ containing $X$. For $\varepsilon>0$ sufficiently small, consider the map
\begin{equation*}
E:D\times(-\varepsilon,\varepsilon)\to Y,\qquad (p,t)\mapsto \exp_{e(p)}\big(t\,n_e(p)\big)\ .
\end{equation*}
For $\varepsilon$ small enough, $E$ is an immersion. Indeed, it is precisely the \emph{Fermi parametrization} of some neighbourhood of $e(D)$ in $Y$. Given $u\in C^2(D)$ with $\|u\|_{L^\infty}<\varepsilon$, we define an immersion $e_u:D\to Y$ by
\begin{equation*}
e_u(p):=E\big(p,u(p)\big)\ .
\end{equation*}
In this manner, immersions sufficiently close to $e$ are represented as normal graphs over $e$. We recall that the geometry of graphs in Fermi coordinates has been studied in \S\ref{appendix:asymptotic}.

\begin{defi}[The $\Phi$-Jacobi operator]
The $\Phi$-\emph{Jacobi operator} $J^\Phi$ is defined as the linearization of $\Phi\circ\jete$ at $e$. That is for all $f\in C^2(D)$,
\begin{equation*}
J^\Phi f := \partial_t\Phi\circ\jete_{t f}|_{t=0}\ .
\end{equation*}
\end{defi}

We will see that this is a second-order elliptic differential operator with respect to the variation $f \in C^2(D)$. In order to describe it explicitly, we first introduce the operator
\begin{equation*}
W(v) := \overline{R}(n_e,v)n_e\ ,
\end{equation*}
where here we use the convention \eqref{equation:riemann} for the Riemann curvature tensor. Let $B(p)$ denote the fibrewise derivative of $\Phi$ in $\jetX$ at the point $\mathcal{J}^2e(p)$ as defined in \eqref{equation:ba}.

\begin{lemma}\label{lem.Jacobi_elliptic}
$J^\Phi$ is a second-order, elliptic, linear partial differential operator satisfying, for all $p\in D$,
\begin{equation}
\label{equation:jphi0}
J^\Phi f(p) = -\tr\big(B(p)\cdot\Hess_g(f)(p)\big) + Lf(p)\
\end{equation}
for some first-order, linear, partial differential operator $L$. Furthermore, when $\Phi$ only depends on (the conjugacy class of) $A$, for all $p\in D$,
\begin{equation}
\label{equation:jphi}
J^\Phi f(p) = \tr(B(p)\left(-\Hess_g(f)(p)+f(p)(W(p)-A_e^2(p)) \right)\ .
\end{equation}
\end{lemma}
\begin{remark} Curvature functions which depend only on the conjugacy class of $A$ are trivially independent both of the position in the manifold and of the normal vector. By $\mathrm{O}(2)$-invariance, such functions thus only depend on the principal curvatures $\kappa_1$ and $\kappa_2$ of the surface. The only possibilities are therefore the mean curvature, the extrinsic curvature, and combinations of these functions.
\end{remark}
\begin{proof}
It suffices to determine the second-order term. By \eqref{eq.graph_I_II_expansion}, for all $p\in D$,
\begin{equation}
\label{eqn:VariationOfShapeOperator}
A_{e_{tf}}(p) = A_e(p) -t(\Hess_g f(p))^\sharp + tf(p)(W(p)-A_e^2(p)) + \cO(t^2)\ ,
\end{equation}
where $\sharp$ denotes Berger's musical isomorphism with respect to $g$. Let $B(p)$ denote the fibrewise derivative of $\Phi$ in $\jetX$ at the point $\mathcal{J}^2e(p)$. Applying $\Phi$ to the previous equation, and differentiating with respect to $t$, we obtain \eqref{equation:jphi0}. Finally, since $B$ is positive-definite, $J^\Phi$ is elliptic and the result follows.
\end{proof}

\subsubsection{Mean curvature}

Mean curvature corresponds to the choice
\begin{equation*}
\Lambda := \jetX \qquad\text{and}\qquad \Phi(A):=\tr(A)/2\ .
\end{equation*}
This corresponds to the case of minimal ($k=0$) and constant mean curvature ($k>0$) surfaces. Here, the associated $\Phi$-Jacobi operator reduces to the classical Jacobi operator. Indeed, since $B(p)=1/2\cdot \mathbf{1}$ is a multiple of the identity, for any immersed surface $e:D\to X$ with unit normal $n_e$ and shape operator $A_e$, up to a factor of $2$, \eqref{equation:jphi} yields
\begin{equation*}
J^\Phi f = - \Delta f + \tr(W - A_e^2)\,f\ ,
\end{equation*}
where $\Delta f = \tr\big(\Hess(f)\big)$ denotes the Laplace operator of the induced metric on $D$.

\subsubsection{Extrinsic curvature}
In the case of extrinsic curvature, we take
\begin{equation*}
\Lambda := \big\{ (p,n,A)\in\jetX \mid A>0 \big\} \qquad\text{and}\qquad \Phi(A) := \det(A)^{\frac{1}{2}}\ .
\end{equation*}
A direct computation shows that the fibrewise derivative of $\Phi$ at $A$ is
\begin{equation*}
\dd_2\Phi(A) = \frac{1}{2}\det(A)^{\frac{1}{2}}A^{-1}\ ,
\end{equation*}
which is positive-definite so that, once again, $\Phi$ is elliptic. The corresponding $(\Phi,k)$-surfaces are precisely Labourie's $k$-surfaces. Using \eqref{equation:jphi}, up to a multiplicative factor, the associated Jacobi operator is given by (see \cite[Proposition 3.0.2]{Labourie-00})
\begin{equation*}
J^\Phi f = -\tr(A_e^{-1}\Hess(f)) + \tr(A_e^{-1}W - A_e)\,f\ .
\end{equation*}

\subsubsection{More examples of curvature functions}\label{sss.more_examples} We highlight several other notions of curvature for which our results apply.\\

{\emph{(i) The curvature quotient $K/H$.}}\ In this case
\begin{equation*}
\Lambda := \big\{ (p,n,A)\in\jetX \mid A>0 \big\}\qquad\text{and}\qquad \Phi(A) := 2\det(A)/\tr(A)\ .
\end{equation*}
Surfaces for which this curvature quotient is constant form a special class of \emph{Weingarten surfaces}. General curvature functions of this and related types were studied by Guan--Spruck in \cite{GuanSpruck2004Locally}, using G\aa rding's theory of hyperbolic polynomials (see \cite{HarveyLawson_2013,Garding_1959}), and building on the work \cite{Caffarelli-Nirenberg-Spruck-88} of Caffarelli--Nirenberg--Spruck. Note that $2H/K=\rho_1+\rho_2$, where $\rho_1$ and $\rho_2$ denote the principal radii of curvature. For closed convex surfaces in Euclidean space, the problem of prescribing this quantity as a function of the unit normal is known as \emph{Christoffel's problem}. It was posed by Christoffel \cite{Christoffel} in 1865, and was solved independently by Firey \cite{Firey} and Berg \cite{Berg} a century later. For analogues of this problem in hyperbolic space and in Minkowski space, we refer to \cite{EspinarGalvezMira} and \cite{FillastreVeronelli} respectively. \\

{\emph{(ii) Special Lagrangian curvature.}}\ Special Lagrangian curvature was introduced by the fourth author in \cite{Smith2013SpecialLagrangian} as a higher-dimensional curvature notion exhibiting the same geometric properties as extrinsic curvature (see also \cite{HarveyLawson2021Pseudoconvexity}). Here
\begin{equation*}
\Lambda := \big\{ (p,n,A)\in\jetX \mid A>0 \big\}\qquad\text{and}\qquad \Phi(A) := R_\theta(A)^{-1}\ ,
\end{equation*}
where $r:=R_\theta$ is defined by the equation
$$\theta=\sum_i\arctan(r\kappa_i)\in[\pi/2,\pi),$$
as in \cite{Smith2013SpecialLagrangian}. In two dimensions, special Lagrangian curvature interpolates between extrinsic curvature (when $\theta=\pi/2$), and the curvature quotient $2\det(A)/\tr(A)$ (in the limit as $\theta$ tends to $\pi$). \\

{\emph{(iii) Elliptic regularizations.}}\ Elliptic regularizations are perturbations of curvature functions designed to improve their regularity. The following elliptic regularization of extrinsic curvature was introduced by Gerhardt in \cite{Gerhardt2003Hypersurfaces} in order to study curvature flows. For $\lambda\in(0,\infty)$, we define
\begin{equation*}
\Lambda_\lambda := \{ (p,n,A)\in\jetX \mid  \det(A)+\lambda\tr(A)^2>0\ ,\ \tr(A)>0 \}\ ,
\end{equation*}
and
\begin{equation*}
\Phi_\lambda(A) = \frac{\big(\det(A) + \lambda\tr(A)^2\big)^{\frac{1}{2}}}{\sqrt{1+\lambda}}\ .
\end{equation*}
This family interpolates between extrinsic curvature (when $\lambda=0$) and mean curvature (in the limit as $\lambda\rightarrow\infty$).

\subsection{Elliptic operators and strict stability}
\label{ss:Stability}
We now introduce notions of stability and strict stability for Jacobi operators, which extend to higher dimensions the concept of absence of conjugate points. We formulate this condition in the general setting of second-order elliptic operators.

\subsubsection{The maximum principle} Let $\closedtwoball$ be the closed unit disc in $\R^2$, and let
\begin{equation}
\label{eqn:GeneralEllipticOperator}
L:=a^{ij}(x)\partial_i\partial_j + b^i(x)\partial_i + c(x)
\end{equation}
be a second-order, linear, elliptic partial differential operator over $\closedtwoball$ with smooth, real-valued coefficients. Our analysis will rely on the structure of the zero sets of the eigenfunctions of such operators.

The following version of the maximum principle will be used repeatedly in the sequel (see \cite[Lemma~3.4 and Theorem~3.5]{Gilbarg-Trudinger-01}).

\begin{theorem}[Maximum principle]
\label{thm:MaxPrinc}
Let $f\in C^0(\closedtwoball)\cap C^2(\twoball)$ satisfy $Lf=0$.
\begin{enumerate}[label=\emph{(\roman*)}]
\item If $c\leq 0$, then $f$ cannot achieve a non-negative maximum (or a non-positive minimum) in $\twoball$ unless it is constant.
\item If $c= 0$, then $f$ cannot achieve its maximum or minimum in $\twoball$ unless it is constant.
\item If $f\geq 0$ and $f(p)=0$ for some boundary point $p\in\mathbb{S}^1$, then $(\partial_\nu f)(p)<0$, where $\nu$ denotes the outward-pointing unit normal to $\mathbb{S}^1$ at $p$.
\end{enumerate}
\end{theorem}

\begin{remark}\label{rem_max_principle}
In what follows, we will use the following explicit reformulation of Theorem~\ref{thm:MaxPrinc} (i): if $c\leq 0$, if $\Omega\subset\closedtwoball$, and if $f\in C^0(\overline{\Omega})\cap C^2(\Omega)$ satisfies $Lf=0$, then $f|_{\partial\Omega}\geq 0$ (resp. $f|_{\partial\Omega}\leq 0$) implies $f\geq 0$ in $\Omega$ (resp. $f\leq 0$ in $\Omega$).
\end{remark}

\subsubsection{Strict stability}
We introduce a notion of strict stability for the operator $L$. We note first that, in surface theory, one often restricts attention to self-adjoint operators, or to operators arising as Euler--Lagrange equations of variational problems. Since, the present setting requires a more general class of operators, we will use the following definition.

\begin{defi}
\label{defn:Stability}
We say that the operator $L$ is \emph{strictly stable} whenever there exists a smooth, positive function $u:\closedtwoball\rightarrow\R$ such that $Lu=0$.
\end{defi}

This notion extends to surfaces the classical concept of stability for geodesics. Indeed, a geodesic segment in a Riemannian manifold is stable whenever no two of its points are conjugate. In other words, a geodesic segment in a surface is stable whenever it possesses a non-vanishing Jacobi field, that is, a non-vanishing solution of the Jacobi equation. Definition~\ref{defn:Stability} is precisely the $2$-dimensional analogue of this concept.

Strict stability can be established by the following simple criterion. Recall first that $\lambda\in\mathbb{C}$ is a \emph{Dirichlet eigenvalue} of $L$ whenever there exists a non-trivial $f\in C^0(\closedtwoball)\cap C^2(\twoball)$, such that
\begin{equation*}
\begin{cases}
Lf + \lambda f = 0 & \text{in } \twoball\ \text{and}\\
f = 0 & \text{on } \partial \twoball\ .
\end{cases}
\end{equation*}
The \emph{Dirichlet spectrum} of $L$, which we denote by $\opSpec_D(L)$, is the set of all of its Dirichlet eigenvalues. It is a discrete subset of $\C$ and, when $L$ is self-adjoint (with respect to, say, the Lebesgue measure), it is contained in $\mathbb{R}$ and accumulates only at $+\infty$ (see \cite[Section 8]{Gilbarg-Trudinger-01}).

\begin{lemma}
\label{lemma:CharacterizationOfStability}
\begin{enumerate}[label=\emph{(\roman*)}]
\item If $c\leq0$, then $L$ is strictly stable in the sense of Definition \ref{defn:Stability}.
\item If $L$ is self-adjoint, then it is strictly stable in the sense of Definition \ref{defn:Stability} if and only if its Dirichlet spectrum satisfies
\begin{equation*}
\opSpec_D(L) \subseteq (0,\infty)\ .
\end{equation*}
\end{enumerate}
\end{lemma}

\begin{remark}
In Lemma \ref{lemma:CharacterizationOfStability}, (ii), self-adjointness is only required to guarantee that the spectrum is real.
\end{remark}

In will be convenient here and in the sequel to introduce the differential operator $L'$, defined as follows. Let $u$ be as in Definition \ref{defn:Stability}, and let $M_u$ denote the operator of multiplication by this function. Since $u$ is strictly positive, $M_u$ is invertible, and we define
\begin{equation}
\label{eqn:DefinitionOfLPrime}
L' := M_u^{-1} L M_u\ .
\end{equation}
Since $Lu=0$,
\begin{equation*}
L'=a^{ij}\partial_i\partial_j+b^i\partial_i+a^{ij}u^{-1}(\partial_i u\partial_j+\partial_j u\partial_i)\ .
\end{equation*}
In particular, the zeroth-order coefficient of this operator vanishes.

\begin{proof}
(i) By Theorem \ref{thm:MaxPrinc}, Item (i), $L$ has trivial Dirichlet kernel, for if $Lu=0$ and $u|_{\partial\closedtwoball}=0$ then (see Remark \ref{rem_max_principle}) $u\leq 0$ and $u\geq 0$, so that $u=0$. By the Fredholm alternative and elliptic regularity, there therefore exists a unique, smooth function $u:\closedtwoball\rightarrow\Bbb{R}$ solving the Dirichlet problem for $L$ with boundary value $1$. By Theorem \ref{thm:MaxPrinc}, Item (i) again, this function is positive over $\closedtwoball$, and the claim follows. \\

(ii) Suppose first that $L$ is strictly stable and let $u > 0$ be such that $Lu = 0$. Let $f \in C^\infty(\closedtwoball)$ be such that $f = 0$ on $\partial\closedtwoball$ and $Lf = -\lambda f$ for some $\lambda$. We now show $\lambda>0$. Indeed, otherwise, $\lambda\leq 0$. However
\begin{equation*}
L'(f/u) = u^{-1} L (u \cdot f/u) = u^{-1} L f = -\lambda (f/u)\ ,
\end{equation*}
so that
\begin{equation*}
(L'+\lambda)(f/u) = 0\ .
\end{equation*}
Since the zero'th order coefficient of $L'+\lambda$ is $\lambda\leq 0$, and since $f/u$ vanishes over $\partial\closedtwoball$, it follows by the maximum principle as before that $f/u=0$, and therefore $f=0$. This is absurd, and it follows that $\lambda>0$, as desired.

We now prove the converse. Since $0 \notin \opSpec_D(L)$, the Fredholm alternative \cite[Theorem 6.15]{Gilbarg-Trudinger-01} yields a smooth function $u$ satisfying $Lu = 0$ and $u = 1$ on $\partial \twoball$. We first show that $u\geq 0$. Indeed, suppose the contrary, and let $\Omega$ be a connected component of $\{u<0\}$. Since $Lu=0$ over $\Omega$ and $u=0$ over $\partial\Omega$, it follows that $0 \in \opSpec_D(\Omega, L)$, so that $\lambda_1(\Omega, L) \leq 0$. Since the least element of the Dirichlet spectrum is strictly monotone with respect to inclusion of sets, it follows that
\begin{equation*}
\lambda_1(\twoball,L)<\lambda_1(\Omega,L)\leq 0\ .
\end{equation*}
This is absurd, and it follows that $u\geq 0$, as asserted. Finally, by Harnack's inequality (see, for example, \cite[Theorem 8.20]{Gilbarg-Trudinger-01}), $u > 0$, and this completes the proof.
\end{proof}

We will apply strict stability in the sequel via the following version of the maximum principle.
\begin{lemma}
\label{lemma:WeakMP}
If $L$ is strictly stable in the sense of Definition \ref{defn:Stability}, then there exists $C>0$ such that for every open set $\Omega \subseteq \closedtwoball$, and for all $f \in C^0(\overline{\Omega}) \cap C^\infty(\Omega)$ such that $Lf = 0$ and $f|_{\partial\Omega} \geq 0$,
\begin{equation}
\label{eqn:WeakMP}
\frac{1}{C}\inf_{x\in\partial\Omega} f(x) \leq \inf_{x\in\overline{\Omega}} f(x)\qquad\text{and}\qquad
\sup_{x\in\overline{\Omega}} f(x) \leq C\sup_{x\in\partial\Omega} f(x)\ .
\end{equation}
In particular, $L$ has trivial Dirichlet kernel.
\end{lemma}

\begin{proof}
Indeed, let $u$ be as in Definition \ref{defn:Stability}, and note that
\begin{equation*}
L'(M_u^{-1}f) = M_u^{-1}LM_uM_u^{-1}f = M_u^{-1}Lf = 0\ .
\end{equation*}
Thus, by Theorem \ref{thm:MaxPrinc}, Item (ii), applied to $L'$,
\begin{equation*}
\inf_{x\in\overline{\Omega}}(f/u)(x) = \inf_{x\in\partial\Omega}(f/u)(x)\qquad\text{and}\qquad
\sup_{x\in\overline{\Omega}}(f/u)(x) = \sup_{x\in\partial\Omega}(f/u)(x)\ .
\end{equation*}
Since $u$ is smooth and strictly positive on $\overline{\Omega}$, this implies \eqref{eqn:WeakMP}. In particular, every element of the Dirichlet kernel of $L$ vanishes over $\partial\Omega$, and thus vanishes identically over $\Omega$. $L$ then has trivial Dirichlet kernel, and this completes the proof.
\end{proof}

We will also make use of the analogous result over the boundary.
\begin{lemma}
\label{lemma:BoundaryWeakMP}
If $L$ is strictly stable in the sense of Definition \ref{defn:Stability}, then for every non-constant, non-negative function $f \in C^1(\closedtwoball)\cap C^2(\twoball)$ satisfying $Lf = 0$, and for every $p \in \mathbb{S}^1$, one has
\begin{equation*}
f(p) = 0\ \Rightarrow\ (\partial_\nu f)(p) < 0\ ,
\end{equation*}
where $\partial_\nu$ denotes the derivative in the outward-pointing, unit normal direction of $\mathbb{S}^1$.
\end{lemma}
\begin{proof}
Using the same construction as in the proof of Lemma \ref{lemma:WeakMP}, Theorem \ref{thm:MaxPrinc} (iii) yields
\begin{equation*}
\partial_\nu(f/u)(p) < 0\ .
\end{equation*}
and the result then follows by the quotient rule.
\end{proof}

\subsection{Plateau problems}  \label{ssection:plateau-problems}

We now introduce the Plateau problem for $(\Phi,k)$-surfaces, where $(\Phi,\Lambda)$ is an elliptic curvature function and $k \geq 0$. This defines the class of boundary value problems that will be considered in the sequel.
\begin{defi}[$(\Phi,k)$-Plateau problems] Let $c\subseteq\partial X$ be a smooth, oriented, simple, closed curve, and choose $k>0$.
\label{defn:StableSolutions}
\begin{enumerate}[label=\emph{(\roman*)}]
\item A solution to the $(\Phi,k)$-\emph{Plateau problem} with boundary $c$ is a smooth, oriented, embedded $(\Phi,k)$-disc $D \subset X$ with oriented boundary $c$.
\item A solution $D$ is said to be \emph{strictly stable} whenever its $\Phi$-Jacobi operator is strictly stable in the sense of Definition \ref{defn:Stability}.
\end{enumerate}

In the case where $\Phi=(1/2)\tr$, this definition also applies when $k=0$.
\end{defi}

When $\Phi=\tr$, $\Lambda=\jetX$, and $k=0$, this corresponds to the classical Plateau problem for minimal surfaces (see, for example, \cite{Colding-Minicozzi-11}). Hypotheses \hyperlink{AA1}{\rm(A1-A4)} provide explicit conditions on the metric $\bar g$ and the boundary curve $c$ for existence, uniqueness, and strict stability of solutions to hold. When $k>0$, it corresponds to the Plateau problem for constant mean curvature surfaces, see \cite{Lopez-13} for instance.

When $\Phi=\det^{\frac{1}{2}}$ and $\Lambda=\{(x,n,A)\in\jetX : A>0\}$, the corresponding $(\Phi,k)$-surfaces are Labourie's $k$-surfaces. The compact Plateau problem in this setting was studied in \cite{Smith-20}, where existence, uniqueness, and stability of solutions were established.

When $\Phi$ is either the curvature quotient or the special Lagrangian curvature (see \S \ref{sss.more_examples}), the $(\Phi,k)$-Plateau problem also admits a unique strictly stable solution, as proven in \cite{Smith-20}. These curvature functions, as well as extrinsic curvature, give rise to elliptic Monge--Ampère equations within our general framework. We have not studied under which circumstances uniqueness and stability hold for the class of curvature functions obtained by elliptic regularization.

In what follows, we restrict to the three-dimensional space of boundary conditions given by oriented round circles $c \in \cC$ and study the associated ``phase space'' of the Plateau problem in the sense of \cite{Labourie-02}. More precisely, we will show that the space of pointed $(\Phi,k)$-discs spanned by round circles naturally identifies with $SX^\bullet$, provided that for every oriented round circle the $(\Phi,k)$-Plateau problem admits a unique solution which, in addition, is strictly stable (see Theorem \ref{th.Foliated_Plateau_posta}).

\section{The asymptotic analysis of small $(\Phi,k)$-discs}
\label{section:asymptotic}

\subsection{Small $(\Phi,k)$-surfaces}
\label{s.asymptotic_Analysis}
We now begin our study of the foliated $(\Phi,k)$-Plateau problem. In this section, we use perturbation theory to solve the Plateau problem for small circles. Our main objective will be to show that the resulting foliation extends smoothly (even analytically in the case of analytic data) across the boundary. This property, which will be shown using a careful asymptotic analysis of the solutions constructed, will allow us in \S\ref{s.Radon} to establish boundary ellipticity of the normal operator of the Radon transform.

\subsubsection{The main result} Let $\bar g$ be a Riemannian metric on $X$. We identify $X$ with the closed $3$-ball via some smooth diffeomorphism $\alpha:\closedthreeball\to X$, and we extend $\bar g$ beyond $\partial X=\mathbb{S}^2$ to some slightly larger ball. Let $\nu$ denote the outward-pointing unit normal vector field over $\partial X$, and let $A$ denote the corresponding shape operator.

For the statement of the following theorem, it will be convenient to suppose that $\mathcal{C}$ is the family of standard circles over $\partial X=\mathbb{S}^2$. Note that this always holds up to suitably modifying the diffeomorphism $\alpha$. We denote by $d_\sph$ the spherical distance over $\mathbb{S}^2$ and, for all $p\in\mathbb{S}^2$ and for all $r\in(0,\pi)$, we denote by $c^+_r(p)$ (resp. $c^-_r(p)$) the positively (resp. negatively) oriented circle of spherical radius $r$ about $p$ in $\mathbb{S}^2$.

\begin{theorem}
\label{thm:SmallCircles}
Let $(\Phi,\Lambda)$ be an elliptic curvature function, and choose $k>0$. Suppose that $\partial X$ is $\Lambda$-convex and that, for all $x\in\partial X$,
\begin{equation}
\label{eq_convexity_condition}
\Phi\big(x,\nu(x),A(x)\big) > k\ .
\end{equation}
For sufficiently small $s>0$, there exists a neighbourhood $U$ of the diagonal in $\partial X\times\partial X$ and smooth functions $\mathbf{f}^\pm:(-s_0,s_0)\times U\rightarrow\R$ such that, for all $(p,r)\in\partial X\times(0,\sqrt{s_0})$, the graph
\begin{equation}
\mathrm{D}_{p,r}^\pm :=\big\{ \exp_q\big(\mathbf{f}^\pm(r^2,p,q)\cdot\nu(q)\big)\ :\ d(p,q)\leq r\big\}\ ,
\end{equation}
oriented by its normal in the direction of $\pm\nu$, solves the $(\Phi,k)$-Plateau problem with boundary curve $c^\pm_r$. Furthermore,
\begin{enumerate}[label=(\roman*)]
\item for all $p$ and for all $r$, the disk $\mathrm{D}_{p,r}^\pm$ is strictly stable in the sense of Definition \ref{defn:StableSolutions};
\item for small $s$,
\begin{equation}
\label{equation:graal}
\mathbf{f}^\pm(s,p,q) = \frac{1}{2}\lambda^\pm(p)\big(d_\sph(p,q)^2 - s\big) + \opO\big(s^2 + sd(p,q)\big)\ ,
\end{equation}
where $\lambda^\pm:\partial X\rightarrow\R$ are smooth functions (given by \eqref{eqn:DefinitionOfLambdaForAsymptoticAnalysis}, below) which depend only on $(\Phi,k)$ and $\partial X$; and
\item when the metric $\bar g$ and the diffeomorphism $\alpha$ are analytic, $\mathbf{f}^\pm$ is also analytic.
\end{enumerate}
Finally, in the case where $\Phi(A)=\tr(A)/2$, the result also holds for $k=0$, and $\mathbf{f}^+=\mathbf{f}^-$ in this case.
\end{theorem}

The following corollary will be of use in establishing the global diffeomorphism property in \S\ref{section:plateau}.
\begin{corollary}
\label{cor:LimitsOfSmallCircles}
With the same hypotheses as in Theorem \ref{thm:SmallCircles}, for all $p$, and for all sufficiently small $r$, let $\hat{\mathrm{D}}^\pm_{p,r}$ denote the Gauss lift of $\mathrm{D}^\pm_{p,r}$. If $(p_m)_{m\in\mathbb{N}}$ tends to $p_\infty\in\partial X$, and if $(r_m)_{m\in\mathbb{N}}$ tends to $0$, then $\big(\mathrm{D}^\pm_{p,r}\big)_{m\in\mathbb{N}}$ Hausdorff converges towards the singleton $\{\pm\nu(p_\infty)\}$.
\end{corollary}

\subsubsection{The Fermi chart} \label{sssection:fermi}
We begin by providing a local model for neighborhoods of boundary points. We choose two relatively compact open sets $\Omega'\dans\Omega\dans\R^2$ parametrizing open subsets of $\partial X$ in such a way that there exists $r_0>0$ satisfying the following properties.
\begin{enumerate}[label=(\roman*)]
\item The Euclidean $r_0$-neighborhood of $\overline{\Omega}'$ is contained in $\Omega$; and
\item for every $p\in\overline{\Omega}'$ and every $0<r<r_0$, the Euclidean circle centered at $p$ with radius $r$ is mapped onto a round circle $c\in\cC$ in $\partial X$.
\end{enumerate}
These open subsets can be obtained by restricting to a small open subset of $\partial X$ the composition of the map $\alpha^{-1}|_{\partial X}:\partial X\to\mathbb{S}^2$ with the stereographic projection. We henceforth work with \emph{Fermi coordinates} along the boundary. This is a parametrization of a neighbourhood of the boundary by $\Omega\times (-\eps,0]$ such that the vertical segments in $\Omega\times (-\varepsilon,0)$ correspond to unit-speed geodesics orthogonal to $\partial X$, and such that the vector field $\partial_t$ has constant unit length. In addition, $X$ can be extended to a slightly larger domain by adding to it (near $\Omega$) the set $\Omega \times [0,\eps)$.

\subsubsection{The asymptotics of the Plateau problem}Our first step towards proving Theorem \ref{thm:SmallCircles} consists in establishing the existence and low order regularity of the functions $\mathbf{f}^\pm$. Higher order regularity will then follow from a careful asymptotic analysis which we will carry out in \S\ref{ssection:smooth-coordinates}. Throughout what follows, we will make considerable use of the notation discussed in Appendix \ref{ss:Notation}.

For $p\in\Omega'$ and $r>0$, let $B_r(p)$ denote the Euclidean ball of radius $r$ centered at $p$. Here and in all that follows, by mild abuse of notation, for all $r$ and $p$, we denote by $\opc^\pm_r(p)$ the Euclidean circle of radius $r$ centered at $p$, endowed with the orientation induced by $\pm\partial_t$. The rest of this section will be devoted to proving the following result.

\begin{theorem}
\label{thm:AsymptoticPlateauProblem}
Let $(\Phi,\Lambda)$ be an elliptic curvature function, choose $k>0$, and suppose that the convexity condition \eqref{eq_convexity_condition} holds. Upon reducing $r_0$ if necessary, there exist smooth functions
\begin{equation}
\label{eqn:DefinitionOfLambdaForAsymptoticAnalysis}
\lambda^\pm:\Omega'\to [0,\infty), \qquad
\phi^\pm:\Omega'\times[0,r_0)\times \closedtwoball \to \R
\end{equation}
such that
\begin{equation*}
0<\lambda_+<\lambda_-\qquad\phi^\pm(p,0,\cdot)=0\qquad\text{and}\qquad\phi^\pm(p,r,\cdot)|_{\SS^1}=0\ ,
\end{equation*}
and such that the following holds. For every $(p,r)\in \Omega'\times(0,r_0)$, if
\begin{equation}
\label{eqn:MainAsymptoticResultFormulaForF}
f^\pm_{p,r,\phi}(q) := \frac{1}{2}\lambda_\pm(p)\big(\|q-p\|^2-r^2\big)+r^2 \phi^\pm\left(p,r,\frac{q-p}{r}\right), \qquad q\in B_r(p)\ ,
\end{equation}
then the graph
\begin{equation*}
\mathrm{D}_{p,r}^\pm:=\big\{\big(q, f_{p,r,\pm}(q)\big) ~:~q\in B_r(p)\big\}
\end{equation*}
of $f^\pm_{p,r,\phi}$ over $B_r(p)$, oriented by its unit normal in the direction of $\pm \partial_t$, is a solution of the $(\Phi,k)$-Plateau problem in $\Omega\times(-\varepsilon,0]$ with oriented boundary $\opc^\pm_r(p)$. Furthermore, these solutions are strictly stable in the sense of Definition~\ref{defn:StableSolutions}.

When the metric $\bar g$ is analytic, the functions $\lambda^\pm$ and $\phi^\pm:\Omega'\times[0,r_0)\times \closedtwoball \to \R$ are also analytic.

Finally, in the case where $\Phi=\tr/2$, $\lambda^+=\lambda^-$, and the result also holds for $k=0$ and $\phi^+=\phi^-$ in this case.
\end{theorem}

At leading order, the solutions are given explicitly by paraboloids determined by the functions $\lambda_\pm$. These in turn depend only on the extrinsic geometry of $\partial X$ and are characterized in Lemmas~\ref{lemma:ChoiceOfC1} and~\ref{lemma:ChoiceOfC2}. The higher-order terms, encoded in $\phi^\pm$, capture the finer dependence on the family of boundary circles.

\subsubsection{The rescaled family} Consider a smooth function $f:\Omega\rightarrow(-\eps,\eps)$, and let $\hat{f}(q):=\big(q,f(q)\big)$ denote the natural parametrization of its graph. Recall from Lemma \ref{lemma:FormulaeForIAndII} that, in the Fermi coordinates introduced above, the unit normal vector field, first and second fundamental form, and shape operator of $\hat{f}$ satisfy
\begin{equation}\label{eq.graph_expansions_minimal}
\begin{aligned}
n_{\hat f} &=\partial_t-\nabla^g f+\cO\big((J^1\phi)^2\big)\ ,\\
\I_{\hat f} & = g+2 f\,\II_0 + \cO\big((J^1\phi)^2\big)\ ,\\
\II_{\hat f} &= \II_0-\Hess^g\,f + f(\mathcal R_{0,\nu}+\III_0)+ \cO\big((J^1 f)\cdot(J^2 f)\big)\ ,\ \text{and}\\
A_{\hat f} &  = A_0 -(\Hess^g f)^\sharp + f(\mathcal R_{0,\nu}^\sharp-A_0^2) + \cO\big((J^1\phi)\cdot(J^2\phi)\big)\ .
\end{aligned}
\end{equation}
Here, $\nabla^g$ and $\Hess^g$ denote respectively the gradient and hessian operators of the restriction of $g$ to $\Omega=\{t=0\}$ and $\sharp$ denotes Berger's musical isomorphism with respect to this metric. Likewise, $\II_0$, $\III_0$, and $A_0$ denote respectively the second and third fundamental forms and shape operator of $\Omega=\{t=0\}$, and $\mathcal{R}_{0,\nu}$ denotes the normal curvature tensor of this surface. We recall that the remainder terms in \eqref{eq.graph_expansions_minimal} are expressed using the notation of Apppendix \ref{ss:Notation}

We emphasize that the above formulae are statements about the smooth dependence of these geometric objects on their arguments. More precisely, recalling the jet spaces introduced in \S\ref{sss:graphs}, Equations \eqref{eq.graph_expansions_minimal} state two things. Firstly, they state that $n_{\hat f}$ and $\I_{\hat f}$ are smooth functions of $\big(p,q,J^1f(q)\big)$, whilst $\II_{\hat f}$ and $A_{\hat f}$ are smooth functions of $\big(p,q,J^2f(q)\big)$. Secondly, they explicitly describe the lowest order terms of the Taylor expansions of these functions.

Let $G(p)$ denote the matrix of $g(p)$ with respect to the Euclidian metric over $\Omega$, that is, for all $p\in\Omega$, and for all vectors $u,v\in\R^2$,
\begin{equation*}
g(p)(u,v) := \langle G(p)u,v\rangle\ .
\end{equation*}
By hypothesis, the shape operator $A_0$ of the boundary satisfies, for all $p\in\Omega$,
\begin{equation*}
\Phi\big(p,\partial_t,A_0(p)\big)>k\ .
\end{equation*}
Hence by Lemma~\ref{lemma:ChoiceOfC1}, for all $p\in\Omega$ there exist unique constants $0<\lambda^+(p)<\lambda^-(p)$, depending smoothly on $p$, such that
\begin{equation}
\label{eqn:ChoiceOfC}
\Phi\bigl(p,\partial_t,A^\pm(p)\bigr)=k\ ,
\end{equation}
where
\begin{equation}
\label{eqn:DefnOfTildeA}
A^+(p) := A_0(p)-\lambda^+(p)G(p)^{-1}\quad\text{and}\quad A^-(p) := \lambda^-(p)G(p)^{-1}-A_0(p)\ .
\end{equation}
In the particular case where $\Phi=\frac{\tr}{2}$ and $k=0$,
\begin{equation}
\label{eqn:DefnLambda}
\lambda_+(p)=\lambda_-(p)=\frac{\tr(A(p))}{\tr(G(p)^{-1})}\ .
\end{equation}

For any fixed $(p,r)\in\Omega\times(0,r_0)$, consider now the reparametrization,
\begin{equation}
\label{eqn:ReparametrizationQX}
q := q(x) := p + rx\qquad x := x(q) := \frac{q-p}{r}\ .
\end{equation}
Given $\phi:\closedtwoball\rightarrow\R$, we consider the function
\begin{equation}
\label{eqn:ReparametrizedPhi}
\begin{aligned}
f^\pm_{p,r,\phi}(q) &:= \frac{1}{2}\lambda^\pm(p)r^2\big(\|x(q)\|^2-1\big) + r^2\phi\big(x(q)\big)\\
&=\frac{1}{2}\lambda^\pm(p)\big(\|q-p\|^2 - r^2\big) + r^2\phi\big((q-p)/r\big)\ .
\end{aligned}
\end{equation}
Trivally, denoting $f:=f^\pm_{p,r,\phi}$, for all $(i,j)$, the first and second derivatives of $f$ satisfy
\begin{equation}
\label{eqn:AsymptoticsOfF}
\begin{aligned}
(\dd_q f)\big(q(x)\big)\big(\partial_{q_i}\big) &= \lambda^\pm(q)\langle q-p,\partial_{q_i}\rangle + r(\dd_x\phi)(x)\big(\partial_{x_i}\big)\ \text{and}\\
(\dd_q^2 f)\big(q(x)\big)\big(\partial_{q_i},\partial_{q_j}\big)
&= \lambda^\pm(p)\delta_{ij} + (\dd^2_x\phi)(x)\big(\partial_{x_i},\partial_{x_j}\big)\ .
\end{aligned}
\end{equation}
We consider the parametrization $e^\pm_{p,r,\phi}:\closedtwoball\rightarrow\Omega\times(-\eps,\eps)$ of the graph of $f^\pm_{p,r,\phi}$ given by
\begin{equation}
\label{eqn:Reparametrization}
\begin{aligned}
e^\pm_{p,r,\phi}(x) &:= \hat{f}^\pm_{p,r,\phi}\big(q(x)\big)\\
&=\bigg(p+rx,\pm\frac{1}{2}\lambda^\pm(p)r^2\big(\|x\|^2-1\big) + r^2\phi(x)\bigg)\ ,
\end{aligned}
\end{equation}
and we note that the boundary condition $\phi|_{\mathbb{S}^1}=0$ corresponds to the geometric condition that the image of $e^\pm_{p,r,\phi}$ is bounded by $c_r^\pm(p)$.

We will henceforth restrict our attention to the study of $f^+_{p,r,\phi}$ and $e^+_{p,r,\phi}$, as the other case is addressed similarly, and we will therefore suppress the superscript ``$\pm$'' in all that follows. Let $n_{p,r,\phi}$ and $A_{p,r,\phi}$ denote respectively the unit normal vector field and shape operator of the immersion $e_{p,r,\phi}$. We view these as functions over $\closedtwoball$ taking values in $\R^3$ and $\opEnd(\R^2)$ respectively. We first ascertain how these functions depend on $\phi$.

\begin{lemma}
\label{lemma:RescaledNormal}
There exists a smooth function $\mathbf{n}:\Omega\times[0,r_0)\times\closedtwoball\times\mathcal{J}^1\closedtwoball\rightarrow\R^3$ such that, for all $p$, $r$, $\phi$, and $x$,
\begin{equation}
\label{eqn:RescaledNormalI}
n_{p,r,\phi}(x) = \mathbf{n}\big(p,r,x,J^1\phi(x)\big)\ .
\end{equation}
Furthermore
\begin{equation}
\label{eqn:RescaledNormalII}
\mathbf{n}\big(p,r,x,J^1\phi(x)\big) = \partial_t + \cO(r)\ .
\end{equation}
\end{lemma}
\begin{remark} We emphasize here that $\mathbf{n}$ is smooth and well-defined even at $r=0$, and that \eqref{eqn:RescaledNormalII} is a statement about its Taylor expansion in $r$ at this point.
\end{remark}
\begin{proof} Denoting $f:=f_{p,r,\phi}$, \eqref{eqn:ReparametrizedPhi} and \eqref{eqn:AsymptoticsOfF} yield
\begin{equation*}
f\big(q(x)\big) = \cO(r^2)\qquad\text{and}\qquad (d_qf)\big(q(x)\big) = \cO(r)\ .
\end{equation*}
The result now follows immediately from \eqref{eq.graph_expansions_minimal} and the subsequent discussion.
\end{proof}

\begin{lemma}
\label{lemma.rescaled_shape_minimal}
There exists a smooth function $\mathbf{A}:\Omega\times[0,r^0)\times\closedtwoball\times\mathcal{J}^2\R^2\rightarrow\opEnd(\R^2)$ such that, for all $p$, $r$, $\phi$, and $x$,
\begin{equation}
A_{p,r,\phi}(x) = \mathbf{A}\big(p,r,x,J^2\phi(x)\big)\ .
\end{equation}
Furthermore
\begin{equation}\label{eq.rescaled_normal_shape_minimal}
\begin{aligned}
\mathbf{A}\big(p,r,x,J^2\phi(x)\big) = A_0(p)-\lambda(p)\cdot G(p)^{-1}- G(p)^{-1}(d_x^2\phi)(x)+\cO(r)\ .
\end{aligned}
\end{equation}
\end{lemma}
\begin{remark} We emphasize again that $\mathbf{A}$ is smooth and well-defined even at $r=0$, and that \eqref{eq.rescaled_normal_shape_minimal} is a statement about its Taylor expansion in $r$ at this point.
\end{remark}
\begin{proof} Denoting $f:=f_{p,r,\phi}$, \eqref{eqn:ReparametrizedPhi} and \eqref{eqn:AsymptoticsOfF} yield
\begin{equation}
\label{eqn:ProofOfRescaledShapeI}
\begin{aligned}
f\big(q(x)\big) &= \cO(r^2)\ ,\\
(d_qf)\big(q(x)\big) &= \cO(r)\ ,\ \text{and}\\
(d^2_qf)\big(q(x)\big)\big(\partial_{q_i},\partial_{q_j}\big) &= \lambda^\pm(p)\delta_{ij} + (d_x^2\phi)(x)\big(\partial_{x_i},\partial_{x_j}\big)\ .
\end{aligned}
\end{equation}
In particular, $(d^2_qf)\big(q(x)\big)=\cO(1)$. Since the reparametrization $q$ is given by translation and rescaling by a constant factor, the shape operator of the immersion $e$ is given by
\begin{equation*}
A_e(x) = q^*\big(A_{\hat{f}}\big(q(x)\big)\big) = A_{\hat{f}}\big(q(x)\big)\ .
\end{equation*}
It now follows by \eqref{eq.graph_expansions_minimal} and the subsequent discussion that the function $\mathbf{A}$ given above is indeed smooth even at $r=0$. Furthermore,
\begin{equation}
\label{eqn:ProofOfRescaledShapeII}
\begin{aligned}
A_e(x) &= A_0\big(q(x)\big) - \Hess^g(f)\big(q(x)\big)^\sharp + \cO(r)\\
&= A_0\big(q(x)\big) - G\big(q(x)\big)^{-1}\cdot\Hess^g(f)\big(q(x)\big) + \cO(r)\ .
\end{aligned}
\end{equation}
By \eqref{eqn:ReparametrizationQX},
\begin{equation*}
A_0\big(q(x)\big) = A_0(p) + \cO(r)\qquad\text{and}\qquad G\big(q(x)\big)= G(p) + \cO(r)\ .
\end{equation*}
Denoting the Christoffel symbol of $g$ by $\Gamma$, the hessian of $f$ with respect to this metric is given by
\begin{equation*}
\Hess^g(f)\big(q(x)\big)\big(\partial_{q_i},\partial_{q_j}\big)
=(d_q^2f)\big(q(x)\big)\big(\partial_{q_i},\partial_{q_j}\big) - \sum_k\Gamma^k_{ij}\big(q(x)\big)\cdot (d_qf)\big(q(x)\big)\big(\partial_{q_k}\big)\ .
\end{equation*}
Hence, by \eqref{eqn:ProofOfRescaledShapeI},
\begin{equation*}
\Hess^g(f)\big(q(x)\big)\big(\partial_{q_i},\partial_{q_j}\big)
=\lambda^\pm(p)\delta_{ij} + (d^2_x\phi)(x)\big(\partial_{x_i},\partial_{x_j}\big)+\cO(r)\ ,
\end{equation*}
The result follows upon substituting these identities into \eqref{eqn:ProofOfRescaledShapeII}.
\end{proof}

We now define $H_{p,r,\phi}(x):\closedtwoball\rightarrow\R$ by
\begin{equation}
\label{eqn:DefinitionOfH}
H_{p,r,\phi}(x) := \Phi\big(e_{p,r,\phi}(x),n_{p,r,\phi}(x),A_{p,r,\phi}(x)\big) - k\ .
\end{equation}
The image of $e_{p,r,\phi}$ is thus a $(\Phi,k)$-disk if and only if $H_{p,r,\phi}$ vanishes identically.
\begin{lemma}
\label{lemma:smooth-functional}
There exists a smooth function $\mathbf{H}:\Omega\times[0,r_0)\times\closedtwoball\times\mathcal{J}^2\R^2\rightarrow\R$ such that, for all $p$, $r$, $\phi$, and $x$,
\begin{equation}
H_{p,r,\phi}(x) = \mathbf{H}\big(p,r,x,J^2\phi(x)\big)\ .
\end{equation}
Furthermore
\begin{equation}\label{eq:smooth-functionalII}
\begin{aligned}
\mathbf{H}\big(p,r,x,J^2\phi(x)\big) = \cO(r)\ .
\end{aligned}
\end{equation}
\end{lemma}
\begin{proof} It follows immediately from Lemmas \ref{lemma:RescaledNormal} and \ref{lemma.rescaled_shape_minimal} and the smoothness of $\Phi$ that $\mathbf{H}$ is smooth. By the definition \eqref{eqn:DefnLambda} of $\lambda$, together with \eqref{eqn:RescaledNormalII} and \eqref{eq.rescaled_normal_shape_minimal}, $\mathbf{H}(0,0,0,\mathbf{0})=0$, from which \eqref{eq:smooth-functionalII} follows. This completes the proof.
\end{proof}

\begin{lemma}
\label{lemma:smooth-functionalII}
Let $\mathbf{H}$ be as in Lemma \ref{lemma:smooth-functional}. $\mathbf{H}$ extends to a smooth function over $\Omega\times(-r_0,r_0)\times\R^2\times\mathcal{J}^2\R^2$ such that
\begin{equation}
\label{equation:inversion}
\mathbf{H}\big(p,-r,-x,(J^2I^*\phi)(-x)\big) = \mathbf{H}\big(p,r,x,J^2\phi(x)\big)\ ,
\end{equation}
where, for any $\phi$,
\begin{equation*}
(I^*\phi)(x) = \phi(-x)\ .
\end{equation*}
\end{lemma}
\begin{proof} Smooth extendibility is trivial. In order to study parity, denote $i(x):=-x$ and observe that
\begin{equation*}
e_{p,-r,I^*\phi}\big(i(x)\big) = e_{p,-r,I^*\phi}(-x) = e_{p,r,\phi}(x)\ .
\end{equation*}
Since curvature is invariant under reparametrization, it follows that
\begin{equation*}
\mathbf{H}\big(p,-r,-x,(J^2I^*\phi)(-x)\big) = \mathbf{H}\big(p,-r,i(x),(J^2I^*\phi)(i(x))\big) = \mathbf{H}\big(p,r,x,J^2\phi(x)\big)\ ,
\end{equation*}
as desired.
\end{proof}

\subsubsection{Proof of Theorem \ref{thm:AsymptoticPlateauProblem}}
We continue to work with the case of $f^+_{p,r,\phi}$ and $e^+_{p,r,\phi}$, and hence suppress the superscript ``$+$'' in all that follows. We define the non-linear operator $\mathcal{H}:\Omega'\times[0,r_0)\times C^{2,\alpha}_0(\closedtwoball)\rightarrow C^{0,\alpha}(\closedtwoball)$ by
\begin{equation*}
\mathcal{H}(p,r,\phi)(x) := \mathbf{H}\big(p,r,x,J^2_x\phi(x)\big)\ .
\end{equation*}
By Lemma \ref{lemma:smooth-functionalII}, $\mathcal{H}$ is obtained as a finite combination of differentiations and composition by smooth functions. It thus defines a smooth functional between Banach spaces (see, for example, \cite{delaLlaveObaya1999,BourdaudLanzaDeCristoforis2002}, or \cite[Appendix A]{Smith2024AsymptoticGeometry}). Furthermore, by \eqref{eq:smooth-functionalII}, for all $p$ and for all $x$,
\begin{equation*}
\mathcal{H}(p,0,0)(x) = 0\ .
\end{equation*}
Theorem \ref{thm:AsymptoticPlateauProblem} will follow upon applying the implicit function theorem to this functional.

\begin{proof}[Proof of Theorem~\ref{thm:AsymptoticPlateauProblem}]
By ellipticity of the curvature function $\Phi$, the differential $\dd_3\Phi\big(p,\partial_t,A(p)\big)$ is given by a symmetric, positive-definite matrix $B^+(p)$, say. That is, for every matrix $M$, self-adjoint with respect to $g$ (see \eqref{equation:ba}),
\begin{equation}\label{eqn.derivative_third_variable}
\dd_3\Phi\big(p,\partial_t,A(p)\big)\cdot M=\tr\big(B(p)\cdot M\big)\ .
\end{equation}
By \eqref{eq.rescaled_normal_shape_minimal}, for all $p$ and for all $\phi$,
\begin{equation*}
d_3\mathcal{H}(p,0,0)\cdot\phi = L_p\phi\ ,
\end{equation*}
where the linear partial differential operator $L_p$ is given by
\begin{equation*}
L_p\phi:= -\tr\big(B(p)\cdot G(p)^{-1}\cdot d_x^2\phi\big)\ .
\end{equation*}
This operator is second order and elliptic, with constant coefficients. By Lemma \ref{lemma:CharacterizationOfStability}, Item (i), it has trivial Dirichlet kernel over $\closedtwoball$ and hence, by Fredholm's alternative, it defines a linear isomorphism $L_p:C^{2,\alpha}_0(\closedtwoball)\longrightarrow C^{0,\alpha}(\closedtwoball)$.

After reducing $r_0$ if necessary, the implicit function theorem yields a unique smooth function
\begin{equation*}
\phi:\Omega'\times[0,r_0)\longrightarrow C^{2,\alpha}_0(\closedtwoball)
\end{equation*}
such that $\phi(p,0)=0$ and, for all $p\in\Omega'$ and $r\in[0,r_0)$,
\begin{equation*}
\mathcal{H}\big(p,r,\phi(p,r)\big)=0\ .
\end{equation*}
In particular, for all $(p,r)$, $\phi(p,r)$ vanishes over $\mathbb{S}^1$. Furthermore, by elliptic regularity, $\phi$ is a smooth function of all its arguments.

We now verify stability. For small $r$, the linearization of $\mathcal{H}$ about $\phi(p,r)$ is a small $C^{0,\alpha}$ perturbation of $L_p$. Since the first Dirichlet eigenvalue of $L_p$ is strictly positive, it follows that the same holds for this linearized operator provided that $r$ is sufficiently small. Strict stability follows.
\end{proof}

\subsection{Smoothness and analytic extendibility}
\label{ssection:smooth-coordinates}
We now address higher-order regularity and extendibility of the functions constructed in Theorem \ref{thm:AsymptoticPlateauProblem}. Before proceeding, we review an elementary example which illustrates in the local setting the phenomenon that we wish to reproduce. The family of graphs
\begin{equation}
\label{equation:trivial-example}
\Omega \times [0,r_0) \times B_r(p) \ni (p,r,y) \mapsto \big(y, \lambda(p)(|y|^2-r^2)/2\big) \in X
\end{equation}
extends smoothly across the boundary $\partial X$ using the change of variables $s = r^2$ and by allowing $y \in B_1(p)$. That is, with respect to the coordinates $s=r^2$, we consider the family
\begin{equation*}
\Omega \times [-s_0,s_0] \times B_1(p) \ni (p,s,y) \mapsto (y,\lambda(p)\big(|y|^2-s)/2\big)\ .
\end{equation*}
This family is smooth with respect to all parameters, extends the initial family \eqref{equation:trivial-example}, and lies outside $X$ for $s \leq 0$. This is the structure that we aim to extend to general small $(\Phi,k)$-discs. Thus, for $s_0 > 0$, we consider the set
\begin{equation*}
E := \{ (p,s,y) ~:~ p \in \Omega, s \in [-s_0,s_0], y \in \Omega, |y-p| \leq 1\}\ .
\end{equation*}
Let $f^\pm_{p,r,\phi}$ be as in Theorem \ref{thm:AsymptoticPlateauProblem}.
\begin{theorem}
\label{theorem:uniform-smoothness}
For sufficiently small $s_0>0$, there exist smooth functions $\tilde{f}^\pm : E \to \R$ such that, for all $s \geq 0$, and for all $|y|^2 \leq s$,
\begin{equation*}
f^\pm_{p,\sqrt{s},\phi}(q) = \tilde{f}^\pm(p,s,q)\ .
\end{equation*}
In addition, if $\bar g$ is analytic, then, upon reducing the size of $E$, $\tilde{f}^\pm$ are analytic and uniquely determined by $f^\pm$ on $\{p \in \partial X, s \geq 0, |y|^2 \leq s\}$.
\end{theorem}
\noindent As in the previous section, we will henceforth restrict attention to the study of $f^+_{p,r,\phi}$, and we will suppress the superscript ``$\pm$'' in all that follows.

\subsubsection{The Taylor expansion}
Let $\phi$ be as in Theorem \ref{thm:AsymptoticPlateauProblem}. We consider the Taylor expansion of this function about $r=0$, that is
\begin{equation}
\label{eqn:TaylorSeriesOfPhi}
\phi(p,r,x) = \sum_{i=0}^N r^k\phi_k(p,x) + \cO\big(r^{N+1}\big)\ .
\end{equation}
We first show that each term of this Taylor series has a well-defined parity.
\begin{prop}
\label{prop:ParityOfTaylorTerms}
For all $p$, for all $k$, and for all $k$,
\begin{equation}
\label{eqn:ParityOfTaylorTerms}
\phi_k(p,-x) = (-1)^k\phi_k(p,x)\ .
\end{equation}
\end{prop}
\begin{proof}
Indeed, denote $\tilde{\phi}(p,r,x):=\phi(p,-r,-x)$. By \eqref{equation:inversion},
\begin{equation*}
\mathbf{H}\big(p,r,x,(J^2\tilde{\phi})(p,r,x)\big) = \mathbf{H}\big(p,-r,-x,(J^2\phi)(p,-r,-x)\big) = 0\ .
\end{equation*}
It follows by uniqueness that $\tilde{\phi}=\phi$. Comparing coefficients of the Taylor series \eqref{eqn:TaylorSeriesOfPhi} yields, for all $k$,
\begin{equation}
\phi_k(p,-x) = (-1)^k\phi_k(p,x)\ ,
\end{equation}
as desired.
\end{proof}

In order to prove Theorem \ref{theorem:uniform-smoothness}, it will suffice to show that, for all $p$, and for all $k$, the function $\phi_k(p,\cdot)$ is a polynomial in $x$ of degree at most $(k+2)$. We note first that the functional $H_{p,r,\phi}$ be can be expressed as a non-linear partial differential operator in a slightly different manner to that used in Section \ref{s.asymptotic_Analysis}.
\begin{lemma}
\label{lemma:SecondAsymptoticsOfH}
There exists a smooth function $\mathbf{H}':\Omega\times[0,r_0)\times\R^2\times\mathcal{J}^2\R^2\rightarrow\R$ such that, for all $p$, $r$, $\phi$, and $x$,
\begin{equation}
\label{eqn:SecondAsymptoticsOfHI}
H_{p,r,\phi}(x) = \mathbf{H}'\big(p,r,\big(rx,r^2\phi(x),r\dd_x\phi(x),\dd^2_x\phi(x)\big)\big)\ .
\end{equation}
Furthermore, or all $p$,
\begin{equation}
\label{eqn:SecondAsymptoticsOfHII}
\mathbf{H}'(p,0,\mathbf{0}) = 0\ .
\end{equation}
\end{lemma}
\begin{proof} This follows in the same manner as in the proofs of Lemmas \ref{lemma:RescaledNormal}, \ref{lemma.rescaled_shape_minimal} and \ref{lemma:smooth-functional} using the formulae \eqref{eq.graph_expansions_minimal} for the normal and shape operator, the formulae \eqref{eqn:ReparametrizedPhi} for $f_{p,r\phi}$ and the formula \eqref{eqn:AsymptoticsOfF} for the derivatives of this function.
\end{proof}

For all $k$, we denote by $\mathcal{P}^k$ the space of polynomials of degree at most $k$ on $\BB^2$, and by $\mathcal{P}^k_0$ the subspace consisting of those polynomials which vanish along $\mathbb{S}^1$. We will use induction to prove the following key technical result.
\begin{prop}
\label{proposition:key}
For all $k$, and for all $p$, $\phi_k(p,\cdot)\in\mathcal{P}^{k+2}_0$.
\end{prop}

In order to prove this proposition, we first consider smooth functions $F:\R\times\R^N\rightarrow\R$ and $Y:[0,r_0]\times\closedtwoball\rightarrow\R^N$ such that
\begin{enumerate}[label=(\roman*)]
\item $F(0,\mathbf{0})=0$;
\item $Y(0,x)=0$ for all $x \in \BB^2$; and
\item $F\big(r,Y(r,x)\big)=0$ for all $r \in [0,r_0]$ and for all $x \in \BB^2$.
\end{enumerate}
We consider the Taylor expansion of $F$ about $(0,\mathbf{0})$, that is
\begin{equation*}
F(r,y) = \sum_{n=0}^NF_n(r,y) + \cO\big((r,y)^{N+1}\big)\ ,
\end{equation*}
where, for all $n$, $F_n$ is a homogeneous polynomial of order $n$. By hypothesis,
\begin{equation*}
F_0=0\ ,
\end{equation*}
and, by Taylor's theorem,
\begin{equation*}
F_1(r,y) = \dd_1F(0)\cdot r + \dd_2F(0)\cdot y\ .
\end{equation*}
We likewise consider the Taylor expansion of $Y$ in $r$ about $r=0$, that is
\begin{equation*}
Y(r,x) = \sum_{n=0}^N r^n Y_n(x) + \cO\big(r^{N+1}\big)\ .
\end{equation*}
By hypothesis
\begin{equation*}
Y_0=0\ .
\end{equation*}
The inductive step in the proof of Proposition \ref{proposition:key} is proven as follows.
\begin{lemma}
\label{lemma:InductiveStepOfTaylorSeriesForPhi}
For all $k\geq 1$, if $Y_l\in\mathcal{P}^l$ for all $0\leq l\leq k-1$, then
\begin{equation*}
\dd_2 F(0)\cdot Y_k(x)\in\mathcal{P}^k\ ,
\end{equation*}
\end{lemma}
\begin{proof}
Let $C_k$ denote the coefficient of $r^k$ in the Taylor expansion of $F\big(r,Y(r,x)\big)$. Since $Y_0$ vanishes,
\begin{equation*}
C_k = \dd_2F(0)\cdot Y_k(x) + \sum_{\alpha\in A} Q_\alpha\big(Y_1(x),\cdots,Y_{k-1}(x)\big)\ ,
\end{equation*}
for a suitable family of monomials $(Q_\alpha)_{\alpha\in A}$. Since $F\big(r,Y(r,x)\big)$ vanishes, this coefficient also vanishes, and hence
\begin{equation*}
\dd_2F(0)\cdot Y_k(x) = -\sum_{\alpha\in A} Q_\alpha\big(Y_1(x),\cdots,Y_{k-1}(x)\big)\ .
\end{equation*}

Now choose $\alpha\in A$, and let $\beta_1,\cdots,\beta_{k-1}$ denote the respective exponents of $Y_1,\cdots,Y_{k-1}$ in the monomial $Q_\alpha$. Since, for all $l$, $Y_l$ contributes a factor of $r^l$,
\begin{equation*}
\sum_{l=1}^{k-1}l\cdot\beta_l \leq k\ .
\end{equation*}
It follows by hypothesis that, for all $\alpha\in A$,
\begin{equation*}
Q_\alpha\big(Y_1(x),\cdots,Y_{k-1}(x)\big) \in \mathcal{P}^k\ ,
\end{equation*}
so that
\begin{equation*}
\dd_2 F(0)\cdot Y_k(x)\in\mathcal{P}^k\ ,
\end{equation*}
as desired.
\end{proof}
We will also require the following technical lemma.
\begin{lemma}
\label{lemma:beta}
Consider the second order, linear, elliptic partial differential operator
\begin{equation*}
Lu:=a^{ij}\partial_i\partial_j u\ ,
\end{equation*}
where the coefficients $a^{ij}$ are constant. For all $k$, the operator
\begin{equation*}
T_k : \mathcal{P}^k \to \mathcal{P}^k;\ f\mapsto L\big((1-|x|^2)f\big)
\end{equation*}
is a linear isomorphism.
\end{lemma}
\begin{proof}
For all $k$, since the zero'th and first order coefficients of $L$ vanish, $T_k$ defines a linear map from $\mathcal{P}^k$ to itself. Since the Dirichlet kernel of $L$ is trivial, $T_k$ is a linear isomorphism for all $k$, and the result follows.
\end{proof}

\begin{proof}[Proof of Proposition \ref{proposition:key}.] The result is proven by induction. We therefore suppose that, for all $0\leq l<k$, $\phi_l(p,\cdot)\in\mathcal{P}^{l+2}_0$. To simplify presentation, we suppress the base point $p$ in all that follows. Consider now the function
\begin{equation*}
Y(r,x) := \big(rx,r^2\phi(x),r\dd_x\phi(x),\dd_x^2\phi(x)\big)\ ,
\end{equation*}
with Taylor expansion
\begin{equation*}
Y(r,x) = \sum_{n=0}^N r^n Y_n(x) + \cO\big(r^{N+1}\big)\ .
\end{equation*}
For all $n\geq 1$,
\begin{equation}
\label{eqn:DecompositionOfY}
Y_n(x) = (\delta_{n1}x,\phi_{n-2}(x),\dd_x\phi_{n-1}(x),\dd^2_x\phi_n(x)\big)\ ,
\end{equation}
where, by convention $\phi_n=0$ for $n<0$. By hypothesis, for all $0\leq l<k$, $Y_l\in\mathcal{P}^l$. Furthermore, since $\phi(0,x)=0$,
\begin{equation*}
Y_0(x) = \big(0,0,0,\dd_x^2\phi_0(x)\big) = \mathbf{0}\ .
\end{equation*}

We now denote
\begin{equation*}
F(r,Y) := \mathbf{H}'\big(r,rx,r^2\phi(x),r\dd_x\phi(x),\dd_x^2\phi(x)\big)\ ,
\end{equation*}
and we consider the Taylor expansion of this function in $(r,Y)$, that is
\begin{equation*}
F(r,Y) := \sum_{n=0}^N F_n(r,Y) + \cO\big((r,Y)^{N+1}\big)\ ,
\end{equation*}
where, for all $n$, $F_n$ is a homogeneous polynomial of order $n$ in its arguments. By \eqref{eqn:SecondAsymptoticsOfHII}, $F(0,\mathbf{0})=0$, so that
\begin{equation*}
F_0 = 0\ .
\end{equation*}
It follows by Lemma \ref{lemma:InductiveStepOfTaylorSeriesForPhi} that
\begin{equation*}
\dd_2F(0,\mathbf{0})\cdot Y_k \in\mathcal{P}^k\ .
\end{equation*}
However, by \eqref{eqn:DecompositionOfY},
\begin{equation*}
\begin{aligned}
\dd_2F(0,\mathbf{0})\cdot Y_k &= \dd_2\mathbf{H}'(0,\mathbf{0})\cdot\delta_{k1}x + \dd_3\mathbf{H}'(0,\mathbf{0})\cdot\phi_{k-2}(x)\\
&\qquad + \dd_4\mathbf{H}'(0,\mathbf{0})\cdot\big(\dd_x\phi_{k-1}(x)\big) + \dd_5\mathbf{H}'(0,\mathbf{0})\cdot\big(\dd_x^2\phi_k(x)\big)\ .
\end{aligned}
\end{equation*}
By the inductive hypothesis, $\phi_{k-2},\dd_x\phi_{k-1}\in\mathcal{P}^k$, and hence
\begin{equation*}
\dd_5\mathbf{H}'(0,\mathbf{0})\cdot\big(\dd_x^2\phi_k(x)\big) \in \mathcal{P}^k\ .
\end{equation*}
Finally, as in the proof of Theorem \ref{thm:AsymptoticPlateauProblem}, for all $\psi$,
\begin{equation*}
d_5\mathbf{H}'(0,\mathbf{0})\cdot(\dd_x^2\psi) = L_p\psi\ ,
\end{equation*}
where the linear partial differential operator $L_p$ is given by
\begin{equation*}
L_p\psi:= -\tr\big(B\cdot G^{-1}\cdot d_x^2\psi\big)\ ,
\end{equation*}
for some positive-definite matrices $B$ and $G$. It follows by Lemma \ref{lemma:beta} that $\phi_k\in\mathcal{P}_0^{k+2}$, and this completes the proof.
\end{proof}

\subsubsection{Proof of smoothness and analyticity}

We are now ready to prove Theorem \ref{theorem:uniform-smoothness}.
\begin{proof}[Proof of Theorem \ref{theorem:uniform-smoothness}]
\emph{The smooth case.} Consider the domain
\begin{equation*}
\hat{\Omega} := \big\{(p,s,q)\ \big|\ p\in\Omega'\ ,\ s>0\ ,\ \|p-q\|^2\leq s\big\}\ ,
\end{equation*}
and consider the function
\begin{equation*}
\tilde{f}(p,s,q) := f_{p,\sqrt{s},\phi}(q)\ .
\end{equation*}
We will show that all derivatives of $\tilde{f}$ to all orders are uniformly bounded over $\hat{\Omega}$. Since the boundary of $\hat{\Omega}$ is smooth, Theorem \ref{theorem:uniform-smoothness} then follows immediately by Whitney's extension theorem.

By Theorem \ref{thm:AsymptoticPlateauProblem}, for all $N$,
\begin{equation*}
f_{p,r,\phi}(q) = \frac{1}{2}\lambda_\pm(p)\big(\|q-p\|^2 - r^2\big) + \sum_{k=0}^{2N}r^{k+2}\phi_k\bigg(p,\frac{q-p}{r}\bigg) + r^{2N+2}\psi_N\bigg(p,\frac{q-p}{r}\bigg)\ ,
\end{equation*}
for some smooth function $\psi_N:\Omega'\times[0,r_0)\times\closedtwoball\rightarrow\R$. By Proposition \ref{proposition:key}, for all $k$, $\phi_k$ is a polynomial of order at most $(k+2)$ in $(q-p)/r$. Furthermore, by \eqref{eqn:ParityOfTaylorTerms}, when $k$ is even (resp. odd), $\phi_k$ only contains terms of even (resp. odd) order in $r$. There therefore exists a sequence $P_0,\cdots,P_{2N}:\Omega'\times\R\times\closedtwoball\rightarrow\R$ of functions, which are polynomials in $\R\times\closedtwoball$ such that
\begin{equation*}
f_{p,r,\phi}(q) = \frac{1}{2}\lambda_\pm(p)\big(\|q-p\|^2-r^2\big) + \sum_{k=0}^{2N}P_k(p,r^2,q-p) + r^{2N+2}\psi_N\bigg(p,\frac{q-p}{r}\bigg)\ .
\end{equation*}
Setting $s:=r^2$ yields
\begin{equation}
\label{eqn:AsymptoticsOfTildeF}
\tilde{f}(p,s,q) = \frac{1}{2}\lambda_\pm(p)\big(\|q-p\|^2-s\big) + \sum_{k=0}^{2N}P_k(p,s,q-p) + s^{N+1}\psi_N\bigg(p,\frac{q-p}{\sqrt{s}}\bigg)\ .
\end{equation}

Consider now the last term of \eqref{eqn:AsymptoticsOfTildeF}. Differentiating with respect to $q$ and $p$ introduces a factor of at most $1/\sqrt{s}$, and differentiating with respect to $s$ introduces a factor of at most $(q-p)/\sqrt{s}^3$. Since $\|q-p\|\leq\sqrt{s}$ over $\hat{\Omega}$, it follows that, for all $i,j,k$ such that $(i+2j+k)\leq2(N+1)$, there exists $C\geq 0$ such that, over $\hat{\Omega}$,
\begin{equation*}
\big\|\dd_p^i \dd_s^j \dd_q^k\big(s^{N+1}\psi_N\big(p,(q-p)/\sqrt{s}\big)\big)\big\|_{C^0} \leq C\ .
\end{equation*}
For all such $i,j,k$, there therefore exists $C'\geq 0$ such that, over $\hat{\Omega}$,
\begin{equation*}
\big\|\dd_p^i \dd_s^j \dd_q^k\tilde{f}\big\|_{C^0} \leq C'\ .
\end{equation*}
Since $N$ is arbitrary, it follows that all derivatives of $\tilde{f}$ to all orders are uniformly bounded over $\hat{\Omega}$, and this completes the proof in the smooth case. \\

\emph{The analytic case.} This in fact follows immediately from the preceding discussion. To see this, consider the analytic function $h:(-r_0,r_0)\times\closedtwoball\rightarrow\mathbb{R}$ given by
\begin{equation*}
h(r,x) = \sum_{k\geq 1, \atop{\left|\gamma\right|\leq k+2\atop\left|\gamma\right| = k\ \mathrm{mod}\ 2}}a_{k,\alpha}r^kx^\gamma\ .
\end{equation*}
By analyticity, $h$ has non-trivial radius of convergence about $(0,0)$. Consider now the function
\begin{equation}
\label{eqn:AnalyticityComposedFunction}
\tilde{h}(s,x) := sh\big(\sqrt{s},x/\sqrt{s}\big)
=\sum_{k\geq 1, \atop{\left|\gamma\right|\leq k+2\atop\left|\gamma\right| = k\ \mathrm{mod}\ 2}}a_{k,\alpha}s^{1+(k-\left|\gamma\right|)/2}x^\gamma
=\sum_{m\geq 0,\gamma}a_{\left|\gamma\right|+2(m-1),\gamma}s^m x^\gamma\ .
\end{equation}
Note, in particular, that by Propositions \ref{proposition:key} and \ref{proposition:key}, the above Taylor series contains no singular terms. By analyticity of $h$, there exists $s_0$ and $\epsilon>0$ such that $\tilde{h}$ converges absolutely over $\{s_0\}\times B_\epsilon(0)$. For all $x\in B_\epsilon(0)$, the radius of convergence with respect to $s$ of the series in \eqref{eqn:AnalyticityComposedFunction} is thus at least $s_0$, and $\tilde{h}$ therefore extends analytically to $(-s_0,s_0)\times B_\epsilon(0)$.

By \eqref{eqn:ParityOfTaylorTerms}, Proposition \ref{proposition:key}, and elliptic regularity for analytic functions, the function $(r,q) \mapsto \phi(p,r,q)$ satisfies all of the above properties, and analyticity follows.
\end{proof}

\subsubsection{Proof of Theorem \ref{thm:SmallCircles}}

We now complete the proof of Theorem \ref{thm:SmallCircles}.
\begin{proof}[Proof of Theorem \ref{thm:SmallCircles}]
Theorem \ref{thm:SmallCircles} is a straightforward combination of Theorems \ref{thm:AsymptoticPlateauProblem} and \ref{theorem:uniform-smoothness}, where $\mathbf{f}^\pm(s,p,q)$ correspond to the functions $\tilde{f}^\pm(p,s,q)$ expressed in the original coordinates on $\Ss^2$.
\end{proof}

\section{The foliated Plateau problem, double fibrations, and the global Bolker condition}

\label{section:plateau}

This section is devoted to the proof of Theorem~\ref{th.Foliated_Plateau_posta}. The proof is divided into several parts. In the first part, we study an abstract two-dimensional foliated Plateau problem in the closed unit $2$-ball. In the second part, we apply this to the study of Jacobi operators of $(\Phi,k)$-discs spanned by round circles, and thereby prove Theorem~\ref{th.Foliated_Plateau_posta}, item (i), except for the Bolker condition, which is established in \S\ref{ssection:bolker}. The analytic case is treated in \S\ref{ssection:analytic}. Finally, we prove that the foliation extends smoothly (resp. analytically) across the boundary in \S\ref{ssection:extended-foliation} and prove Theorem~\ref{th.Foliated_Plateau_posta}, item (ii).

\subsection{The two-dimensional foliated Plateau problem}

\label{ss:TwoDSetting}

Let $L$ be a second-order linear elliptic partial differential operator on $\closedtwoball$ with smooth coefficients, as in \eqref{eqn:GeneralEllipticOperator}. We assume furthermore that $L$ is strictly stable in the sense of Definition~\ref{defn:Stability}.

\subsubsection{The boundary conditions for the Dirichlet problem}
\label{sss.boundary_Dirichlet}

Recall the space $\cT$ of trigonometric polynomials from \S\ref{sss.First_order_variations}.
Let $g:\mathbb{S}^1 \to \mathbb{R}$ be a smooth, positive function, and define
\begin{equation}\label{def_boundary_condition}
E := \big\{\theta \mapsto g(\theta)f(\theta) ~:~ f \in \cT\big\}\qquad\text{and}\qquad
E^\ast := \big\{\theta \mapsto g(\theta)f(\theta) ~:~ f \in \cT^\ast\big\}\ ,
\end{equation}
for $\ast \in \{-,0,+\}$. The vector space $E$ will be used to describe infinitesimal variations of circles in $\mathcal C$. We denote by $P^+E$ and $P^+E^\ast$ the corresponding oriented projectivizations. Up to identifying $\cT$ with $E$ via multiplication by the positive function $g$, the sign of the discriminant $\Delta$ defined on $\cT$ extends naturally to $E$ and $P^+E$. Trivially, it is positive on $E^+$ and $P^+E^+$, negative on $E^-$ and $P^+E^-$, and vanishes on $E^0$ and $P^+E^0$.

\subsubsection{The zero sets of solutions}
\label{sss.zero_sets}
By strict stability, $L$ has trivial Dirichlet kernel. Hence, by Fredholm's alternative, for every $f \in E$ there exists a unique solution
$\hat{f} \in C^\infty(\closedtwoball)$ of the Dirichlet problem
\begin{equation*}
L\hat{f} = 0
\qquad\text{and}\qquad
\hat{f}|_{\mathbb{S}^1} = f\ .
\end{equation*}
For every non-zero $f \in E$, we define
\begin{equation*}
Z_f := \hat{f}^{-1}\big(\{0\}\big)
\end{equation*}
to be the zero set of its Dirichlet extension. Note that $Z_f$ is a smooth curve precisely at those points where $\dd\hat{f} \neq 0$.

Let $\hat{Z}_f \subseteq S\closedtwoball$ denote the set of unit vectors which are normal to $Z_f$ at its smooth points and which point in the direction in which $\hat{f}$ increases. For every non-zero $f$, the set $\hat{Z}_f$ is a smooth (not necessarily compact) embedded curve in $S\closedtwoball$. Finally, note that both $Z_f$ and $\hat{Z}_f$ depend only on the class of $f$ in $P^+E$.

\subsubsection{The two-dimensional foliated Plateau problem}
The topological properties required for our construction are summarized in the following theorem. Recall that, by \eqref{def_boundary_condition}, every $f \in E^+$ has exactly two zeros on $\mathbb{S}^1$, which depend only on the class of $f$ in $PE^+$.

\begin{theorem}
\label{thm:coveringofunitbundle}
Let $L$ be strictly stable in the sense of Definition~\ref{defn:Stability}. Then:
\begin{enumerate}[label=\emph{(\roman*)}]
\item for every $f \in E^+$, the zero set $Z_f$ is a smooth embedded curve in $\closedtwoball$;
\item the family $(Z_f)_{f \in E^+}$ depends smoothly on $f$;
\item the lifted family $(\hat{Z}_f)_{f \in E^+}$ covers $S\twoball$;
\item for any pair of distinct points $x,y \in \twoball$, there exists $f \in E^+$ such that $x \in Z_f$ and $y \notin Z_f$.
\end{enumerate}
\end{theorem}

The rest of the section is devoted to the proof of this theorem. Properties (i) and (ii) are proven in Lemma \ref{lemma:SmoothVariation}. Properties~(iii) and~(iv), which encode the key geometric features of the family $(Z_f)_{f\in E^+}$, are proven respectively in Lemmas \ref{lemma:CoversUnitBundleOfdisc} and \ref{lemma:DistinctPointsOfdisc}. Property~(iii) asserts that the lifted curves $(\hat{Z}_f)$ cover the unit tangent bundle, in analogy to the way in which lifts of oriented geodesics between pairs of distinct ideal points foliate the unit tangent bundle of the hyperbolic disc. Property~(iv) ensures that the family separates points in $\twoball$. Together, these two properties will play a central role in the proof of the global Bolker condition.

\subsubsection{The geometry of zero sets} \label{sssection:geometry-zero-set}
We will use the following local description of zero sets (see \cite[Section 6.2]{Colding-Minicozzi-11}).

\begin{lemma}
\label{lemma:structureofnullset}
Let $\hat f:\twoball \to \mathbb{R}$ be a non-constant solution of $L\hat f=0$. Then, for every $x \in \twoball$ such that $\hat f(x)=0$, there exist an integer $k \geq 1$, a neighbourhood $U$ of $0$ in $\mathbb{C}$, a neighbourhood $V$ of $x$ in $\twoball$, and a $C^1$ real diffeomorphism $\phi:U \to V$ such that $\phi(0)=x$ and
\begin{equation*}
(\hat f \circ \phi)(z) = \Re(z^k)\ .
\end{equation*}
\end{lemma}

\begin{remark}\label{rem_locally_connected}
This lemma implies in particular that $Z_f$ is locally connected.
\end{remark}

We now describe the topology of $Z_f$, depending on the number of zeroes of $f$ in $\mathbb{S}^1$. As for first-order trigonometric polynomials on $\mathbb{S}^1$, elements $f\in E$ satisfy
\begin{equation*}
\# (Z_f\cap\mathbb{S}^1)\in\{0,1,2\}\ ,
\end{equation*}
and furthermore, when $\# (Z_f\cap\mathbb{S}^1)=2$, these zeroes are non-degenerate, in the sense that $f'\neq 0$ at those points.

\begin{lemma}
\label{lemma:ZeroOrOneZero}
Let $f$ be an element of $E$.
\begin{enumerate}[label=\emph{(\roman*)}]
\item If $\# (Z_f\cap\mathbb{S}^1)=0$, then $Z_f=\emptyset$.
\item If $\# (Z_f\cap\mathbb{S}^1)=1$, then $Z_f\cap\twoball=\emptyset$.
\end{enumerate}
\end{lemma}
\begin{proof}
Assertion (i) follows directly from Lemma~\ref{lemma:WeakMP}. For (ii), let $u$ be as in Definition~\ref{defn:Stability}. Let $L'$ be as in \eqref{eqn:DefinitionOfLPrime}, and note that its zero'th order coefficient vanishes. Without loss of generality, $f/u\geq 0$ over $\mathbb{S}^1$. Since $L'(f/u)=0$, it follows by the maximum principle (Theorem~\ref{thm:MaxPrinc}) that $\hat{f}/u>0$ over $\mathbb{B}^2$. Consequently, $\hat{f}>0$ over $\mathbb{B}^2$ and so $Z_f\cap\mathbb{B}^2=\emptyset$, as desired.
\end{proof}

We now address the case where $\# (Z_f\cap\mathbb{S}^1)=2$.

\begin{lemma}
\label{lemma:Connected}
If $Z_f \neq \emptyset$, then $Z_f \cup \mathbb{S}^1$ is connected.
\end{lemma}

\begin{proof}
Suppose the contrary, so that $Z_f \cup \mathbb{S}^1$ is not connected. Let $Z_0\subseteq Z_f\cap\mathbb{B}^2$ be a connected component not meeting $\mathbb{S}^1$. Since $Z_f\cap\mathbb{B}^2$ is locally connected, $Z_0$ is closed, and hence compact. Let $U \subset \twoball$ be a connected open neighbourhood of $Z_0$ whose closure is disjoint from the other connected components of $Z_f$, and such that $\partial U$ is smooth. Note that $U$ is topologically a disk with finitely many holes removed. Upon filling in the inner components of the complement of $U$, we may suppose that $U$ has a single boundary component. Since $\partial U \cap Z_f = \emptyset$, the function $\hat{f}$ has a constant sign on $\partial U$. Upon replacing $f$ by $-f$ if necessary, we may assume that $\hat{f} > 0$ on $\partial U$. It follows by the maximum principle (Lemma~\ref{lemma:WeakMP}) that $\hat{f} > 0$ in $U$, contradicting the fact that $Z_0 \subset U$. It follows that $Z_f \cup \mathbb{S}^1$ is connected, as desired.
\end{proof}

It follows from Lemma~\ref{lemma:structureofnullset}, Lemma~\ref{lemma:Connected}, and the non-degeneracy of the points of $Z_f \cap \mathbb{S}^1$, that $Z_f \cup \mathbb{S}^1$ is a compact, connected graph. Let $V$, $E$, and $F$ denote respectively the numbers of vertices, edges, and faces of this graph. We decompose the sets of vertices and edges into interior and boundary parts. We denote respectively by $V_i$ and $E_i$ the numbers of vertices and edges contained in $\twoball$, and respectively by let $V_e$ and $E_e$ the numbers of vertices and edges contained in $\mathbb{S}^1$.

\begin{lemma}
\label{lemma:NoInternalRegions}
Every face of $\twoball \setminus Z_f$ shares at least one edge with $\mathbb{S}^1 = \partial \twoball$.
\end{lemma}

\begin{proof}
Suppose the contrary, so that there exists a face $U$ of $\twoball \setminus Z_f$ whose boundary does not contain any edge of $\mathbb{S}^1$. Then $\partial U \subseteq Z_f$, and hence $\hat{f}|_{\partial U} = 0$. It follows by the maximum principle (Lemma~\ref{lemma:WeakMP}) that $\hat{f}$ vanishes identically on $U$. By unique continuation (see \cite{Aronszajn-57} or \cite[Chapter XVII, Section 17.2]{Hormander-07}), $\hat{f}$ must then vanish identically on $\twoball$, which contradicts the assumption that $f \neq 0$.
\end{proof}

\begin{lemma}
\label{lemma:EulerI}
If $\#(Z_f\cap\mathbb{S}^1) = 2$, then $F\geqslant 2 + V_i$.
\end{lemma}

\begin{proof}
Note first that $E_e=V_e=2$. Euler's formula therefore yields
\begin{equation*}
F - (E_i+2) + (V_i+2) = 1\ .
\end{equation*}
Hence
\begin{equation}
\label{eqn:EulerIStepI}
F - E_i + V_i = 1\ .
\end{equation}
However, by Lemma \ref{lemma:structureofnullset}, every internal vertex has valency at least $4$, whilst every external vertex has valency at least $3$. Since every edge meets exactly $2$ vertices, it follows that
\begin{equation*}
4V_i + 3V_e = 4V_i + 6 \leq 2E_i + 2E_e = 2E_i + 4\ .
\end{equation*}
Thus
\begin{equation}
\label{eqn:EulerIStepII}
E_i - 2V_i \geq 1\ .
\end{equation}
Summing \eqref{eqn:EulerIStepI} and \eqref{eqn:EulerIStepII} yields $F - V_i \geq 1 + 1=2$ and the result follows.
\end{proof}

\begin{lemma}
\label{lemma:CaseManyIntersections}
If $\#(Z_f\cap\mathbb{S}^1)=2$, then $V_i =0$.
\end{lemma}
\begin{proof}
Suppose the contrary, so that $V_i \geq 1$. Then, by Lemma~\ref{lemma:EulerI}, $Z_f$ divides $\twoball$ into at least $F \geq 3$ regions. On the other hand, the boundary $\mathbb{S}^1$ is divided into $2$ edges, and each such edge is adjacent to exactly one face of the complement of $Z_f$, so that at most $2$ faces can meet an edge of $\mathbb{S}^1$. There therefore exists a face $U$ that does not meet any edge of $\mathbb{S}^1$. This contradicts Lemma~\ref{lemma:NoInternalRegions}, and the result follows.
\end{proof}

We are now ready to prove Items (i) and (ii) of Theorem \ref{thm:coveringofunitbundle}.
\begin{lemma}
\label{lemma:SmoothVariation}
For all $f\in E^+$, $Z_f$ is a smooth, embedded curve in $\closedtwoball$ which connects the two zeros of $f$. Furthermore, $Z_f$ varies smoothly with $f\in E^+$.
\end{lemma}
\begin{proof}
Indeed, by Lemma \ref{lemma:CaseManyIntersections}, $Z_f$ has no internal vertices. It is therefore a smoothly embedded curve in $\twoball$. By Lemma \ref{lemma:Connected} it is connected. By non-degeneracy, $Z_f$ is also smooth at its two end-points. This proves the first assertion.  To prove the second assertion, consider the function $e:E^+\times\closedtwoball\rightarrow\Bbb{R}$ given by
\begin{equation*}
e(f,x):=\hat{f}(x)\ ,
\end{equation*}
and let $\dd_2 e$ denote its differential with respect to the second component. For all $f\in E^+$, the absence of internal vertices implies that, for all $x\in Z_f$,
\begin{equation*}
\dd\hat{f}(x)\neq 0\ ,
\end{equation*}
that is
\begin{equation*}
(\dd_2 e)(f,x)=\dd\hat{f}(x)\neq 0
\end{equation*}
at every point of $e^{-1}(\{0\})$. It now follows by the implicit function theorem that the zero set $Z_f = \{x : e(f,x)=0\}$ depends smoothly on $f$, and this completes the proof.
\end{proof}

\subsubsection{The structure of the foliation}

Recall that elements of $E^+$ are those elements of $E$ with two non-degenerate zeros, and that elements of $E^0$ are those with only one zero. In what follows, we will make use of the foliation $\cW$ of $P^+E^+$ by open arcs defined in \S\ref{sss.foliating_space_like}, as well as the map
\begin{equation*}
\zeta:PE^0 \to \mathbb{S}^1,\quad [f]\mapsto \Zero(f)\ .
\end{equation*}
Recall that these objects are invariant under multiplication by $-1$ (see Lemma~\ref{lem.invariant_-1}).

Given a leaf $W$ of $\cW$, let $\partial W$ denote its two extremities in $P^+E^0$. A priori, $\zeta(\partial W)$ consists of either one or two points of $\mathbb{S}^1$. We now show that it always consists of two distinct points, and that each leaf of $\mathcal{W}$ gives rise to a foliation of the disc by nodal sets, analogous to the foliation of the hyperbolic plane by geodesics orthogonal to some fixed geodesic.

\begin{lemma}
\label{lemma:SmoothNDFoliation}
For every leaf $W$ of $\mathcal{W}$, the set $\zeta(\partial W)$ consists of exactly two distinct points of $\mathbb{S}^1$. Moreover, the family $(Z_f)_{f\in W}$ defines a smooth, non-singular foliation of $\closedtwoball \setminus \zeta(\partial W)$.
\end{lemma}

\begin{proof}
Let $W$ be such a leaf. Any $[f_0] \in W$ has a unique representative of the form
\begin{equation*}
f_0(\theta)=g(\theta)(\gamma - \alpha\cos(\theta) - \beta\sin(\theta))
\end{equation*}
where $(\alpha,\beta)\in\mathbb{R}^2$ satisfies $\alpha^2+\beta^2=1$ and $|\gamma|<1$. The leaf $W$ is thus parametrized by
\begin{equation*}
h_t(\theta)=g(\theta)\big(t - \alpha\cos(\theta) - \beta\sin(\theta)\big),
\qquad (t,\theta)\in(-1,1)\times\mathbb{S}^1\ .
\end{equation*}
Since $g$ is positive and smooth on $\mathbb{S}^1$, there exists $\varepsilon>0$ such that, for all $(t,\theta)\in(-1,1)\times\mathbb{S}^1$,
\begin{equation*}
\partial_t h_t(\theta)=g(\theta)\geq \varepsilon\ .
\end{equation*}
In particular,
\begin{equation*}
h_1(\theta) - h_{-1}(\theta) \geq 2\eps\ ,
\end{equation*}
so that $\zeta\big([h_1]\big)\neq\zeta\big([h_{-1}]\big)$, and $\zeta(\partial W)$ therefore consists of $2$ distinct points. For all $t$, let $\hat{h}_t\in C^\infty(\closedtwoball)$ denote the unique solution of
\begin{equation*}
L\hat{h}_t = 0\qquad\text{and}\qquad \hat{h}_t|_{\mathbb{S}^1} = h_t\ .
\end{equation*}
Since
\begin{equation*}
L\partial_t\hat{h}_t = \partial_tL\hat{h}_t=0\qquad\text{and}\qquad \partial_t\hat{h}_t|_{\mathbb{S}^1} = g \geq\eps\ ,
\end{equation*}
so that, by Lemma \ref{lemma:WeakMP}, upon reducing $\eps$ if necessary,
\begin{equation*}
\partial_t\hat{h}_t \geq \eps\ .
\end{equation*}
It follows by Lemma \ref{lemma:SmoothVariation} that $Z_{h_t}$ smoothly foliates some open subset $U$ of $\closedtwoball$. It remains only to show that $U=\closedtwoball\setminus\zeta(\partial W)$. However, by continuity, every point $x\in\partial U$ is an element of $Z_f$, for some $f\in\partial W$, so that $\partial U\subseteq\zeta(\partial W)$. Since the converse inclusion is trivial, this completes the proof.
\end{proof}

For $x \in \twoball$, let $C(x)$ denote the set of all $[f] \in P^+E^+$ such that $x \in Z_f$.
\begin{lemma}
For every $x \in \twoball$, the set $C(x)$ satisfies
\begin{enumerate}[label=\emph{(\roman*)}]
\item $C(x)$ is a smooth, embedded submanifold of $P^+E^+$, diffeomorphic to $\mathbb{S}^1$;
\item $C(x)$ intersects each leaf of $\mathcal{W}$ transversally and in exactly one point;
\item $C(x)$ is invariant under multiplication by $-1$.
\end{enumerate}
\end{lemma}

\begin{proof}
We first construct a smooth map $\phi:P^+E^+\times[0,1]\rightarrow\closedtwoball$ such that, for all $[f]\in P^+E^+$, $\phi([f],\cdot)$ parametrizes $Z_f$. Note first that, for all $[f]\in P^+E^+$, since $\nabla\hat{f}\neq 0$ over $Z_f$, $Z_f$ carries a natural orientation, defined such that $\nabla\hat{f}$ points leftwards along this curve. We now fix an arbitrary riemannian metric over $\closedtwoball$ and we define $\phi$ such that, for all such $[f]$, $\phi([f],\cdot)$ is the unique positively oriented, constant speed parametrization of $Z_f$ with respect to this metric. By Lemma \ref{lemma:SmoothNDFoliation}, $\phi$ is smooth. We will henceforth only be interested in the restriction of this map to the open subset $P^+E^+\times(0,1)$. The result will follow upon studying the properties of this restriction.

Consider now a leaf $W$ of $\mathcal{W}$, and let
\begin{equation*}
h_s(\theta) := g(\theta)\big(s-\alpha\cos(\theta)-\beta\sin(\theta)\big)
\end{equation*}
denote its normalized parametrization, is in the proof of Lemma \ref{lemma:SmoothNDFoliation}. We introduce the function
\begin{equation*}
\psi:=\psi_W:(0,1)^2\rightarrow\mathbb{B}^2;(s,t)\mapsto\phi\big([h_s],t\big)\ .
\end{equation*}
We first claim that $\psi$ is bijective. Indeed, by Lemma \ref{lemma:SmoothNDFoliation}, the family $(Z_{h_s})_{s\in(-1,1)}$ foliates $\mathbb{B}^2$. Every point of $\mathbb{B}^2$ therefore lies on $Z_{h_s}$, for some $s$, and is thus equal to $\psi(s,t)$, for some $t$. This proves surjectivity. Suppose now that $\psi(s,t)=\psi(s',t')$ for some $(s,t)$ and $(s',t')$. Since these points trivially lie on the same leaf, $s=s'$, and since $\psi(s,\cdot)$ parametrizes this leaf, $t=t'$. This proves injectivity, so that $\psi$ is bijective, as asserted.

We now claim that $\psi$ is a smooth diffeomorphism. To this end, choose $(s,t)\in(0,1)^2$, and consider the curve
\begin{equation*}
\gamma(r) := \psi(s+r,t) = \phi\big([h_{s+r}],t\big)\ .
\end{equation*}
For all $r$, $\gamma(r)$ is an element of $Z_{h_{s+r}}$, so that
\begin{equation*}
\hat{h}_{s+r}\big(\gamma(r)\big) = 0\ .
\end{equation*}
Differentiating this identity yields
\begin{equation*}
\frac{\partial\hat{h}}{\partial r}\big(\gamma(0)\big)\bigg|_{r=0} + d\hat{h}_s\big(\gamma(0)\big)\cdot\dot{\gamma}(0) = 0\ ,
\end{equation*}
so that, as in the proof of Lemma \ref{lemma:SmoothNDFoliation},
\begin{equation*}
d\hat{h}_s\big(\psi(s,t)\big)\cdot\dot{\gamma}(0) = -\frac{\partial\hat{h}}{\partial r}\big(\psi(s,t)\big)\bigg|_{r=0} < 0\ .
\end{equation*}
It follows that $\partial_s\psi(s,t)=\dot{\gamma}(0)$ is non-zero and transverse to the kernel of $d\hat{h}_s(\psi(s,t))$. Since this kernel is the tangent line of $Z_{h_s}$, it follows that $\partial_s\psi(s,t)$ is transverse to $Z_{h_s}$. Since, by construction, $\partial_t\psi(s,t)$ is non-zero and tangent to $Z_{h_s}$, it follows that $D\psi(s,t)$ is a linear isomorphism. Finally, it follows by bijectivity and the inverse function theorem that $\psi$ is a smooth diffeomorphism, as asserted.

Let $\pi:P^+E^+\times(0,1)\rightarrow P^+E^+$ denote the projection onto first factor. Note that, for any leaf $W$ of $\mathcal{W}$, $\psi_W$ is, up reparametrization, the restriction of $\psi$ to the surface $\pi^{-1}(W)$. In particular, since $\psi_W$ is a diffeomorphism for all such $W$, it follows that $\phi$ is a smooth submersion, and the set $\hat{C}(x):=\phi^{-1}(\{x\})$ is therefore a smooth, $1$-dimensional, embedded submanifold of $P^+E^+\times(0,1)$.

Since $\phi([f],\cdot)$ converges uniformly to the constant map $\zeta([f_\infty])$ as $[f]\to [f_\infty]\in P^+E^0$, the closure of $\hat{C}(x)$ in $\overline{P^+E^+}\times[0,1]$ does not meet $\partial(P^+E^+\times[0,1])$. The set $\hat{C}(x)$ is thus compact, and therefore diffeomorphic to a disjoint union of finitely many copies of $\mathbb{S}^1$.

Since $\phi$ is constant over $\hat{C}(x)$, but restricts to a smooth embedding over every fibre of $\pi$, it follows that $\hat{C}(x)$ meets every such fibre transversally at at most one point. The projection $C(x)=\pi(\hat{C}(x))$ is therefore a smooth, $1$-dimensional, embedded submanifold of $P^+E^+$, also diffeomorphic to a disjoint union of finitely many copies of $\mathbb{S}^1$.

Now let $W$ be a leaf of $\mathcal{W}$. We claim that $C(x)$ meets $W$ transversally at exactly one point. Indeed, since $\phi$ is constant over $\hat{C}(x)$, but restricts to a smooth diffeomorphism over $\pi^{-1}(W)$, the curve $\hat{C}(x)$ meets the surface $\pi^{-1}(W)$ transversally at exactly one point, and the assertion follows upon taking projections. This proves (ii). We see also that $C(x)$ is homeomorphic to the leaf space of $\mathcal{W}$. It is therefore connected, and thus diffeomorphic to $\mathbb{S}^1$. This proves (i). Finally, since all elements of this construction are invariant under multiplication by $-1$, (iii) follows immediately, and this completes the proof.
\end{proof}

We are now ready to prove Items (iii) and (iv) of Theorem \ref{thm:coveringofunitbundle}. Recall that, for each $f$, $\hat{Z}_f \subseteq S\closedtwoball$ denotes the set of unit normal vectors over $Z_f$ which point in the direction in which $\hat{f}$ increases.

\begin{lemma}
\label{lemma:CoversUnitBundleOfdisc}
For all $v\in S\twoball$, there exists $[f]\in PE^+$ such that $v\in\hat{Z}_f$.
\end{lemma}
\begin{proof}
Fix $x\in\twoball$, and consider the map
\begin{equation*}
N:C(x)\rightarrow S_x\twoball
\end{equation*}
defined such that, for each $[f]\in C(x)$, $N([f])$ is the unit normal vector to $Z_f$ at $x$, pointing in the direction in which $\hat{f}$ increases. This map is smooth, and hence continuous. Furthermore, it satisfies $N([-f]) = -N([f])$ for all $[f]\in C(x)$. Since $C(x)\simeq \mathbb{S}^1$ and $S_x\twoball\simeq \mathbb{S}^1$, it follows that the degree of $N$ is odd, and therefore non-zero. $N$ is thus surjective, and the result follows.
\end{proof}

\begin{lemma}
\label{lemma:DistinctPointsOfdisc}
For all $x\neq y\in\twoball$, $C(x)\neq C(y)$.
\end{lemma}
\begin{remark} Since $C(x)$ and $C(y)$ both meet each leaf of $\mathcal{W}$ exactly once, it follows that both $C(x)\setminus C(y)$ and $C(y)\setminus C(x)$ are non-empty. This proves Theorem \ref{thm:coveringofunitbundle}, Item $(iv)$.
\end{remark}
\begin{proof}
Suppose the contrary, so that $C(x)=C(y)$. Then, for every $[f]\in C(x)$, the curve $Z_f$ passes through both $x$ and $y$. For each such $f$, we orient $Z_f$ so that $\hat{f}$ is positive on its left-hand side. Since $Z_f$ is a smoothly embedded curve, this orientation induces an ordering along $Z_f$. In particular, exactly one of the following two possibilities holds. Either $y$ lies after $x$ along $Z_f$, or it lies before $x$. We thus obtain a partition of $C(x)$ into two open subsets, namely
\begin{equation*}
\begin{aligned}
I^+ &:= \big\{[f]\in C(x)\ : \ y \text{ lies after } x \text{ along } Z_f\big\}\ \text{and}\\
I^- &:= \big\{[f]\in C(x)\ : \ y \text{ lies before } x \text{ along } Z_f\big\}\ .
\end{aligned}
\end{equation*}
These sets are clearly disjoint, and their union is $C(x)$. Thus, since $C(x)$ is connected, one of the two is empty. However, for every $f$, the curve $Z_{-f}$ coincides with $Z_f$ with the opposite orientation, and multiplication by $-1$ thus maps $I^+$ bijectively onto $I^-$. This is absurd, and the result follows.
\end{proof}

This completes the proof of Theorem~\ref{thm:coveringofunitbundle}. The two key properties, namely covering of the unit tangent bundle and separation of points, hold for any strictly stable operator $L$. This reflects the topological analogy between the family $\mathcal{Z}:=(Z_f)_{f\in E^+}$ of nodal sets and the space $\mathcal{G}^+$ of oriented geodesics of the hyperbolic disc. The space $\mathcal{G}^+$ itself naturally identifies with $(1+1)$-dimensional de Sitter space $\dS^{1,1}$, with its spacetime structure being manifested by its incidence geometry. Indeed, two geodesics are in spacelike relation whenever they intersect non-trivially, they are in null relation whenever they are asymptotic to one another, and are in timelike relation otherwise. This is the intuition behind the manner in which the spacetime structure discussed in \S\ref{sss.First_order_variations} influences the incidence geometry of $\mathcal{Z}$, and the key property making this possible is precisely the strict stability of $L$.

\subsection{The three-dimensional foliated Plateau problem} \label{ssection:plateau-3d}

We now return to our geometric setting. Let $\bar g$ be a smooth (resp. analytic) Riemannian metric on $X$ which we extend smoothly (resp. analytically) to a larger smooth (resp. analytic) Riemannian ball $(X_e, \bar g_e)$ containing $(X,\bar g)$. See Appendix \ref{section:extension}. Recall that $\partial X$ is oriented by the outward-pointing unit normal vector field $\nu$. Let $(\Phi,\Lambda)$ be an elliptic curvature function, as in Definition \ref{definition:elliptic-curvature-function}. Throughout this section, we assume that the metric $\bar g$ is $(\Phi,k)$-simple in the sense of Definition \ref{definition:phik-simplicity}.

\subsubsection{The phase space}
We define the space $\MD$ of \emph{marked $(\Phi,k)$-discs} by
\begin{equation*}
\MD := \big\{ (x,c)\in X\times\cC : x\in\opD(c)\big\}\ ,
\end{equation*}
and we consider the diagram
\begin{equation}
\label{equation:diagram_MD}
\begin{tikzcd}
& \MD \arrow[rd, "\pi_2"] \arrow[ld,"\pi_1"'] & \\
X & & \cC
\end{tikzcd}
\end{equation}
where $\pi_i$ denotes the projection onto the $i$-th factor.

We recall that $SX$ denotes the unit tangent bundle of $X$, and we denote by $\pi:SX\rightarrow X$ the canonical projection. This is a fibre bundle, with typical fibre $\mathbb{S}^2$. We define the Gauss map $N:\MD\rightarrow SX$ by assigning to each $(x,c)$ the unit normal vector of $\opD(c)$ at $x$ which is compatible with its orientation. Recall the open subset $SX^{\bullet}$ of $SX$ introduced in \S\ref{sss_unit_tangent_bundle}, namely
\begin{equation*}
SX^{\bullet}=SX\setminus\big\{(x,\pm \nu(x))\ : \ x\in\partial X\big\}\ .
\end{equation*}
The main result of this section is the following theorem.

\begin{theorem}[Structure of the phase space]
\label{thm:StructureOfdiscSpace}
If, for all $c\in\cC$, $\opD(c)$ is strictly stable in the sense of Definition~\ref{defn:StableSolutions}, then
\begin{enumerate}[label=\emph{(\roman*)}]
\item $\MD$ is a smoothly embedded hypersurface of $X \times\cC$;
\item the projection $\pi_2:\MD\rightarrow\cC$ makes $\MD$ into a smooth bundle with fibre $\closedtwoball$; and
\item the Gauss map $N$ defines a smooth diffeomorphism from $\MD$ onto $SX^{\bullet}$ that conjugates~\eqref{equation:diagram_MD} to the double fibration
\begin{equation}
\label{equation:diagram}
\begin{tikzcd}
& SX^{\bullet} \arrow[rd, "\pi_R"] \arrow[ld,"\pi"'] & \\
X & & \cC
\end{tikzcd}
\end{equation}
where $\pi$ denotes the basepoint projection, and $\pi_R:=\pi_2\circ N^{-1}$, that is, if $(x,c)=N^{-1}(x,v)$, then $\pi_R(x,v)=c$.
\end{enumerate}
\end{theorem}
\begin{remark} Note that $\pi:SX^\bullet$ trivially restricts to a fibre bundle over $\mathrm{Int}(X)$ with typical fibre $\mathbb{S}^2$. On the other hand, the fibre over each boundary point $x\in\partial X$ is the double-punctured sphere $S_xX\setminus\{\pm\nu(x)\}$.
\end{remark}

\subsubsection{The smooth structure}
We first prove Theorem~\ref{thm:StructureOfdiscSpace}, Items (i) and (ii).
\begin{lemma}
\label{lemma:MDIsSubmanifold}
If, for all $c\in\cC$, $\opD(c)$ is strictly stable in the sense of Definition~\ref{defn:Stability}, then
\begin{enumerate}[label=\emph{(\roman*)}]
\item $\MD$ is a smoothly embedded hypersurface of $X \times \cC$; and
\item the projection $\pi_2:\MD\rightarrow\cC$ makes $\MD$ into a smooth bundle with fibre $\closedtwoball$.
\end{enumerate}
\end{lemma}
\begin{proof}
We will prove (i) by showing that $\MD$ is everywhere locally the zero set of some submersion. In the process, this will yield, over a neighbourhood $U$ of every point $c_0\in\cC$, a natural parametrization of $\MD$ by $U\times\closedtwoball$, from which (ii) will follow.

Fix $c_0\in\cC$ and denote $D_0:=\opD(c_0)$. We first construct a smooth parametrization of a tubular neighbourhood $\Omega$ of $D_0$ in $X$ as follows. Let $e:\closedtwoball\rightarrow D_0$ be a smooth parametrization, let $\xi$ be a smooth vector field over $X$ which is transverse to $D_0$ and tangent to $\partial X$, and let $(\phi_t)_{t \in \R}$ denote its flow. For sufficiently small $\eps>0$, the map
\begin{equation*}
E:\closedtwoball\times(-2\eps,2\eps)\rightarrow X,\qquad (p,t)\mapsto\phi_t\big(e(p)\big)\ ,
\end{equation*}
is a smooth parametrization of a neighbourhood $\Omega:=\mathrm{Im}(E)$ of $D_0$ in $X$. We henceforth identify $\Omega$ via this parametrization with $\closedtwoball\times(-2\eps,2\eps)$. In particular, $D_0$ identifies with $\closedtwoball\times\{0\}$ and its boundary $c_0$ identifies with $\mathbb{S}^1\times\{0\}$.

Every curve $c\in\cC$ sufficiently close to $c_0$ is a graph over $c_0$. More precisely, there exists a neighbourhood $U$ of $c_0$ in $\cC$ and a smooth function
\begin{equation*}
f:U\times\mathbb{S}^1\rightarrow(-\eps,\eps)
\end{equation*}
such that each $c\in U$ is parametrized by the function $p\mapsto \phi_{f(c,p)}\big(e(p)\big)$. We extend $f$ to a smooth function on $U\times\closedtwoball$ such that, for all $p\in\closedtwoball$, $f(c_0,p)=0$.

For all $(k,\alpha)$, let $C^{k,\alpha}:=C^{k,\alpha}(\closedtwoball)$ denote the H\"older space of functions over $\closedtwoball$ with $k$'th derivatives of H\"older class $\alpha$, let $C^{k,\alpha}_0$ denote the closed subspace consisting of those functions which vanish along the boundary, and let $\mathcal{V}_0^{k,\alpha}$ denote the open subset of $C_0^{k,\alpha}$ consisting of those functions which take values in $(-\eps,\eps)$. Define
\begin{equation*}
F:U\times\closedtwoball\times\mathcal{V}_0^{k+2,\alpha}\rightarrow\closedtwoball\times(-2\eps,2\eps)
\end{equation*}
by
\begin{equation*}
F(c,p,g) := \big(p,f(p,c) + g(p)\big)\ .
\end{equation*}

We aim to construct a function $u:U\times\closedtwoball\rightarrow(-2\eps,2\eps)$ such that, for all $c\in U$, $F(c,\cdot,u(c,\cdot))$ parametrizes the disk $\opD(c)$. To this end, we will first apply the implicit function theorem over $C^{2,\alpha}$ to construct a function $u$, which depends smoothly on $c$, but only $C^{2,\alpha}$ on $p$, such that, for all $c$, $F(c,\cdot,u(c,\cdot))$ parametrizes some $(\Phi,k)$-surface. It will then follow by elliptic regularity that $u$ depends smoothly on all its arguments and by uniqueness that, for all $c$, $F(c,\cdot,u(c,\cdot))$ parametrizes $\opD(c)$.

For all $(c,g)$, the function $F_{c,g}:=F(c,\cdot,g)$ is a smooth immersion. We thus define
\begin{equation*}
\mathcal{F}:U\times\mathcal{V}_0^{2,\alpha}\rightarrow C^{0,\alpha}
\end{equation*}
such that
\begin{equation*}
\mathcal{F}(c,g) := \Phi\circ\mathcal{J}^2F_{c,g}\ ,
\end{equation*}
where, we recall, $\mathcal{J}^2F_{c,g}$ denotes the $2$-jet of the immersion $F_{c,g}$. Since $\mathcal{F}$ is obtained by a finite combination of differentiations and compositions with smooth functions, it defines a smooth map between open subsets of Banach spaces (see, for example, \cite{delaLlaveObaya1999,BourdaudLanzaDeCristoforis2002}, or \cite[Appendix A]{Smith2024AsymptoticGeometry}). Moreover, since the value of $\mathcal{F}(c,g)$ at a point depends only on the $2$-jet of $g$ at that point, it is a second-order, non-linear partial differential operator. Let $L:=\dd_2\mathcal{F}(c_0,0)$ denote its partial derivative with respect to the second component at $(c_0,0)$.

We now compute $L$. Let $N$ denote the unit normal vector field along $D_0$. If $\xi$ were everywhere normal to $D_0$, then by definition of the $\Phi$-Jacobi operator, we would have, for all $h$,
\begin{equation}
\label{eqn:FormulaForDerivative}
Lh = \dd_2\mathcal{F}(c_0,0)\cdot h = J^\Phi\bigl(\langle \xi,N\rangle h\bigr)\ .
\end{equation}
On the other hand, if $\xi$ were tangent to $D_0$, then, for any $h$, the variation $F(c_0,\cdot,sh)$ would be, up to first order in $s$, a reparametrization of $D_0$. Hence,
\begin{equation*}
\frac{d}{ds}\Phi\circ \mathcal{J}^2F_{c_0,sh}\bigg|_{s=0} = 0\ .
\end{equation*}
Since any vector field $\xi$ decomposes into normal and tangential components, it follows by linearity of the differential that \eqref{eqn:FormulaForDerivative} holds in general.

Since the function $\langle \xi,N\rangle$ is everywhere strictly positive, it follows that $Lh = J^\Phi(\langle \xi,N\rangle h)$ is invertible whenever $J^\Phi$ is. However, by hypothesis, $J^\Phi$ is strictly stable, and thus has trivial kernel. By Fredholm's alternative (see \cite[Theorem 6.15]{Gilbarg-Trudinger-01}), $J^\Phi$ defines a surjective map from $C^{2,\alpha}_0$ to $C^{0,\alpha}$. By the closed graph theorem, this map is invertible, and therefore so too is $L:C^{2,\alpha}_0\to C^{0,\alpha}$. It now follows by the implicit function theorem for smooth functions over Banach manifolds that, upon reducing $U$ if necessary, there exists a smooth function $u:U\rightarrow\mathcal{V}_0^{2,\alpha}$ such that, for all $c\in U$,
\begin{equation*}
\mathcal{F}\big(c,u(c)\big) = k\ .
\end{equation*}
We denote $u(c,p):=u(c)(p)$ and we observe that, by elliptic regularity, this function is smooth. Finally, by the uniqueness assumption of Hypothesis (ii), the graph of $f(c,\bullet)+u(c,\bullet)$ is the unique solution $\opD(c)$ of the $(\Phi,k)$-Plateau problem with boundary $c$.

We now define the function $\Phi:U\times\closedtwoball\times(-2\eps,2\eps)\rightarrow\Bbb{R}$ by
\begin{equation*}
\Phi(c,p,t) := f(c,p) + u(c,p) - t\ .
\end{equation*}
This function is trivially a submersion, and its zero set is precisely the intersection of $\MD$ with $U\times\Omega$. It follows that $\MD$ is a smoothly embedded hypersurface of $X\times\cC$, and this proves (i). Finally, this intersection is explicitly parametrized by
\begin{equation*}
(c,p)\in U\times\closedtwoball\mapsto\big(\phi_{f(p,c)+u(p,c)}(p),c\big)\in X\times\cC\ .
\end{equation*}
In particular, $\pi_2$ restricts to a trivial fibre bundle over $U$ with fibre $\closedtwoball$, and this proves (ii).
\end{proof}

The remainder of this section is devoted to proving Item (iii) of Theorem~\ref{thm:StructureOfdiscSpace}, that is, that $N:\MD\to SX^\bullet$ is a smooth diffeomorphism. Observe that it suffices to prove the following properties.
\begin{enumerate}[label=(\roman*)]
\item The target space $SX^\bullet$ is simply connected.
\item The Gauss map $N$ is smooth and proper.
\item The Gauss map is a local diffeomorphism.
\end{enumerate}
Indeed, Properties (ii) and (iii) imply that $N$ is covering map. Since, by Property (i), $SX^\bullet$ is its own universal cover, it then follows that $N$ is a diffeomorphism, as desired. The proof of Property (i) is trivial. Indeed, since $X$ is a closed $3$-ball and since $\mathbb{S}^2$ is simply connected, the unit tangent bundle $SX$ is also simply connected. Since the two Gauss lifts of $\partial X$ have codimension $3$ in $SX$, they do not contribute to the fundamental group. That is, any homotopy of curves in $SX$ passing through one or other Gauss lift of $\partial X$ can be perturbed off this set. This proves Property (i), namely simple connectedness of $SX^\bullet$. Property (ii), namely properness of $N:\MD\rightarrow SX^\bullet$, is an immediate consequence of Corollary \ref{cor:LimitsOfSmallCircles}. It thus remains only to prove Property (iii), namely that $N$ is a local diffeomorphism. We first make the following observation.

\begin{remark}\label{rem_vector_to_Dirichlet}
Let $(c_t)_{t\in(-\eps,\eps)}$ be a smooth variation of curves in $\cC$ centred on $c_0$. As in \S\ref{sss.First_order_variations}, up to first order in $t$, $c_t$ has a parametrization of the form
\begin{equation*}
\theta\mapsto\exp_{c_0(\theta)}\big(t f(\theta)\,n_{c_0}(\theta)\big)\ ,
\end{equation*}
where $\exp$ denotes the exponential map of $\SS^2$, $f\in\cT$ is a first-order trigonometric polynomial, and $n_{c_0}$ is the unit normal vector field over $c_0$ in $\partial X$ compatible with the orientation.

As in the proof of Lemma~\ref{lemma:MDIsSubmanifold}, the corresponding normal variation of $\opD(c_s)$ is parametrized by
\begin{equation}
\label{equation:sortie}
\opD(c_0) \ni p\mapsto\exp_p\big(s\hat{f}(p)\,N(p)\big) \in X\ ,
\end{equation}
where $N$ denotes the unit normal vector field over $\opD(c_0)$ compatible with the orientation, and $\hat{f}$ denotes the unique solution of the Dirichlet problem
\begin{equation*}
J^\Phi\hat{f}=0\qquad\text{and}\qquad\hat{f}|_{\mathbb{S}^1}=\langle N,n_{c_0}\rangle\,hf\ .
\end{equation*}
Note that \eqref{equation:sortie} is not necessarily well-defined for all $p \in \opD(c_0)$ since the image of $p$ may lie outside of $X$. This is however readily addressed by working in a slightly larger manifold $(X_e,g_e)$ containing $(X,g)$ in its interior. In what follows, we will apply the results of \S\ref{ss:TwoDSetting} with the operator $L:=J^\Phi$ and the space of boundary data
\begin{equation*}
E:=\big\{\langle N,n_{c_0} \rangle h_{c_0} f:f\in\cT\big\}\ .
\end{equation*}
\end{remark}

\subsubsection{The Gauss map is a local diffeomorphism}
Recall that $X^\circ$ denotes the interior of $X$ and that $S(X^\circ)$ denotes the unit tangent bundle of this interior.

\begin{lemma}
\label{lemma:LocalDiffeomorphismI}
If, for all $c$, $\opD(c)$ is strictly stable in the sense of Definition \ref{defn:StableSolutions}, then every point of $S(X^\circ)$ is a regular value of the Gauss map $N:\MD\to SX^\bullet$.
\end{lemma}

\begin{proof}
Choose $(p_0,c_0)\in\MD$ such that $N(p_0,c_0)\in S(X^\circ)$. Denote $D_0:=D(c_0)$, let $N_0$ denote the unit normal vector field over $D_0$, and let $J_0$ denote the $\Phi$-Jacobi operator of $D_0$. By hypothesis, $J_0$ is strictly stable in the sense of Definition \ref{defn:Stability}. Furthermore, since $N(p_0,c_0)$ is an element of $S (X^\circ)$, $p_0$ is an interior point of $D_0$. Let $v$ be a tangent vector of $\cC$ at $c_0$, let $(c_t)_{t\in]-\delta,\delta[}$ be a smooth curve in $\cC$ passing through $c_0$ with derivative $v$ at this point, and, for all $t$, denote $D_t:=\opD(c_t)$.

We first show that $\dd(\pi\circ N)$ is surjective at $(p_0,c_0)$, were $\pi:SX\rightarrow X$ denote the base-point projection. To this end, consider the first order approximation of $D_t$ discussed in Remark \ref{rem_vector_to_Dirichlet}, and note that the curve $t\mapsto \exp_{p_0}\big(t\hat{f}(p_0)\,N_0(p_0)\big)$ is tangent at $t=0$ to a smooth curve $t\mapsto p_t$ in $X$ such that $(p_t,c_t)\in\MD$ for all $t$. Since
\begin{equation*}
\partial_t p_t|_{t=0} = \hat{f}(p_0)\,N_0(p_0)\ ,
\end{equation*}
it follows that
\begin{equation*}
\hat{f}(p_0)N_0(p_0)\in\OpIm\big(\dd(\pi\circ N)(p_0,c_0)\big)\ .
\end{equation*}

Suppose now that $v$ is chosen in such a way that the corresponding first-order variation $f$ of the boundary curve is an element of $E^-$. When this holds, $f$ is nowhere vanishing, and, by strict stability and Lemma \ref{lemma:WeakMP}, $\hat{f}$ is also nowhere vanishing, so that
\begin{equation*}
N_0(p_0)\in\OpIm\big(\dd(\pi\circ N)(p_0,c_0)\big)\ .
\end{equation*}
On the other hand, by varying $p$ whilst keeping $c_0$ fixed, we see that
\begin{equation*}
T_{p_0}D_0 \subseteq\OpIm\big(\dd(\pi\circ N)(p_0,c_0)\big)\ ,
\end{equation*}
from which it follows that $\dd(\pi\circ N)$ is surjective, as asserted. In other words,
\begin{equation*}
\OpIm\big(\dd(\pi\circ N)(p_0,c_0)\big)=T_{p_0}X\ .
\end{equation*}

It remains only to show that the vertical fibre of $SX$ at $N(p_0,c_0)$ is contained in the image of $\dd N(p_0,c_0)$. Note that this fibre identifies with $N(p_0,c_0)^\perp=T_{p_0}D_0$. Let $\xi$ be a unit vector in $T_{p_0}D_0$. By Theorem \ref{thm:coveringofunitbundle}, Item (iv), there exists $f\in E^+$ such that $p_0\in Z_f$ and $\nabla\hat f(p_0)$ is colinear with $\xi$. Let $D_t$ be a variation of disks associated to $f$. As before, these discs are, up to first order, given by the normal graphs of functions of the form
\begin{equation*}
p\mapsto\exp_p(t\hat f(p)N_0(p))\ .
\end{equation*}
Let $\tilde N_t$ denote the unit normal vector fields along these normal graphs. Since $\hat f(p_0)=0$, for all $t$, $\exp_{p_0}\big(t\hat f(p_0)N_0(p_0)\big)=p_0$, and
\begin{equation*}
\tilde {N}_t(p_0)=N_0(p_0)-t\nabla\hat f(p_0)+o(t)\ .
\end{equation*}
There therefore exists a smooth curve $t\mapsto p_t$ in $X$ such that $(p_t,c_t)\in\MD$ for all $t$,
\begin{equation*}
\partial_t p_t|_{t=0}=\hat f(p_0)N_0(p_0)=0\ ,
\end{equation*}
and
\begin{equation*}
\partial_t N(p_t,c_t)|_{t=0}=-\nabla\hat f(p_0)\ .
\end{equation*}
Thus
\begin{equation*}
-\nabla\hat f(p_0)\in\text{Im}\big(\dd N(p_0,c_0)\big)\ .
\end{equation*}
Since $\xi$ is colinear with $\nabla\hat f(p_0)$, it also belongs to $\text{Im}\big(\dd N(p_0,c_0)\big)$. Since $\xi$ is arbitrary, it follows that the whole vertical fibre at $N(p_0,c_0)$ is included inside $\text{Im}\big(\dd N(p_0,c_0)\big)$. We conclude that $\dd N(p_0,c_0)$ is surjective, and $(p_0,c_0)$ is thus a regular point of $N:\MD\to SX^\bullet$, as desired.
\end{proof}

We now consider the case of vectors based at boundary points, which are not orthogonal to the boundary, that is $S(\partial X)^\bullet$.

\begin{lemma}
\label{lemma:LocalDiffeomorphismII}
If, for all $c$, $\opD(c)$ is strictly stable in the sense of Definition \ref{defn:StableSolutions}, then every point of $S(\partial X)^\bullet$ is a regular value of $N$.
\end{lemma}

\begin{proof} We continue to use the construction and notations of the proof of Lemma \ref{lemma:LocalDiffeomorphismI}. Assume that $N(p_0,c_0)\in S(\partial X)^\bullet$ so in particular $p_0\in\partial X$, and note that
\begin{equation*}
T_{p_0}X \subseteq T_{p_0}(\partial X)+T_{p_0}D_0 \subseteq \OpIm\big(\dd(\pi\circ N)(p_0,c_0)\big)\ ,
\end{equation*}
so that $\dd(\pi\circ N)$ is surjective at $(p_0,c_0)$.

We now show that the vertical fibre of $TSX$ at $N(p_0,c_0)$ is contained in the image of $\dd N(p_0,c_0)$. Indeed, let $\xi$ be a non-vanishing, unit vector tangent to $D_0$ at $p_0$. By linearity, it suffices to address the cases where $\xi$ is normal to $c_0$ and where $\xi$ is tangent to $c_0$. In the former case, we choose a variation of circles such that the corresponding $f\in E$ vanishes only at $p_0$, and is positive at every other point. By Lemma \ref{lemma:ZeroOrOneZero}, Item (ii), $\hat{f}$ is positive over $D_0$, and by Lemma \ref{lemma:BoundaryWeakMP}, $\nabla \hat{f}(p_0)$ is non-vanishing and colinear with $\xi$. Repeating the argument of the previous lemma yields a smooth curve $(p_t,c_t)$ such that  $\partial_t p_t|_{t=0}=0$ and
\begin{equation*}
\partial_t\big(N(p_t,c_t)\big)\big|_{t=0} = - \nabla\hat f(p_0)\ .
\end{equation*}
Since $\nabla\hat f(p_0)$ is colinear with $\xi$, it follows that $\xi\in\text{Im}(\dd N(p_0,c_0))$, as desired.

Finally, suppose that $\xi$ is tangent to $c_0$. We choose a variation of circles such that the corresponding $f\in E$ vanishes at $p_0$, and also at some other boundary point $q_0$, say. By hypothesis, the restriction of $\hat f$ to $c_0$ has a non-degenerate zero at $p_0$, so that $\nabla \hat f(p_0)=\lambda\xi+\mu$, for some non-zero $\lambda$, and for some $\mu$ normal to $c_0$ at $p_0$, and it follows as before that $\lambda\xi+\mu$ lies in the image of $\dd N(p_0,c_0)$. Since we have already shown that $\mu$ lies in the image of $\dd N(p_0,c_0)$, it follows that so too does $\xi$, and this completes the proof.
\end{proof}

Lemmas \ref{lemma:LocalDiffeomorphismI} and \ref{lemma:LocalDiffeomorphismII} together establish Property (iii) given above. It follows that $N$ is a smooth diffeomorphism, and this completes the proof of Theorem~\ref{thm:StructureOfdiscSpace}.

\subsection{The global Bolker condition}  \label{ssection:bolker}
We now state and prove the \emph{global Bolker condition}, which ensures that the Schwarz kernel of the normal operator of the Radon transform is smooth away from the diagonal (see \S\ref{s.Radon}). Given $p\in X^\circ$, we define
\begin{equation}
\label{equation:coiffeur}
\begin{aligned}
\Sigma(p) &:= \{ c\in\cC\ :\ p\in\opD(c)\}\ ,\ \text{and}\\
\hat{\Sigma}(p) &:= \{ (q,c)\in\MD\ :\ p,q\in\opD(c)\}\ .
\end{aligned}
\end{equation}

\begin{lemma}
\label{lemma:BolkerI}
For all $p\in X^\circ$,
\begin{enumerate}[label=\emph{(\roman*)}]
\item $\Sigma(p)$ is a smoothly embedded, $2$-dimensional submanifold of $\cC$ diffeomorphic to $\mathbb{S}^2$; and
\item $\hat{\Sigma}(p)$ is a smoothly embedded, $4$-dimensional submanifold of $\MD$ diffeomorphic to the unit ball in $T\mathbb{S}^2$.
\end{enumerate}
\end{lemma}

\begin{proof} Let $\pi_1:\MD\rightarrow X$ and $\pi_2:\MD\rightarrow\cC$ be as in \eqref{equation:diagram_MD}. Let $\pi:SX\rightarrow X$ denote the base-point projection, and note that $\pi_1=\pi\circ N$. Since $N$ is a smooth diffeomorphism, it follows that $A(p):=\pi_1^{-1}(\{p\})$ is a smoothly embedded submanifold of $\MD$ diffeomorphic to $\mathbb{S}^2$.

Choose $c\in\cC$, and consider the fibre $B(c):=\pi_2^{-1}(\{c\})$. By definition, the restriction of $\pi_1$ to this fibre is a smooth embedding onto the disk $\opD(c)$. Consequently, for all $p\in\opD(c)$,
\begin{equation*}
\begin{aligned}
\#\big(\pi_2|_{A(p)}\big)^{-1}(\{c\})& =\#\big(A(p)\cap\pi_2^{-1}(\{c\})\big)\\
                             & =\#\big(\pi_1^{-1}(\{p\})\cap B(c)\big)\\
                             & =\#\big(\pi_1|_{B(c)}\big)^{-1}(\{p\}) \in \{0,1\}\ .
\end{aligned}
\end{equation*}
Likewise, at $(p,c)$,
\begin{equation*}
\begin{aligned}
\ker\big(\dd\pi_2|_{TA(p)}\big) &=TA(p)\cap\ker(\dd\pi_2)\\
                        &=\ker(\dd\pi_1)\cap TB(c)\\
                        &=\ker\big(\dd\pi_1|_{TB(c)}\big) = \{0\}\ .
\end{aligned}
\end{equation*}
Since $c\in\cC$ is arbitrary, it follows that $\pi_2$ restricts to a smooth embedding from $A(p)$ to $\cC$, and since $\Sigma(p)=\pi_2(A(p))$, this proves (i). Finally, since $\pi_2$ is a submersion with fibre $\twoball$, and since $\hat{\Sigma}(p)=\pi_2^{-1}(\Sigma(p))$, (ii) follows, and this completes the proof.
\end{proof}

This yields the \emph{global Bolker condition}.
\begin{theorem}[Global Bolker condition]
\label{theorem:bolker}
For all $p\in X^\circ$, the map
\begin{equation*}
\pi_1 : \hat{\Sigma}(p)\setminus\pi_1^{-1}(\{p\}) \to X \setminus \{p\}
\end{equation*}
is a surjective submersion.
\end{theorem}

\begin{proof}
We first prove the submersion property. Choose $(q_0,c_0)\in\hat{\Sigma}(p)\setminus\pi_1^{-1}(\{p\})$, denote $D_0:=\opD(c_0)$, and let $J_0$ denote its $\Phi$-Jacobi operator. Let $v$ be a tangent vector of $\cC$ at $c_0$, let $(c_t)_{t\in(-\eps,\eps)}$ be a smooth variation of $c_0$ in $\cC$ which is tangent to $v$, and, for all $t$, denote $D_t:=\opD(c_t)$. Recall from Remark \ref{rem_vector_to_Dirichlet} that the normal variation of $(D_t)_{t\in(-\eps,\eps)}$ at $D_0$ is coincides up to first order with the graph of $\hat{f}N_0$, where $\hat{f}$ solves the Dirichlet problem for $J_0$ with boundary value $f\in E$ corresponding to $v$.

The vector $(0,v)$ is tangent to $\hat{\Sigma}(p)$ whenever $\hat f(p)=0$. When this holds,
\begin{equation*}
\hat f(q_0)N_0(q_0)\in \dd \pi_1\big(T_{(q_0,c_0)}\hat{\Sigma}(p)\big)\ .
\end{equation*}
Since $J_0$ is stable, by Theorem \ref{thm:coveringofunitbundle}, Item (iv), $v$ may be chosen such that $\hat f(q_0)\neq 0$, so that
\begin{equation*}
N_0(q_0)\in \dd\pi_1\big(T_{(q_0,c_0)}\hat{\Sigma}(p)\big)\ .
\end{equation*}
On the other hand, by varying $q$ whilst keeping $c_0$ fixed, we see that
\begin{equation*}
T_{q_0}D_0 \subseteq \dd\pi_1\big(T_{(q_0,c_0)}\hat{\Sigma}(p)\big)\ .
\end{equation*}
The restriction of $\dd\pi_1(q_0,c_0)$ to $T\hat{\Sigma}(p)$ is thus surjective. Since $(q_0,c_0)$ is arbitrary, it follows that $\pi_1$ restricts to a submersion over $\hat{\Sigma}(p)$, as desired.

We now use a connectedness argument to prove surjectivity. Indeed, note first that, by the submersion property, the image $\pi_1\big(\hat{\Sigma}(p)\setminus\pi_1^{-1}(\{p\})\big)$ is open in $X \setminus \{p\}$. We now show that this image is closed in $X\setminus\{p\}$. Let $(z_n)_{n\in\mathbb{N}}$ be a sequence of points in $\hat{\Sigma}(p)$. For all $n$, denote $y_n := \pi_1(z_n)$, and suppose that the sequence $(y_n)_{n\in\mathbb{N}}$ converges to some limit $y_\infty$, say, in $X\setminus\{p\}$. Recall that $\pi_1=\pi\circ N$, where $\pi:SX\rightarrow X$ denotes the base-point projection. Since $\pi^{-1}(\{p\})$ is compact, it follows by properness of $N$ that $(z_n)_{n\in\mathbb{N}}$ subconverges to some limit $z_\infty$, say, in $\MD$ so that, by continuity, $y_\infty=\pi_1(z_\infty)\in\pi_1\big(\hat{\Sigma}(p)\setminus\pi_1^{-1}(\{p\})\big)$. This proves closedness, and it follows by connectedness of $X\setminus\{p\}$ that the restriction of $\pi_1$ to $\hat{\Sigma}(p)$ is surjective, as desired.
\end{proof}

For all $p\neq q\in X^\circ$, denote
\begin{equation*}
\Sigma(p,q) := \{ c\in\cC\ :\ p,q\in\opD(c) \}\ .
\end{equation*}

\begin{lemma}
\label{lemma:circle}
For all $p\neq q\in X^\circ$, $\Sigma(p,q)$ is an embedded, simple closed curve in $\cC$.
\end{lemma}
\begin{proof}
By Theorem \ref{theorem:bolker}, $\pi_1|_{\hat{\Sigma}(p)}$ is a submersion over $\hat{\Sigma}(p)\setminus\pi_1^{-1}(\{p\})$. The preimage
$A(p,q)=(\pi_1|_{\hat{\Sigma}(p)})^{-1}(\{q\})$ is thus a smoothly embedded $1$-dimensional submanifold of $\hat{\Sigma}(p)$, and since $\hat{\Sigma}(p)$ is compact, it is also compact.

It follows as in the proof of Lemma \ref{lemma:BolkerI}, that $\Sigma(p,q):=\pi_2(A(p,q))$ is a compact, smoothly embedded $1$-dimensional submanifold of $\cC$. Finally, upon continuously varying $p$ and $q$ to two distinct points of $\partial X$, we see that $\Sigma(p,q)$ has a single connected component, and is thus an embedded, simple closed curve, as desired.
\end{proof}

\subsection{Analytic regularity of the double fibration} \label{ssection:analytic}

In this subsection we assume that $X$ is a real-analytic ball with analytic boundary, that $\bar g$ is real analytic and extends analytically across
$\partial X$, and that $\mathcal C$ is an analytic family of circles on $\partial X$. We also assume that the curvature function $\Phi$ is real analytic on $\Lambda$.

\begin{theorem}[Analyticity for the foliated Plateau problem]
\label{thm:analytic-foliated-plateau}
Under the above hypotheses, the marked-disc space $\mathrm{MD}$ is a real-analytic manifold, and the Gauss map
\begin{equation*}
N:\mathrm{MD}\to SX^\bullet
\end{equation*}
is a real-analytic diffeomorphism. In particular, the projection
\begin{equation*}
\pi_R:SX^\bullet\to\mathcal C
\end{equation*}
and the foliation of $SX^\bullet$ by Gauss lifts of $(\Phi,k)$-discs are real analytic. Moreover, the Bolker maps appearing in \S\ref{ssection:bolker} are real analytic.
\end{theorem}

We will use without further comment the analytic implicit-function theorem for Banach spaces, together with the analytic regularity theorem of Morrey for analytic nonlinear elliptic systems, including the boundary regularity version. We refer the reader to \cite{Morrey-58-1,Morrey-58-2}, as well as \cite{Wang-06} for a formulation of the problem with boundary.

\begin{proof}
The proof is the analytic analogue of the smooth construction of Lemma~\ref{lemma:MDIsSubmanifold}. We briefly recall the argument. Fix $c_0\in\mathcal C$ and write $D_0=\mathrm D(c_0)$. Since the metric, the boundary, and the family $\mathcal C$ are all analytic, the tubular coordinates used near $D_0$ may be chosen analytically, and that nearby boundary curves $c\in\mathcal C$ may be written as analytic graphs over $c_0$, also depending analytically on $c$.

In these coordinates, nearby discs are expressed as graphs $p\mapsto F(p,c,g)$, with $g\in C^{\ell+2,\alpha}_0(\closedtwoball)$ solving the implicit equation $\mathcal F(c,g):=\Phi\bigl(\mathcal J^2F_{c,g}\bigr)-k=0$. For $\ell\geq0$ and $0<\alpha<1$, the map $\mathcal F:U\times C^{\ell+2,\alpha}_0(\closedtwoball) \longrightarrow C^{\ell,\alpha}(\closedtwoball)$ is real analytic. Its derivative with respect to the second variable at $(c_0,0)$ is $D_g\mathcal F(c_0,0)h=J^\Phi\bigl(\langle \xi,N_0\rangle h\bigr)$, where $N_0$ denotes the unit normal of $D_0$ and $\xi$ denotes the transverse vector field used to parametrize the tubular neighbourhood. Since $\langle \xi,N_0\rangle>0$ and $D_0$ is strictly stable, this Dirichlet operator is an isomorphism $C^{\ell+2,\alpha}_0(\closedtwoball)\rightarrow C^{\ell,\alpha}(\closedtwoball)$. The analytic implicit-function theorem therefore yields, after reducing $U$ if necessary, a real-analytic map $c\mapsto g(c)\in C^{\ell+2,\alpha}_0(\closedtwoball)$ whose graph is the unique solution of the Plateau problem with boundary $c$.

By elliptic regularity for analytic nonlinear elliptic boundary value problems, these solutions are real analytic up to the boundary, and the above local parametrizations are analytic with respect to both the point of the disc and the parameter $c$. It follows that $\mathrm{MD}$ is a real-analytic submanifold of $X\times\mathcal C$, and that the projections $\mathrm{MD}\to X$ and $\mathrm{MD}\to\mathcal C$ are analytic.

The Gauss map $N:\mathrm{MD}\to SX^\bullet$ is obtained from the analytic parametrizations by first taking derivatives and normalizing the resulting normal vector. It is therefore also analytic. Since we have already shown that $N$ is a global diffeomorphism, it is also an analytic diffeomorphism. It follows that $\pi_R=\pi_2\circ N^{-1}:SX^\bullet\to\mathcal C$ is analytic, and that the leaves $\pi_R^{-1}(c)$ form a real-analytic foliation. The maps used in the proof of the Bolker condition are obtained from the same analytic maps by restriction and composition, and are therefore analytic as well. This completes the proof.
\end{proof}

\subsection{Extending the foliation}

\label{ssection:extended-foliation}

We now describe the smooth extension of the foliation of $SX^\bullet$ by $(\Phi,k)$-discs across the boundary $\partial X$. A preliminary step is to compactify the space of circles.

\subsubsection{Compactification of the space of circles}

\label{ssection:compactification-circles}

The coordinates $(p,s,\pm) \in \partial X \times (0,s_0)$ introduced in \S\ref{ssection:smooth-coordinates} are smooth (resp. real analytic) coordinates on an open subset of $\cC$. The $\pm$ here refers to the two possible orientations of the corresponding boundary circles in $\partial X$. This leads to the following definition.

\begin{defi}[Compactification of the space of circles]
\label{definition:compactification}
The compactified space of circles is defined by
\begin{equation*}
\overline{\cC} := \cC \sqcup \partial \cC_+ \sqcup \partial \cC_-\ , \qquad \text{where } \partial \cC_\pm := \big\{(p,s=0,\pm) ~:~ p \in \partial X\big\}\ .
\end{equation*}
It is endowed with the smooth (resp. analytic) structure induced by the set of smooth (resp. analytic) functions with respect the $(p,s) \in \Ss^2 \times [0,s_0)$ variables.
\end{defi}

In other words, $\overline{\cC}$ is the compactification of $\cC$ obtained by adding the slice $\{s=0\}$ corresponding to two disjoint copies of $\Ss^2$. Note that $\overline{\cC}$ is a smooth (resp. real analytic) manifold which is diffeomorphic to $\Ss^2 \times [-1,1]$. By construction, its boundary $\partial \cC := \overline{\cC} \setminus \cC$ is equal to the union
\begin{equation*}
\partial \cC = \partial \cC_+ \sqcup \partial \cC_-\ .
\end{equation*}

\subsubsection{Smooth and analytic extension} The following theorem asserts that the foliation by $(\Phi,k)$-discs can be smoothly (resp. analytically) extended across the boundary $\partial X$. It relies on the asymptotic analysis developed in \S\ref{section:asymptotic} and will play a key role in the verification of several analytic properties of the Radon transform that will be carried out in \S\ref{s.Radon}. Recall the notion of proper extensions from Appendix \ref{section:extension}.

\begin{theorem}
\label{theorem:extension-foliation}
There exists a proper smooth (resp. analytic) Riemannian extension $(X_e, \bar g_e)$ of $(X, \bar g)$, a proper extension $\cC_e$ of $\overline{\cC}$, and an open subset $U \subset SX_e$ such that $SX \subset U$ and
\begin{equation}
\label{equation:diagram2}
\begin{tikzcd}
& U \arrow[rd, "\pi_R"] \arrow[ld,"\pi"'] & \\
X_e & & \cC_e
\end{tikzcd}
\end{equation}
is a double fibration satisfying the Bolker condition.

Furthermore, it coincides with the double fibration \eqref{equation:diagram} when restricted to $SX^\bullet \subset SX_e$.
\end{theorem}

Observe that Theorem \ref{th.Foliated_Plateau_posta}, item (ii), is an immediate consequence of the previous statement.

\begin{proof} \emph{The analytic case.}
We first address the analytic case which is slightly easier. Upon composing with an analytic diffeomorphism $\psi : X \to \BB^3$ such that $\psi|_{\partial X} = \alpha$, we may further assume that $X = \BB^3 = \{x \in \R^3 ~:~ |x| \leq 1\}$, and that $\bar g$ is analytic over $\BB^3$. Denote $X'_e = \{x \in \R^3 ~:~ |x| \leq 1+\eps'\}$ for $\eps' > 0$, and let $\bar g_e$ denote the analytic extension to $X'_e$ which exists and is unique provided $\eps' > 0$ is small enough. The space of circles $\cC_X$ then coincides with $\cC_{\Ss^2}$. As in \S\ref{ssection:compactification-circles}, we use the coordinates $(p,s,\pm)$ with $p \in \partial X = \Ss^2, s \in [0,s_0)$ near the boundary of the compactified space of circles $\bar\cC$. We let $\cC_e := \bar \cC \sqcup \{(p,s \pm) ~:~p \in \Ss^2, s \in (-s_0,s_0)\}$.

We first address small discs. Let $\mathrm{D}^\pm(p,s)$ denote the unique $(\Phi,k)$-disc bounding the circle $c$ corresponding to $(p,s,\pm)$. Recall from Theorem \ref{thm:SmallCircles} that this disc can be written as the graph
\begin{equation*}
D_{p,s}^\pm :=\big\{ \exp_q\big(\mathbf{f}^\pm(s,p,q)\cdot\nu(q)\big)\ :\ d(p,q)\leq \sqrt{s}\big\}\ .
\end{equation*}
In addition, the functions $\mathbf{f}^\pm$ extend analytically from $\{p \in \partial X, s > 0, q \in B_{\sqrt{s}}(p)\}$ to an open neighbourhood of $\{s=0\}$. Upon reducing $s_0$ if necessary, we thus obtain an analytic family of analytic discs, given for $p \in \partial X, s \in (-s_0,s_0)$ by
\begin{equation*}
\mathrm{D}^\pm(p,s) = \big\{\exp_q\big(\mathbf{f}^\pm(s,p,q)\cdot\nu(q)\big)\ :\ q \in B_{s_0}(p)\big\} \subset X'_e
\end{equation*}
which coincides with the original family of discs in $X$ for $s \geq 0$. By analyticity, the $(\Phi,k)$-surface equation is still verified by the discs $\mathrm{D}^\pm(p,s)$ on $X'_e$, where all relevant geometric objects are computed with respect to the extended metric $\bar g_e$. That is, the discs $\mathrm{D}^\pm(p,s)$ are $(\Phi,k)$-surfaces in $X'_e$.

Observe that $X'_e$ may not be covered by $\mathrm{D}^\pm(p,s)$ if $s_0$ is too small compared with $\eps'$. That is, there may exist points $x \in X_e'$ not lying in $\mathrm{D}^\pm(p,s)$ for any $(p,s)$. However, this issue is addressed by considering the smaller analytic manifold $X_e = \{x \in \R^3 ~:~ |x| \leq 1+\eps\}$ for $\eps > 0$ sufficiently small, and the new family of discs $\mathrm{D}^\pm(p,s) \cap X_e \subset X_e$ for $(p,s) \in \cC_e$.

Away from $\partial \cC$, each disc $\mathrm{D}_{\bar g}(c) \subset X$ is an analytic submanifold intersecting $\partial X$ transversally. It thus uniquely extends to a smooth analytic submanifold of $X_e$ which satisfies the $(\Phi,k)$-surface equation. In addition, as in the proof of Theorem \ref{thm:analytic-foliated-plateau}, the dependence with respect to $c \in \cC$ is analytic.

Finally, we let $U \subset SX_e$ denote the open set consisting of all the Gauss lifts of the $(\Phi,k)$-discs for $c \in \cC_e$. We claim that \eqref{equation:diagram2} is a double fibration. Indeed, the Gauss parametrization $(p,s,q)\mapsto(x,N)$ from the extended space of marked discs $\mathrm{MD}_e$ to $SX_e$ has invertible differential at $s=0, q=p$: using \eqref{equation:graal}, the variation of $q$ supplies the two boundary directions, the variation of $s$ supplies the transverse position direction, and the variation of $p$ supplies the two vertical direction coordinates (corresponding to $x$ being fixed, and $N$ varying in $S_xX$), with nondegeneracy following from $\lambda^\pm(p)>0$. Together with the corresponding invertibility for noncollapsed discs, compactness allows us to shrink the extensions so that the Gauss parametrization is a diffeomorphism onto a neighborhood $U$ of $SX$.

Since $\pi_R$ is a submersion from the space of marked discs, it becomes immediate that $\pi_R : U \to \cC_e$ is a submersion. The map $\pi : U \to X_e$ is surjective by construction of $X_e$. To show that it is a submersion, it suffices to show that $d\pi$ is surjective over $\{(x,\pm\nu(x))\ |\ x \in \partial X\}$ since the range of this map can only increase locally; upon reducing $U$ and $\eps > 0$ if necessary, this then establishes that $d\pi$ is surjective on $U$. However, that $d\pi$ is a submersion on this locus follows from the above argument showing that the Gauss map is a diffeomorphism. Finally, the Bolker condition is still satisfied as it is an open condition: indeed, following \S\ref{ss:TwoDSetting} and \S\ref{ssection:plateau-3d}, the Bolker condition follows from the strict stability of the discs $\mathrm{D}_e(c), c \in \cC_e$; as strict stability holds uniformly for all the discs $\mathrm{D}(c)$, $c \in \cC$ (that is, the first eigenvalue of the Jacobi operator is uniformly bounded from below), the same holds for their extensions. \\

\emph{The smooth case.} The proof is similar, except that the objects do not have a canonical extension to a larger manifold. Indeed, by Theorem \ref{thm:SmallCircles}, the small discs can be smoothly extended, and the other discs as well, upon making arbitrary choices for the extensions. In particular, these do not need to satisfy the $(\Phi,k)$-surface equation anymore outside of $X$ (as opposed to the analytic case). The same proof as in the analytic case then applies.

\end{proof} 

\section{Radon transforms}\label{s.Radon} We are now ready to define and study the Radon transform associated with the foliated Plateau problem. Throughout this section, we assume that $\bar g$ is $(\Phi,k)$-simple (Definition \ref{definition:phik-simplicity}).

\subsection{Definition}\label{ss.def} We begin by introducing the Radon transform on the unit tangent bundle, and then specialize it to functions defined on $X$, or symmetric tensors.

\subsubsection{The Radon transform on $SX$} The \emph{Radon transform} is an integral transform initially defined on functions of $SX$. Given $c \in \cC$, recall that $\mathrm{D}(c) \subset X$ denotes the unique $(\Phi,k)$-disc such that $\partial \mathrm{D}(c) = c$. The surface $\mathrm{D}(c) \subset X$ carries a natural area form $\dd \mathrm{D}(c)$ given by the Riemannian volume of the restriction of the ambient metric $g$ in $X$ to $\mathrm{D}(c)$.

Let $\mathrm{L}(c) \subset \pi^{-1}(\mathrm{D}(c)) \subset SX$ denote the Gauss lift of $\mathrm{D}(c)$, where $\pi : SX \to X$ is the projection. Since $\pi : \mathrm{L}(c) \to \mathrm{D}(c)$ is a diffeomorphism, the area form on $\mathrm{D}(c)$ can be lifted to an area form $\dd \mathrm{L}(c) := \pi^* \dd \mathrm{L}(c)$. The general Radon transform is then defined as
\[
R : C(SX) \to C(\cC), \qquad R f(c) := \int_{\mathrm{L}(c)} f ~\dd \mathrm{L}(c).
\]
It is also easily seen to map $C^\infty(SX)$ into $C^\infty(\cC)$.

%

We use the coordinates $(p,s) \in \partial X \times (0,s_0)$ for small circles on $\cC$ defined in \S\ref{ssection:extended-foliation}. Observe that, for $f \in C^0(SX)$, as $c \to \partial \cC$, $Rf(c) \to 0$ since the size of the disc shrinks. Actually, in the compactified coordinates, $Rf$ vanishes to order $1$ on $\partial \cC$:

\begin{lemma}
There exist $C,s_0 > 0$ such that for all $f \in C^0(SX)$, for all $p \in \partial X, s \in [0,s_0)$:
\[
|Rf(p,s)| \leq C s \|f\|_{C^0(SX)}.
\]
\end{lemma}

\begin{proof}
The proof is straightforward using Theorem \ref{thm:AsymptoticPlateauProblem} and observing that the disc $\mathrm{L}(c)$ for $c \in \cC$ corresponding to the pair $(p,r) \in \partial X \times [0,r_0)$ (where $r^2=s$) has area $\mathcal{O}(r^2) = \mathcal{O}(s)$.
\end{proof}

\subsubsection{Adjoint} \label{sssection:adjoint}

Let $\dd c$ be an arbitrary smooth (resp. analytic) measure on $\cC$ (it does not need to extend smoothly up to $\partial \cC$). Then, there exists a unique smooth (resp. analytic) positive function $a \in C^\infty(SX^\bullet)$ such that the following Fubini-type formula holds for all $f$ with compact support in $SX^\bullet$:
\begin{equation}
\label{equation:fubini}
\int_{SX^\bullet} f(v) \dd \mu(v) = \int_{\cC}\left(\int_{\mathrm{L}(c)} f(v) ~ a(v)^{-1} \dd \mathrm{L}_c(v)\right) \dd c,
\end{equation}
where $\dd\mu$ denotes the Liouville measure on $SX$.

%

We let
\[
R^* : C^\infty_{\mathrm{comp}}(\cC) \to C^\infty(SX)
\]
be the formal adjoint of $R$ such that for all $f \in C^\infty(SX)$, $f' \in C^\infty_{\mathrm{comp}}(\cC)$:
\[
\langle Rf, f' \rangle_{L^2(\cC,\dd c)} = \langle f, R^*f' \rangle_{L^2(SX,\dd\mu)}.
\]
It follows immediately from \eqref{equation:fubini} that $R^* f (v) = a(v) f(\pi_R(v))$, where $\pi_R : SX \to \cC$ is the projection map.

\subsubsection{Radon transform of functions and symmetric tensors} \label{sssection:radon-tensors}

We are mostly interested in applying $R$ to specific functions, such as functions defined on $X$. We introduce
\begin{equation}
\label{equation:radon-transform}
R_0 : C^\infty(X) \to C^\infty(\cC), \qquad R_0 := R \circ \pi^*,
\end{equation}
where $\pi^* : C^\infty(X) \to C^\infty(SX)$ is the pullback operator $\pi^*f = f\circ\pi$.


Another interesting class of geometric operators will appear later on in the context of minimal surfaces. Given $f \in C^\infty(X,S^2 T^*X)$, a symmetric $2$-tensor, we introduce
\[
\pi_\perp^* f \in C^\infty(SX), \qquad \pi_\perp^*f(v) := \tr(f|_{v^\perp}),
\]
and define
\begin{equation}
\label{equation:s2}
R_2 : C^\infty(X,S^2 T^*X) \to C^\infty(\cC), \qquad R_2 := R \circ \pi_\perp^*.
\end{equation}
The operator $R_2$ turns out to be (one half of) the differential of the boundary area spectrum for minimal surfaces, see Lemma \ref{lemma:differential-area}.

\subsubsection{The normal operator} Finally, we can associate with these operators a corresponding \emph{normal operator}. We let
\[
\Pi_0 := R_0^*R_0 = \pi_* \circ R^* R \circ \pi^*,
\]
where $\pi_* : L^2(SX, \dd\mu) \to L^2(X, \dd\vol_X)$ is the adjoint of $\pi^*$. Observe that
\[
\pi_* f (x) = \int_{S_xX} f(x,v) \dd v,
\]
where $\dd v$ is the round measure on $S_x X$. This stems from the Fubini-type formula:
\[
\int_{SX} f(x,v) \dd\mu(x,v) = \int_X \left(\int_{S_xX} f(x,v) \dd v\right) \dd\vol_g(x).
\]
Here, the adjoint is computed with respect to the two natural $L^2$-scalar products $L^2(SX,\dd \mu)$ and $L^2(\cC,\dd c)$. A straightforward computation shows that:
\begin{equation}
\label{equation:r0star-r0}
\Pi_0 f (x) = R_0^*R_0 f(x) = \int_{S_x X} \left(\int_{\mathrm{D}(\pi_R(v))} f ~ \dd\mathrm{D}(\pi_R(v))\right) a(v) \dd v.
\end{equation}
Note that, due to the presence of $a$, the operator is only well-defined on compactly supported functions in $X^\circ$.

Similarly, we may introduce
\[
\Pi_2 := R_2^*R_2 := {\pi_\perp}_* R^* R \pi_\perp^*,
\]
where ${\pi_\perp}_* : L^2(SX, \dd\mu) \to L^2(X,S^2T^*X, \dd\vol_X)$ is the adjoint of $\pi_\perp^*$. Note that $\pi_\perp^* = \pi^* \circ \tr - \pi_2^*$, where $\pi_2^*f(x,v) = f_x(v,v)$. Taking the adjoint, we thus find that
\[
{\pi_\perp}_*f = (\pi_* f) g - {\pi_2}_*f.
\]
One can then derive an explicit expression, similar to \eqref{equation:r0star-r0}, using the previous computation. This will not be needed in what follows.

\subsection{Results} \label{ss.results_radon} We now state the main properties of the normal operator in the interior of $X$.

\subsubsection{Pseudodifferential behaviour} The first result is that the associated normal operators $\Pi_0$ and $\Pi_2$ are pseudodifferential operators in the interior $X^\circ$:

\begin{theorem}[Pseudodifferential behaviour] \label{th_pseudo_behaviour}
Let $(X,\bar g)$ be a smooth Riemannian ball such that $\bar g$ is $(\Phi,k)$-simple. Then the following holds:
\begin{enumerate}[label=\emph{(\roman*)}]
\item The operator $\Pi_0$ is a pseudodifferential operator of order $-2$ on $X^\circ$. It is elliptic with principal symbol $\sigma_{\Pi_0} \in S^{-2}(T^*X)$ given by
\[
\sigma_{\Pi_0}(x,\xi) = \dfrac{4\pi^2}{|\xi|^2}\left(a(x,\xi^\sharp/|\xi|) + a(x,-\xi^\sharp/|\xi|)\right), \qquad \xi \neq 0
\]
\item The operator $\Pi_2$ is a pseudodifferential operator of order $-2$ on $X^\circ$ with principal symbol $\sigma_{\Pi_2} \in S^{-2}(T^*X, \mathrm{Hom}(S^2T^*X))$ given by
\begin{equation}
\label{equation:symbol-theta2}
\sigma_{\Pi_2}(x,\xi)f :=\dfrac{4\pi^2}{|\xi|^2}\left(a(x,\xi^\sharp/|\xi|) + a(x,-\xi^\sharp/|\xi|)\right) \langle f,g_\xi\rangle g_\xi, \qquad \xi \neq 0
\end{equation}
where $g_\xi \in S^2 T^*_xX$ is defined by $g_\xi = g$ on $\ker \xi$ and $g_\xi \equiv 0$ on the orthogonal. Here $\langle \bullet, \bullet \rangle$ denotes the metric on $S^2 T^*X$ induced by $g$.

\end{enumerate}
\end{theorem}

The proof of this theorem relies on two \emph{Bolker conditions}: the global one was stated in Theorem \ref{theorem:bolker} and an infinitesimal one stated in \S\ref{sssection:infinitesimal-bolker}. In the case of the standard geodesic Radon transform, the infinitesimal Bolker condition corresponds to the \emph{twist property} of the vertical bundle on $SX$, see \cite[Section 2.2.1]{Bohr-Lefeuvre-Paternain-24} for instance; the global one is equivalent to the absence of conjugate points.

Note that $\Pi_2$ is not elliptic as $\ker \sigma_{\Pi_2}(x,\xi) \neq \{0\}$. The kernel and microlocal kernels of $\Pi_2$ will not be studied in the present paper and are left for future investigation, see \S\ref{ssection:non-conformal} for further discussions. We emphasize that $\Pi_0$ and $\Pi_2$ are seen as pseudodifferential operators on the open manifold $X^\circ$ here. That is they are seen as acting on smooth compactly supported functions (resp. tensors) in $X^\circ$ and yield smooth, non-compactly supported functions (resp. tensors).

\subsubsection{Analytic case} When the metric $\bar g$ is analytic, the operators turn out to be analytic pseudodifferential operators in the sense of Trèves \cite{Treves-80} in the interior of $X$:

\begin{theorem}
\label{theorem:pdo-analytic}
If $(X,\bar g)$ is analytic and $(\Phi,k)$-simple, then $\Pi_0$ and $\Pi_2$ are analytic pseudodifferential operators in $X^\circ$.
\end{theorem}

The proof of the previous theorem is based on a careful study of the Schwartz kernel of these operators near the diagonal on a blown-up space, combined with general statement about analytic pseudodifferential operators.

\subsubsection{Extension of the normal operator}

To treat non-compactly supported functions and study the kernel of the Radon transform, we will use the following theorem which crucially relies on the possibility of smoothly (resp. analytically) extending the $(\Phi,k)$-discs across the boundary $\partial X$ (see \S\ref{ssection:extended-foliation}). Recall that $(X_e, \bar g_e)$ is a proper Riemannian extension of $(X, \bar g)$. Let $E_0 : L^2(X) \to L^2_{\mathrm{comp}}(X_e)$ be the extension operator by $0$ outside of $X$.

The following holds:

\begin{theorem}[Extension of the normal operator]
\label{theorem:extension-normal-operator}
Let $(X,\bar g)$ be a smooth Riemannian ball such that $\bar g$ is $(\Phi,k)$-simple. Then there exists a pseudodifferential operator $P$ of order $-2$ on $X_e^\circ$, elliptic on $X$, such that if $f \in C(X)$ and $R_0 f = 0$, then $P E_0 f = 0$. Furthermore, if $\bar g$ is analytic, then $P$ is analytic too.
\end{theorem}

The existence of such an extension is standard in the case of the geodesic Radon transform, see \cite{Stefanov-Uhlmann-05} for instance. It is much easier to obtain insofar as geodesics on $X$ can always be realized as restrictions of geodesics on $X_e$.

\subsubsection{Kernel of the Radon transform}

The previous theorem can be used to describe the kernel of the Radon transform:

\begin{corollary}
\label{corollary:kernel}
Let $(X,\bar g)$ be a smooth Riemannian ball such that $\bar g$ is $(\Phi,k)$-simple. Then:
\begin{itemize}
\item[(i)] The operator $R_0 : C(X) \to C(\cC)$ has finite-dimensional kernel;
\item[(ii)] If $f \in C(X)$ is such that $R_0f \equiv 0$, then $f \in C^\infty(X)$ and $f$ vanishes to infinite order on the boundary $\partial X$;
\item[(iii)] There exists $N \gg 1$ such that, for all analytic $\alpha$, there exists a subset $\cU_\alpha$ of the set of area simple metrics, open and dense in the $C^N$ topology, which contains all real analytic area simple metrics, and which has the following property: for all $\bar g \in \cU_\alpha$, the corresponding Radon transform $R_0 : C(X) \to C(\cC)$ is injective.
\end{itemize}
\end{corollary}

More generally, we expect $R_0$ to be injective for all $(\Phi,k)$-simple metrics (see Conjecture \ref{conj:r0}).

\subsection{Pseudodifferential behaviour} \label{ss.pseudo-interior} The proof of Theorem \ref{th_pseudo_behaviour} is split in two steps.

\subsubsection{Smoothness of the kernel of $\Pi_0$ outside of the diagonal} This is the first step in the proof of Theorem \ref{th_pseudo_behaviour}; it relies on the global Bolker condition (Theorem \ref{theorem:bolker}). We prove that the \emph{kernel} $K_0 \in \cD'(X^\circ \times X^\circ)$ of $\Pi_0$ is smooth outside the diagonal. This is the content of the next lemma. To deal with the operator $\Pi_2$, we introduce the vector bundle $E \to X \times X$ whose fibre above $(x,y) \in X \times X$ is given $E_{(x,y)} := \mathrm{Hom}(S^2 T^*_y X, S^2 T^*_x X)$.

\begin{lemma}\label{lem_smooth_kernel} There exists a smooth function $K_0 \in C^\infty(X^\circ \times X^\circ \moins\Delta)$
such that for every $x\in X^\circ$, $f\in C^\infty_{\mathrm{comp}}(X)$:
\[\Pi_0f(x) = \int_X K_0(x,y)f(y) \dd \vol_g(y).
\]
Similarly, there exists a smooth section $K_2 \in C^\infty(X^\circ \times X^\circ \moins \Delta, E)$ such that for every $x \in X^\circ$, $f \in C^\infty_{\mathrm{comp}}(X,S^2 T^*X)$:
\[
\Pi_2 f(x) = \int_X K_2(x,y) f(y) \dd \vol_g(y).
\]
Finally, if $\bar g$ is analytic, then the Schwartz kernels are also analytic on $X^\circ \times X^\circ \moins \Delta$.
\end{lemma}

\begin{proof}
The second case being similar, we only deal with $\Pi_0$ and compute explicitly its kernel. Observe that the space $\{(v,y) ~:~ v \in S_xX, y \in \mathrm{D}(\pi_R(v))\}$ is diffeomorphic to $\hat{\Sigma}(x)$ (defined in \eqref{equation:coiffeur}) via the map $(v,y) \mapsto (\pi_R(v),y)$. We let $\dd\hat{\Sigma}(x)$ be the pushforward of the measure $\dd\mathrm{D}(\pi_R(v)) \dd v$ under this map.

For $f \in C^\infty_{\mathrm{comp}}(X)$, $x \in X^\circ$, using the maps $\pi_1$ and $\pi_2$ defined in \eqref{equation:diagram_MD}, we may rewrite \eqref{equation:r0star-r0} as:
\begin{equation}
\label{equation:ventilo}
\begin{split}
\Pi_0f(x)  = \int_{S_x X} \left(\int_{\mathrm{D}(\pi_R(v))} f \dd\mathrm{D}(\pi_R(v))\right)a(v) \dd v  = \int_{\hat{\Sigma}(x)} f \circ \pi_1 \cdot \tilde{a}~\dd \hat{\Sigma}(x),
\end{split}
\end{equation}
where $\tilde{a}(c) := a(\pi_R^{-1}(c))$, using that $\pi_R : S_xX \to \cC$ is a diffeomorphism. 

Our aim is to describe the kernel of $\Pi_0$ away from the diagonal. Hence, we may fix a point $y_0 \in X^\circ$, consider a function $f$ with support in a neighborhood $U \subset X^\circ$ around $y_0$ and take $x$ outside of $U$. Let $V := \pi_1^{-1}(U) \subset \hat{\Sigma}(x)$. The map $\pi_1 : V \to U$ is a smooth submersion with fibre above $y \in U$ ($y \neq x$) given by $H_{x,y} := \pi_1^{-1}(y) \subset  \hat{\Sigma}(x)$ and diffeomorphic to a circle as $y$ is uniformly away from $x$ (Lemma \ref{lemma:circle}). By Fubini, there is a smooth measure $\dd H_{x,y}$ on $H_{x,y}$ such that for all functions $F \in C^\infty_{\mathrm{comp}}(V)$:
\[
\int_U \left(\int_{H_{x,y}} F ~\dd H_{x,y} \right) \dd \vol_g = \int_{\hat{\Sigma}(x)} F~ \dd \hat{\Sigma}(x),
\]
where $y \mapsto \int_{H_{x,y}} F ~\dd H_{x,y}$ is seen as a function on $U$ (it is a function on $V$ which is constant along the $H_{x,y}$ fibres by construction). Going back to \eqref{equation:ventilo}, we find:
\[
\Pi_0f(x) = \int_X f(y) \left(\int_{H_{x,y}} ~\tilde{a}~ \dd H_{x,y}\right) \dd \vol_g(y).
\]
That is
\begin{equation}
\label{equation:noyau-pi0}
K_0(x,y) = \int_{H_{x,y}} ~\tilde{a}~ \dd H_{x,y}.
\end{equation}
This is a smooth function away from the diagonal.
\end{proof}

\subsubsection{The infinitesimal Bolker condition}  \label{sssection:infinitesimal-bolker}Our goal is to establish that the operator $\Pi_0$ is pseudodifferential. For that, we apply Hörmander's criterion \eqref{equation:pdo-expansion}. Let $f \in C^\infty_{\mathrm{comp}}(X^\circ)$ be a smooth function with compact support in a small open subset $U$ near $x_0 \in X^\circ$, let $S \in C^\infty(X)$ such that $\dd S \neq 0$ on the support of $f$. Since the kernel of $\Pi_0$ is smooth outside of the diagonal (Lemma \ref{lem_smooth_kernel}), it suffices to show that $e^{-i/hS}P(e^{i/h}S f)(x)$ admits a uniform asymptotic expansion for $x \in U$.

We first fix $x \in U$, and write $\xi := \dd S(x)$. Let $v_0 \in S_xX$ be defined by $v_0 := \xi^\sharp/|\xi|$. Notice that $\pm v_0$ are the unique points in $S_xX$ such that $\xi \in N^* D(\pi_R(x,v_0))$ (this will be crucial in applying the stationary phase lemma). We also further assume that $f$ is supported in a small neighborhood of $x$. We can rewrite \eqref{equation:r0star-r0} as
\begin{equation}
\label{equation:stationary}
\Pi_0(e^{iS/h}f)(x) =\int_{S_xX} \int_{T_v(S_xX)} e^{iS(\Psi(v,u))/h} f(\Psi(v,u)) a(v)\dd \mathrm{D}_{\pi_R(v)}(u)  \dd v,
\end{equation}
where $\Psi(v,u) := \exp_x^{v^\perp}(u)$ and $\exp_x^{v^\perp} : v^\perp \to \mathrm{D}(\pi_R(v))$ denotes the leafwise (Riemannian) exponential map, namely, the exponential map on $\mathrm{D}(\pi_R(v))$ for the metric $g$ induced by $\bar g$ on the disc. Observe that the exponential map is a diffeomorphism for small values of $u$; these are the only values that appear since $f$ has small support near $x$. The expression \eqref{equation:stationary} is an oscillatory integral with real-valued phase
\[
\fhi : T(S_xX) \to \R, \qquad \fhi(v,u) :=  S(\Psi(v,u)).
\]

\begin{prop}[Infinitesimal Bolker condition]\label{th_inf_bolker}
Let $\xi=\dd S(x)$ and $v_0=\xi^\sharp/|\xi|$. Then the phase $\fhi$ is only critical at $v = \pm v_0$ and $u=0$, and the Hessian $\dd^2\fhi (\pm v_0,0)$ is non-degenerate.
\end{prop}

\begin{proof}
The tangent space $T_{(v,u)}(TS_xX)$ splits into
\begin{equation}
\label{equation:decomp-s2}
T_{(v,u)}(TS_xX) = \cH_{(v,u)} \oplus \cV_{(v,u)},
\end{equation}
the horizontal and the vertical subspace of the $2$-sphere (corresponding respectively to a variation of $v$, and a variation of $u$). Notice that $\cV_{(v,u)} \simeq T_v S_xX \simeq v^\perp$. For $u \neq 0$, the map $\dd\Psi(v,u) : T_{(v,u)}(TS_x X) \to T_y X$ (where $y := \Psi(v,u)$) is surjective by Theorem \ref{theorem:bolker}. Since $\dd S \neq 0$ on the support of $f$, this implies that no point $u \neq 0$ can be critical for the phase. At $(v,0)$, one has $\fhi(v,0) = S(x)$ which is independent of $v$ and thus $\partial_v \varphi (v,0) = 0$. In addition, for $w \in \cV_{(v,0)} \simeq v^\perp$:
\[
\dd \varphi_{(v,0)}(w) =\dd S_x(w) = \xi(w).
\]
This is $0$ for all variations $w$ if and only if $w$ lies in the kernel of $\xi$, that is $w$ is orthogonal to the line $\R v_0$. This proves that the critical points are $(\pm v_0,0)$.

We now compute the Hessian of $\varphi$ at the two critical points $v=\pm v_0,u=0$. It will be convenient to express this matrix as a $2\times 2$ block matrix in terms of the decomposition \eqref{equation:decomp-s2}. We fix an arbitrary orthonormal basis $(\mathbf{e}_1,\mathbf{e}_2)$ of $v_0^\perp$. Using the Sasaki metric on $T(S_xX)$, we may consider the horizontal and the vertical lifts of this basis; this equips both $\cH_{(v_0,0)}$ and $\cV_{(v_0,0)}$ with orthonormal basis $(\mathbf{h}_1,\mathbf{h}_2)$ and $(\mathbf{v}_1,\mathbf{v}_2)$ (with respect to the Sasaki metric). We claim that
\[
\dd^2\fhi_{(\pm v_0,u)} = \begin{pmatrix}
\pi^* \dd^2 S|_{v_0^\perp} & \pm A \\ \pm A & 0
\end{pmatrix}, \qquad A := \begin{pmatrix} -|\xi| & 0 \\ 0 & -|\xi| \end{pmatrix},
\]
where the first diagonal block of $\dd^2\fhi$ corresponds to the Hessian restricted to $\cH_{(v_0,0)} \times\cH_{(v_0,0)}$, the second diagonal block is the restriction to $\cV_{(v_0,0)} \times\cV_{(v_0,0)}$, and the entries of $A$ are given by
\[
A_{11} = \dd^2\varphi(\mathbf{h}_1,\mathbf{v}_1)=-|\xi|= A_{22} = \dd^2\varphi(\mathbf{h}_2,\mathbf{v}_2), \quad A_{12}=A_{21}=\dd^2\phi(\mathbf{h}_1,\mathbf{v}_2)=0.
\]
In the first diagonal block, $d^2S|_{v_0^\perp}$ denotes the Hessian of $S$ computed with respect to the induced metric on the unique disc with normal given by $v_0$.

The diagonal blocks of $\dd^2\varphi$ are immediate to compute by taking either a path $u(t) \in T_{v_0}(S_xX)$ such that $u(0)=0$ and computing $\partial^2_t \varphi(v_0,u(t))$, or a path $v(t) \in S_xX$ such that $v(0)=v_0$ and computing $\partial^2_t \varphi(v(t),0)$.

Let us now turn to the computation of $A$. Let $\gamma(t,s)$ be a smooth path in $T(S_xX)$ such that $\gamma(t,0)=(v_t,0)$ and $\gamma(t,s)=(v_t,u_{t,s})$ for some $u_{t,s} \in T_{v_t}S_xX$. Assume that $\partial_t \gamma(0,0)=\mathbf{h}_i$ and $\partial_s \gamma(0,0) = \mathbf{v}_j$; then $\partial_t \partial_s \varphi(\gamma(t,s))|_{t=s=0} = \dd^2\fhi(\mathbf{h}_i,\mathbf{v}_j)$. Note that $\partial_s \gamma(t,0) \in \cV_{(v_t,0)} \simeq v_t^\perp$ can be identified with a vector in $v_t^\perp$ (via the connection map), and
\[
\partial_s\fhi(\gamma(t,s))|_{s=0} = \dd S_x(\partial_s\gamma(t,0)) = \xi(\partial_s\gamma(t,0)).
\]
Suppose that $i=1$, that is $v_t$ is moved in the direction $\mathbf{h}_1$ (which can be identified with $\mathbf{e}_1$). We may further assume that $v_t = \exp_{v_0}(t \mathbf{e}_1)$, where $\exp$ denotes the exponential map on $S_xX \simeq S^2$. We can then choose the path $\gamma(t,s)$ such that $\partial_s\gamma(t,0)$ is either constant equal to $\mathbf{e}_2$ (this corresponds to $\partial_s \gamma(0,0) = \mathbf{v}_2$) or $\partial_s\gamma(t,0) = \partial_t v_t$ (that is $\partial_s \gamma(0,0) = \mathbf{v}_1$). In the first case, we find $\fhi(\partial_s\gamma(t,0)) = \dd S_x(\mathbf{e}_2)$ so taking the derivative with respect to $t$ gives $\dd^2 \fhi(\mathbf{h}_1,\mathbf{v}_2) = 0$. In the second case, we can write the variation as $\partial_s \gamma(t,0) = \sin(\alpha_t) v_0 + \cos(\alpha_t) w_t$, where $\alpha_0 = 0$, $w_t \in S_x X \cap v_0^\perp$, $w_0=\mathbf{e}_1$; then $\partial_t \partial_s \gamma(t,0) = \dot{\alpha}_0 v_0 + \dot{w}_0$ with $\dot{w}_0 \perp v_0$ and $\dot{\alpha}_0 = -1$ and we find
\[
\dd^2\fhi(\mathbf{h}_1,\mathbf{v}_1) = \xi(\dot{\alpha}_0 v_0 + \dot{w}_0) = - \xi(v_0) = -|\xi|,
\]
since $\dot{w}_0 \perp \xi$ and $v_0=\xi^\sharp/|\xi|$. The same computation works for $i=j=2$. This proves the claim (the same calculation holds at $v=-v_0$ modulo a flip in the sign).

The Hessian $\dd^2\fhi(0,\pm v_0)$ has determinant given by
\begin{equation}
\label{equation:determinant}
|\det\left(\dd^2\fhi(0,\pm v_0)\right)| = (\det A)^2 = |\xi|^4>0.
\end{equation}
This concludes the proof.
\end{proof}

\subsubsection{Pseudodifferential behaviour} We are now ready to conclude the proof of Theorem \ref{th_pseudo_behaviour}.

\begin{proof}[Proof of Theorem \ref{th_pseudo_behaviour}]
(i) We first deal with $\Pi_0$. The kernel of $\Pi_0$ is smooth outside of the diagonal (Lemma \ref{lem_smooth_kernel}), so it remains to study it near the diagonal. The measure $\dd D_{\pi_R(v)}(u)  \dd v$ on $T(S_xX)$ appearing in \eqref{equation:stationary} is smooth and proportional to the Sasaki volume form $\dd u \dd v$ of $T(S_xX)$ (the volume form of the Sasaki metric on $T(S_xX) \simeq T(S^2)$), that is $\dd D_{\pi_R(v)}(u)  \dd v = \kappa(v,u) \dd u \dd v$ for some smooth positive function $\kappa$ on $T(S_xX)$. In addition, $\kappa(v,0)=1$ (this stems from the fact that the differential of the Riemannian exponential map $\exp_x$ at $0$ is the identity).

We then apply the stationary phase lemma (see \cite[Theorem 3.3.2]{Lefeuvre-book} for instance) to expand \eqref{equation:stationary} in powers of $h$. As the volume form $\dd D_{\pi_R(v)}(u)  \dd v$ coincides with the Sasaki volume form $\dd u \dd v$ at the critical points $v=\pm v_0, u=0$, and the determinant of the Hessian was computed with respect to the Sasaki metric in \eqref{equation:determinant}, there are no additional factors in the stationary phase. We then find that there are two contributions coming from the two critical points:
\[
\begin{split}
\Pi_0(e^{iS/h}f)(x) &\sim  \dfrac{(2\pi h)^2}{|\xi|^2} e^{iS(x)/h} \left(f(x) a(x,v_0)+ \sum_{j \geq 1} h^j A_j^+(x) (\Psi^* f \cdot a)|_{v=+v_0,u=0}\right) \\ &+
\dfrac{(2\pi h)^2}{|\xi|^2} e^{iS(x)/h} \left(f(x)a(x,-v_0)+ \sum_{j \geq 1} h^j A_j^-(x) (\Psi^* f \cdot a)|_{v=-v_0,u=0}\right),
\end{split}
\]
for some differential operators $A_j^\pm(x)$ of order $\leq 2j$ on $T(S_xX)$. In addition, these operators depend smoothly on $x \in X$. This proves the claim and also yields the principal symbol of $R_0^*R_0$, namely
\[
\sigma_{\Pi_0}(x,\xi) = \dfrac{4 \pi^2}{|\xi|^2}\left(a(x,\xi^\sharp/|\xi|) + a(x,-\xi^\sharp/|\xi|)\right) .
\]

(ii) We now deal with $\Pi_2$. For $f_1, f_2 \in C^\infty_{\mathrm{comp}}(X,S^2T^*X)$ and $S$ as above, we compute
\[
\langle \Pi_2(e^{iS/h} f_1)(x), f_2(x) \rangle_{S^2 T^*_xX} = \langle R^*R \pi_\perp^*(e^{iS/h} f_1)(x), \pi_\perp^*f_2\rangle_{L^2(S_xX, \dd v)}.
\]
The right-hand side is very similar to the term computed in (i) (where $f(x)$ in (i) is now replaced by the function  $v \mapsto \tr(f_1|_{v^\perp})\tr(f_2|_{v^\perp})$ on $S_xX$). The same stationary phase argument applies and proves that $\Pi_2$ is a pseudodifferential operator with
\[
\langle\sigma_{\Pi_2}(x,\xi)f_1, f_2\rangle = \dfrac{4 \pi^2}{|\xi|^2}\left(a(x,\xi^\sharp/|\xi|) + a(x,-\xi^\sharp/|\xi|)\right)\tr(f_1|_{(\xi^\sharp)^\perp}) \tr(f_2|_{(\xi^\sharp)^\perp}), \qquad \xi \neq 0.
\]
A quick computation then shows that the principal symbol coincides with \eqref{equation:symbol-theta2}.
\end{proof}

\subsection{The kernel near the diagonal} \label{ss.kernel} To study the operator in the analytic setting, we need a precise expression of the kernel of the operator $\Pi_0$ on the diagonal. As pseudodifferential operators are conormal distributions on the diagonal $\Delta \subset X^\circ \times X^\circ$, it is convenient to blow up the diagonal and consider the blown-up space
\[
 \pi_{(2)} : X^\circ_{(2)} := [X^\circ \times X^\circ ; \Delta] \to X^\circ \times X^\circ
\]
The kernel of $\Pi_0$ naturally lives on $X^\circ_{(2)}$.

In a more concrete way, this boils down, at every $x \in X^\circ$, to introducing polar coordinates $y = \exp_x(r\omega)$ around $x$ (where $\exp$ denotes the Riemannian exponential map of $g$), where $r > 0$ and $\omega \in S_xX$, and compactifying the space by adding the slice $\{r=0\}$. Note that the slice $\{x \in X^\circ, r=0, \omega \in S_xX\}$ is then exactly the blow up of the diagonal (it projects via $\pi_{(2)} : X^\circ_{(2)} \to X^\circ \times X^\circ$ onto the diagonal). Finally, we let $C_{x,\omega} := S_xX \cap \omega^\perp$ be the great circle in $S_xX$ contained in the plane $\omega^\perp$, and $\dd C_{x,\omega}$ be the arclength measure on it.

The following holds:

\begin{theorem}[Expression of the Schwartz kernel near the diagonal]
\label{theorem:wow}
There exists a smooth function $\chi \in C^\infty(X^\circ_{(2)})$ such that the Schwartz kernel $K_0 \in \cD'(X^\circ \times X^\circ)$ of $\Pi_0$ satisfies near the diagonal:
\begin{equation}
\label{equation:wow}
K_0(x,r,\omega) = \dfrac{\chi(x,r,\omega)}{r}.
\end{equation}
In addition,
\[
\chi(x,r,\omega) = \int_{C_{x,\omega}} a~ \dd C_{x,\omega} + \cO(r),
\]
where $\cO(r)$ is a smooth remainder term.

Furthermore, if the foliation is analytic, then $\chi$ is analytic.
\end{theorem}

The proof of \eqref{equation:wow} occupies the rest of this subsection.

\subsubsection{The model case}

The idea is to describe the $(\Phi,k)$-discs passing through $x$ as a perturbation of a Euclidean model case which we now describe.

Assume that $(X,g) \simeq \R^3$ is Euclidean and $x \in X^\circ$. We define the local disc $\mathrm{D}_{\mathrm{euc}}(\pi_R(v))$ passing through $x$ with normal $v$ by:
\begin{equation}
\label{equation:disc-model}
\mathrm{D}_{\mathrm{euc}}(\pi_R(v)) = \{x+u ~:~ u \in v^\perp \}.
\end{equation}
The map $\Psi : T(S_xX) \to X$ is then defined by $\Psi(v,u) = u$. We further assume that $a(v)=1$, where $a$ is the function in \eqref{equation:r0star-r0}. The space $T(S_xX)$ is here equipped with the natural Sasaki metric from $S_xX \simeq S^2$ and the target space $X$ is equipped with the Euclidean metric.

\begin{lemma}
In the model case, $\chi = 2\pi$.
\end{lemma}

\begin{proof}
Without loss of generality, we assume $x=0 \in \R^3$. Let $y$ be a point near $x$ and write $y = x + r \omega$ for some $\omega \in S_xX$. Let $\pi : T(S_xX) \to S_xX$ be the projection and $C_{x,y} := \pi(H_{x,y}) \simeq S^1$; this is the set of vectors on $S_xX$ such that the surface whose Gauss lift at $x$ is $v \in C_{x,y}$ contains $y$. Observe that $C_{x,y} := C_{x,\omega} = S_xX \cap \omega^\perp$ is independent of $r$.

We now compute $\dd\Psi$. Let $v \in C_{x,y}$ and $u \in v^\perp$ such that $y = u$. The tangent space to $T_{(v,u)} \left(T(S_xX)\right)$ splits as $\cH_{(v,u)} \oplus \cV_{(v,u)}$ where $\cV_{(v,u)} \simeq v^\perp$ is equipped with the Euclidean metric. The map $\dd\Psi : \cV_{(v,u)} \to T_y \mathrm{D}(\pi_R(v)) = v^\perp$ is an isometry (it is the identity actually). To compute $\dd\Psi : \cH_{(v,u)} \to T_y X$, we may take $\mathbf{e}_1 = \omega \in T_v(S_xX), \mathbf{e}_2 = \omega^\perp$ (the rotation of $\omega$ by $+\pi/2$). Note that $\mathbf{e}_2$ is tangent to $C_{x,\omega}$. We then let $\mathbf{h}_i$ be the horizontal lifts of $\mathbf{e}_i$. That is $\mathbf{h}_1$ corresponds (infinitesimally) to the parallel transport of $u = r\omega$ along the geodesic spanned by $\mathbf{e}_1$ and $\mathbf{h}_2$ corresponds (infinitesimally) to the parallel transport of $u$ along $C_{x,\omega}$, which keeps $u$ fixed in $v^\perp$. A quick computation then shows that
\[
\dd \Psi(\mathbf{h}_1) = - r v, \qquad \dd \Psi(\mathbf{h}_2) = 0.
\]
Equivalently, letting $\alpha$ be the smooth $1$-form on $T(S_xX)$ defined on $H_{x,y}$ by $\alpha(\mathbf{h}_2)=1$ and $\alpha(\mathbf{h}_1)=0=\alpha(\cV)$ (which has norm $1$), we see that $|\Psi^*(\vol_y) \wedge \alpha/r|  = \dd u \dd v$ is the Sasaki volume. That is
\[
|H_{x,y}| = \int_{H_{x,y}} \alpha/r = 2\pi/r.
\]
This proves the claim.
\end{proof}

\subsubsection{Geometry of $T(S_xX)$} \label{sssection:tsx-geometry} We briefly recall the geometry of $T(S_xX)$ we are considering. Let $\pi : T(S_xX) \to S_xX$ be the footpoint projection. We recall that $\cV_{(v,u)} = \ker \dd\pi$ is the vertical bundle and $\cH_{(v,u)}$ is the horizontal distribution (for the round metric on $S_xX \simeq S^2$). We let $\Psi(v,u) := \exp_x^{v^\perp}(u)$, the leafwise exponential map. Given $(v,u) \in T(S_xX)$, the map $\dd\Psi_{(v,u)} : \cV_{(v,u)} \to T_y\mathrm{D}(\pi_R(v))$ is an isomorphism. By definition, $T_y\mathrm{D}(\pi_R(v))$ is equipped with the metric induced by $(X,g)$.

This allows to define a natural metric $G_x$ on $T(S_xX)$ which we now describe. We first equip $\cV_{(v,u)}$ with the pullback metric using $\dd\Psi$. We make the splitting $\cV \oplus \cH$ orthogonal with respect to $\cH$. On $\cH$, we consider the Sasaki metric, that is for $Z_1,Z_2 \in \cH_{(v,u)}$, we have $G_x(Z_1,Z_2)=\langle \dd\pi(Z_1),\dd\pi(Z_2)\rangle$. This defines a metric $G_x$ on $\cV \oplus \cH$. The volume form induced by this metric is denoted by $\dd v \dd \mathrm{D}(\pi_R(v))$. That is the integral of a function $f$ on $T(S_xX)$ against this measure is given by
\begin{equation}
\label{equation:integration}
\int_{S_xX} \left(\int_{T_v(S_xX)} f(v,u) \dd \mathrm{D}(\pi_R(v))\right) \dd v.
\end{equation}
Observe that this coincides with \eqref{equation:r0star-r0}, except for the absence of the function $a$. Finally, all these objects are smooth with respect to the base point $x \in X^\circ$.

\subsubsection{Local parametrization} \label{sssection:model-case} We now deal with the general case. We first parametrize the $(\Phi,k)$-discs. At any point $x \in X$, it will be convenient to work in an exponential chart around $x$. That is, we identify the neighborhood of $x$ with a neighborhood of $0$ in $T_xX$ under the diffeomorphism $T_xX \ni w \mapsto y:= \exp_x(w) \in X$. Recall that $\cV \to SX$ is the vertical vector bundle whose fibre at $(x,v) \in SX$ is $v^\perp$ (the tangent space to $S_xX$). We let $\cV^\bullet$ be the restriction of $\cV$ to $SX^\bullet$.

The following holds:

\begin{lemma}
Assume that $\bar g$ is $(\Phi,k)$-simple. Then there exists a smooth function $f : \cV^\bullet \to \R$ defined near the $0$-section in $\cV^\bullet$ such that for all $v \in SX$ with $x = \pi(v)$:
\begin{equation}
\label{equation:parametrization-discs}
\mathrm{D}(\pi_R(v)) \cap U_x = \{\exp_x(u + f(x,v,u)v) ~:~ v \in S_xX, u \in v^\perp\},
\end{equation}
where $U_x$ is a small neighborhood of $x$. In addition, $f(x,v,u) = \cO(|u|^2)$.
\end{lemma}

Here $u + f(v,u)v \in T_x X$. In the following, we will see the point $x$ as being a fixed parameter, and work in the exponential chart at $x$. That is we will omit the term $\exp_x$. Note that in the model case, $f \equiv 0$, and $\exp_x$ does not appear. The function $f$ inherits the regularity of the foliation by $(\Phi,k)$-discs. If the foliation is analytic, then so is $f$. The proof is immediate and simply relies on the fact that $\mathrm{D}(\pi_R(v))$ is a smooth family of discs (with respect to $v$) passing through $x$ (hence $f(v,0)=0$) and such that $v$ is normal to $\mathrm{D}(\pi_R(v))$ at $x$ (hence $\partial_u f(v,0)=0$). All such discs can be written as graphs over the discs of the model case (see \eqref{equation:disc-model}). The dependence of $f$ with respect to $x$ is also smooth.

\subsubsection{Expression of the curve $C_{x,y}$} To simplify notation, we now fix the point $x \in X$. However, we emphasize that all the objects described in this paragraph are allowed to vary with respect to $x$.

We let $\pi : T(S_xX) \to S_xX$ be the footpoint projection and $C_{x,y} := \pi(H_{x,y})$. We take polar coordinates and write $y(r,\omega) = \exp(r\omega)$ for $r \in [0,r_0) \times S_xX$, where $\exp$ is the Riemannian exponential map of $\bar g$. Note that, as we work in an exponential chart around $x$, this simplifies to $y(r\omega)=r\omega$. We may thus write $C_{x,y} = C_{x,\omega,r}$. We let
\[
C_{x,\omega,0} := S_xX \cap \omega^\perp.
\]
The main claim is that $C_{x,\omega,r}$ are smooth graphs over $C_{x,\omega,0}$ up to $r=0$ (see Lemma \ref{lemma:graph-peyresq} below).

To package this in a single space, we introduce the circle bundle $E \to S_xX$ whose fibre above $\omega \in S_xX$ is $C_{x,\omega,0} = S_xX \cap \omega^\perp$. Given $\omega \in S_xX$, $v \in C_{x,\omega,0}$, and $\eps \in \R$, we let $v_\eps := \exp_v(\eps \omega)$, where the exponential map is the Riemannian exponential map on $S_xX \simeq S^2$. Note that
\begin{equation}
\label{equation:thi}
v_\eps = \cos(\eps) v + \sin(\eps) \omega.
\end{equation}
Finally, we let $D(S_xX) \subset T(S_xX) \to S_xX$ denote the (closed) disc bundle over $S_xX$ whose fibre at $v \in S_xX$ is the disc in $T_v(S_xX)$ of radius $r_0 > 0$ (for $r_0 > 0$ small enough) computed with respect to the Sasaki metric.

The following holds:

\begin{lemma}
\label{lemma:graph-peyresq}
Assume $f \in C^{k,\alpha}(D(S_xX))$ satisfies $f(v,0)=0$, $\partial_u f(v,0)=0$. Then there exist (unique) $C^{k-1,\alpha}$ functions
\[
g : E \times [0,r_0) \to \R, \qquad  u : E \times [0,r_0) \to T_xX, \qquad u(\omega,v,0)=0
\]
such that for every $\omega \in S_xX$, for every $v \in E_\omega = C_{x,\omega,0}$:
\begin{equation}
\label{equation:potrie}
y(r,\omega) = u(\omega,v,r) + f(v_{g(\omega,v,r)}, u(\omega,v,r)) v_{g(\omega,v,r)}, \quad u(\omega,v,r) \perp v_{g(\omega,v,r)},
\end{equation}
and
\begin{equation}
\label{equation:implication}
\partial_r u(\omega,v,0) = \omega, \qquad g(\omega,v,0) = 0.
\end{equation}
In addition, if $f$ is smooth (resp. analytic), then $(g,u)$ are smooth (resp. analytic).
\end{lemma}

Note that in the model case \S\ref{sssection:model-case}, $u_0(\omega,v,r) := r \omega$, $g(\omega,v,r)=0$ and $v_{g(\omega,v,r)} = v_0=v$.  If $f$ is also $C^{k,\alpha}$ with respect to $x \in X$, then the functions $u$ and $g$ also depend $C^{k-1,\alpha}$ smoothly on $x \in X$.

\begin{figure}[h!]

\begin{center}
\includegraphics[scale=0.8]{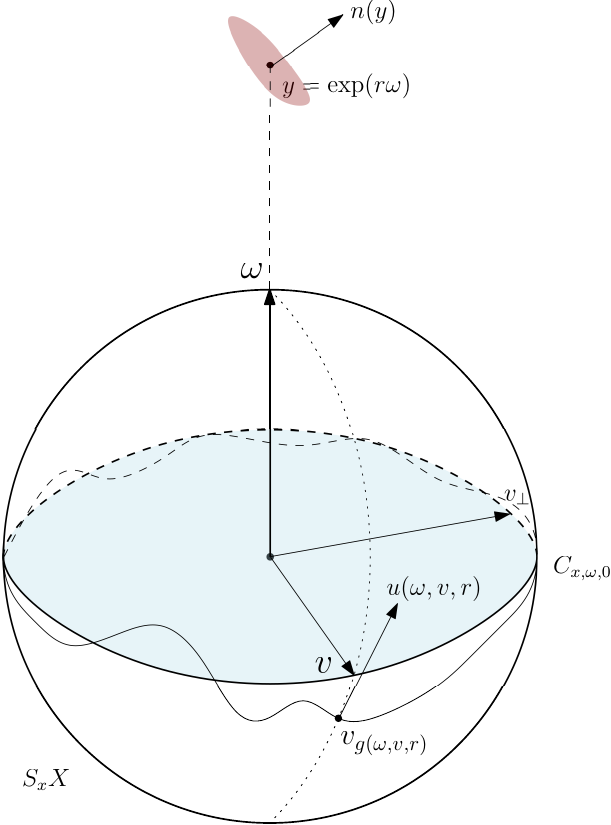}
\caption{Cartoon illustrating the objects introduced in Lemma \ref{lemma:graph-peyresq}.}
\label{figure:sphere}

\end{center}

\end{figure}

\begin{proof}
For fixed $\omega \in S_xX$ and $v \in E_\omega = S_xX \cap \omega^\perp$, we solve the implicit equation $F(r ; (g,u')) = 0$ for small values of $r$ where
\[
F :  [0,r_0) \times (-\eps,\eps) \times T_xX \to T_xX \times \R, \qquad F = (F_1,F_2)
\]
is given by
\[
\begin{split}
F_1(r; (g,u')) := \omega- \left(u' + r^{-1} f(v_{g}, r u') v_{g}\right), \qquad F_2(r; (g,u')) := \langle u', v_{g} \rangle.
\end{split}
\]
Note that $f(v,u) = \cO(|u|^2)$ by assumption so the function $F_1$ is $C^{k-1,\alpha}$ up to $r=0$ if $f$ is $C^{k,\alpha}$, and $r \mapsto r^{-1}f(v,ru')$ vanishes at $r=0$. At $r=0$, $g=0$ and $u'=\omega$ solve the equation $F(0;(0,\omega))=0$. The function $u$ appearing in \eqref{equation:potrie} is $u(r) = r u'(r)$. Indeed, multiplying $F_1$ by $r$, we find
\[
r \omega- \left(u(r) +  f(v_{g(r)}, u(r)) v_{g(r)}\right) = y(r,\omega)- \left(u(r) +  f(v_{g(r)}, u(r)) v_{g(r)}\right) = 0.
\]
To apply the implicit function theorem, we compute the differential of $F$ with respect to $(g,u')$ at $r=0,g=0,u'=\omega$:
\[
\begin{split}
&\partial_g F_1(0;(0,\omega)) = 0, \qquad \partial_{u'} F_1(0;(0,\omega))\delta u' = -\delta u' \\
&\partial_g F_2(0;(0,\omega)) = 1, \qquad \partial_{u'} F_2(0;(0,\omega))\delta u' = \langle \delta u', v \rangle.
\end{split}
\]
This map is clearly invertible, so the implicit function theorem applies.
\end{proof}

\subsubsection{Proof of Theorem \ref{theorem:wow}} We now complete the proof of Theorem \ref{theorem:wow}. To do so, we express the differential of the submersion $\Psi$. Let $y = \exp(r\omega)$, and consider a point $z = (v,u) \in H_{x,y}$. By Lemma \ref{lemma:graph-peyresq}, we can write
\[
z = (v_{g(\omega,v,r)},u(\omega,v,r)), \qquad u(\omega,v,r) \perp v_{g(\omega,v,r)},
\]
for some $\omega \in S_xX, v \in C_{x,\omega,0}$. The differential $\dd \Psi_z : T_z(T(S_xX)) \to T_yX$ can be split using the Sasaki decomposition
\[
T_z(T(S_xX)) = \cV_{(v,u)} \oplus \cH_{(v,u)}.
\]
Recall that this space is equipped with the metric $G_x$ introduced in \S\ref{sssection:tsx-geometry}. By construction,
\[
\dd\Psi_z : \cV_{(v,u)} \to T_y(\mathrm{D}(\pi_R(v))
\]
is an isometry.

The target space $T_yX$ is equipped with the Riemannian metric $g_y$ while $\cH_{(v,u)}$ is equipped with the Sasaki metric. Given $\omega \in S_xX$ and $v \in C_{x,\omega,0} = S_xX \cap \omega^\perp$, we let $v_\perp \in S_xX$ be the unique vector such that $(v,\omega,v_\perp)$ is an orthonormal frame of $T_xX$. Finally, we let $n(y) \in T_yX$ be the unit normal vector at $y$ to $\mathrm{D}(\pi_R(v))$ pointing in the same direction as $v$ (which is normal to $\mathrm{D}(\pi_R(v))$ at the point $x$ by construction). Note that there is a loss of $1$ regularity exponent in $n(y)$ as it involves a derivative of $f$; this accounts for the loss of $1$ exponent of regularity in Theorem \ref{theorem:wow}.

\begin{lemma}
$n(y) = v + \cO(r)$.
\end{lemma}

In the model situation, $n(y)=v$. As in the statement of Theorem \ref{theorem:wow}, the remainder term $\cO(r)$ is a smooth (vector-valued) function up to $r=0$, and the estimate is uniform with respect to $x$ (in a fixed compact set of $X^\circ$) and $\omega$.

\begin{proof}
Using \eqref{equation:parametrization-discs}, we see that the tangent space $T_y\mathrm{D}(\pi_R(v))$ is given by
\begin{equation}
\label{equation:tangent-d}
T_y\mathrm{D}(\pi_R(v)) = \{w + \partial_u f_{(v,u)}(w) v ~:~ w \in \cV_{(v,u)}\}.
\end{equation}
In the exponential chart, the metric $g(y)$ can be written as $g(y) = g_{\mathrm{Euc}} + \cO(r^2)$ (where the remainder is smooth up to $r=0$). In particular, the frame $(v,\omega,v_\perp)$ is an orthonormal frame for $g(y)$ at the point $y$ modulo $\cO(r^2)$. We look for a unit normal vector $n(y)$ of the form $n(y)=(v+w)/|v+w|$ where $\langle v,w\rangle=0$. Using \eqref{equation:tangent-d},  and that $\partial_u f = \cO(r)$, we find that $w = \cO(r)$ and the claim follows.
\end{proof}

As $\dd \Psi_z : \cV_{(v,u)} \to T_y(\mathrm{D}(\pi_R(v))$ is an isometry, it suffices to compute $\dd\Psi$ on transverse vectors to $\cV$. At a point $v_{g(\omega,v,r)} = v_g \in C_{x,\omega,r}$ of coordinates $(\omega,v,r) \in E \times [0,r_0)$, we associate
\[
\hat{T}(\omega,v,r), \hat{N}(\omega,v,r) \in T_{v_{g(\omega,v,r)}}(S_xX)
\]
where $\hat{T}$ is the unit tangent vector to the curve $C_{x,\omega,r} \subset S_xX$, and $\hat{N}$, the unit normal vector to the curve, with the convention that $(\hat{T}, \hat{N})$ has the orientation of $S_xX$. The following holds:
\begin{lemma}
Letting
\begin{equation}
\label{equation:tn}
\begin{split}
T& := -\partial_v g \sin(g) v + \cos(g) v_\perp + \partial_v g \cos(g)\omega, \\
 N & := -\tan(g)v - \dfrac{\partial_v g}{\cos^2(g)} v_\perp + \omega,
\end{split}
\end{equation}
we have $\hat{T} = T/|T|, \hat{N}=N/|N|$. In particular:
\begin{equation}
\label{equation:tn-simple}
\hat{T} = v_\perp + \cO(r), \qquad \hat{N} = \omega + \cO(r).
\end{equation}
\end{lemma}

Note that in the model case, $g=0$, and we obtain that $\hat{T} = v_\perp$, $\hat{N} = \omega$. As above, the remainder $\cO(r)$ are smooth with respect to $(\omega,v)$ up to $r=0$. As above, there is a loss of $1$ exponent of regularity as the formulas \eqref{equation:tn} involve the derivative of $g$.

\begin{proof}
Consider the map $C_{x,\omega,0} \ni v \mapsto v_{g(\omega,v,r)} \in C_{x,\omega,r}$. This curve lifts to the curve $v \mapsto (v_{g(\omega,v,r)},u(\omega,v,r)) \in H_{x,\omega,r} \subset E$. Using \eqref{equation:thi}, and differentiating $v$ in the $v_\perp$ direction, we find (writing $\partial_v g := \partial_vg(v_\perp)$):
\[
T := \partial_v v_{g(\omega,v,r)}(v_\perp) = -\partial_v g \sin(g) v + \cos(g) v_\perp + \partial_v g \cos(g) \omega.
\]
We then normalize $T$ to obtain the unit vector $\hat{T} := T/|T|$.

To compute $N$, we look for a vector field with coordinates $N = \alpha v + \beta v_\perp + \omega$ satisfying the conditions $\langle N,v_g \rangle = 0 = \langle N,T\rangle$. This yields:
\[
\alpha \cos(g) + \sin(g) = 0 = -\alpha \partial_vg \sin(g) + \beta \cos(g) + \partial_v g \cos(g).
\]
That is
\[
\alpha = -\tan(g), \qquad \beta = -\dfrac{\partial_v g}{\cos^2(g)}.
\]
Normalizing $N$, we find that $\hat{N}=N/|N|$.

The estimate \eqref{equation:tn-simple} is obtained by combining \eqref{equation:tn} with \eqref{equation:implication}.
\end{proof}

Finally, we consider $(v,u) \in H_{x,y}$. The vector $\hat{T}$ (which is tangent to $C_{x,y}$ at $v \in C_{x,y}$) lifts uniquely to a vector $\hat{T}^{\mathrm{lift}} \in T_{(v,u)}(T(S_xX))$ which is tangent to $H_{x,y}$. By construction, the horizontal component of $\hat{T}^{\mathrm{lift}}$ projects to $\hat{T}$ but $\hat{T}^{\mathrm{lift}}$ may have a non trivial vertical component. We also consider $\hat{N}^{\mathrm{lift}}$, the horizontal lift of $\hat{N}$ at the point $(v,u)$. If we perturb $(v,u)$ in the direction $\hat{T}^{\mathrm{lift}}$, then $y = \Psi(v,u)$ is not moved since $\Psi(H_{x,y}) = y$ (by definition). That is $\dd \Psi_z(\hat{T}^{\mathrm{lift}})=0$.

We now perturb $(v,u)$ in the direction $\hat{N}^{\mathrm{lift}}$. The point $y = \Psi(v,u)$ is then moved by $\delta y = \dd \Psi_{(v,u)}(\hat{N}^{\mathrm{lift}}) \in T_y X$.

\begin{lemma}
$\delta y = \left(-\langle u,\hat{N} \rangle  + \dd f (\hat{N}^{\mathrm{lift}})\right)v_g + f\hat{N} = -r n(y) + \cO(r^2)$.
\end{lemma}

Note that in the model case ($f=0$), $\hat{N}= \omega$, $g=0, v_g = v = n(y)$ (where $n(y) \in T_y X$ is the normal vector to the disc) and $u(\omega,v,r)=r \omega$, so we find that $\delta y = - r n(y)$.

\begin{proof}
We have $\delta v = \hat{N}$, so $\delta u = -\langle u,\hat{N}\rangle v$. Since $y = \Psi(v,u) = u + f(v,u)v$, we find that
\[
\begin{split}
\delta y & = \dd \Psi_{(v,u)}(\hat{N}^\cH) = \delta u + \dd f_{(v,u)}(\hat{N}^{\mathrm{lift}}) v_g + f \delta v \\
&  =\left(-\langle u,\hat{N} \rangle  + \dd f (\hat{N}^{\mathrm{lift}})\right)v_g + f\hat{N}.
\end{split}
\]
Note that $\dd f(\hat{N}^{\mathrm{lift}}) = \cO(r^2)$ because $\hat{N}^{\mathrm{lift}}$ is a horizontal vector field and $f(v,u)=\cO(|u|^2)$.
Using the previous approximations, we find that $\delta y = -rv + \cO(r^2) = -r n(y) + \cO(r^2)$.
\end{proof}

We pick an arbitrary orthonormal basis $(\mathbf{v}_1,\mathbf{v}_2)$ on $\cV_{(v,u)}$ (for the metric $G_x$ obtained by pullback from $T_y \mathrm{D}(\pi_R(v))$), and let $(\mathbf{e}_1,\mathbf{e}_2) := \dd\Psi_z(\mathbf{v}_1,\mathbf{v}_2)$ be the corresponding orthonormal frame on  $T_y \mathrm{D}(\pi_R(v))$. We can then write the matrix of $\dd \Psi_z$ in the basis $(\mathbf{v}_1,\mathbf{v}_2, \hat{N}^{\mathrm{lift}},\hat{T}^{\mathrm{lift}})$ of $T_z(T(S_xX))$ and in the orthonormal basis $(\mathbf{e}_1,\mathbf{e}_2,n(y))$ of $T_yX$ (with respect to $g_y$):
\begin{equation}
\label{equation:matrice-psi}
\dd\Psi_z = \begin{pmatrix} 1 & 0 & * & 0 \\
0 & 1 & * & 0 \\
0 & 0 & \beta & 0 \end{pmatrix}, \qquad \beta := g_y(\delta y, n(y)) = -r + \cO(r^2).
\end{equation}
Here, $\beta \in C^\infty(E \times [0,r_0])$ is a smooth function on $E \times [0,r_0]$, that is $\beta = \beta(\omega,v,r)$ and this is smooth up to $r=0$.

Observe that $(\mathbf{v}_1,\mathbf{v}_2, \hat{N}^{\mathrm{lift}},\hat{T}^{\mathrm{lift}})$ is \emph{not} an orthonormal frame of $T_z(T(S_xX)) = \cV \oplus \cH$ because $\hat{T}^{\mathrm{lift}}$ may not be purely horizontal and of norm $1$. Nevertheless, this frame has determinant $1$, that is the volume form is given by $\mathbf{v}_1 \wedge \mathbf{v}_2 \wedge \hat{N}^{\mathrm{lift}} \wedge \hat{T}^{\mathrm{lift}}$ after identification of $1$-forms and vectors via the metric $G_x$. We introduce the $1$-form $\alpha$ on $T(S_xX)$ such that:
\[
\alpha(\cV) = 0 = \alpha(\hat{N}^{\mathrm{lift}}), \qquad \alpha(\hat{T}^{\mathrm{lift}})=1.
\]
Observe that the horizontal component of $\hat{T}^{\mathrm{lift}}$ projects to $\hat{T}$. Since $\alpha(\cV)=0$, letting $\hat{T}^{\mathrm{lift}}_{\cH}$ be the horizontal component of $\hat{T}^{\mathrm{lift}}$, we find that $\alpha(\hat{T}^{\mathrm{lift}}_{\cH}) = 1$. Recall that the metric $G_x$ on $T(S_xX)$ induces the Riemannian measure $\dd v \dd \mathrm{D}(\pi_R(v))$ (see \eqref{equation:integration}).

\begin{lemma}
The following holds:
\[
|\Psi^*(\vol_y) \wedge \alpha/\beta| = \dd v \dd\mathrm{D}(\pi_R(v))
\]
\end{lemma}

\begin{proof}
By construction, the basis $(\mathbf{v}_1,\mathbf{v}_2, \hat{N}^{\mathrm{lift}},\hat{T}^{\mathrm{lift}})$ has determinant $1$ for the measure $\dd v \dd \mathrm{D}(\pi_R(v))$. The result is then immediate since:
\[
\begin{split}
\Psi^*(\vol_y) \wedge \alpha (\mathbf{v}_1,\mathbf{v}_2, \hat{N}^{\mathrm{lift}},\hat{T}^{\mathrm{lift}}) & = \vol_y(\dd\Psi(\mathbf{v}_1), \dd\Psi(\mathbf{v}_2),\dd\Psi( \hat{N}^{\mathrm{lift}})) \alpha(\hat{T}^{\mathrm{lift}})  = \beta.
\end{split}
\]
\end{proof}

We can now complete the proof of Theorem \ref{theorem:wow}.

\begin{proof} By construction, for $y = \exp_x(r \omega)$, $H_{x,y} = H_{x,\omega,r}$ is parametrized by
\[
H_{x,\omega,r} = \{(v_{g(\omega,v,r)}, u(\omega,v,r)) ~:~ v \in C_{x,\omega,0}\}.
\]
In turn, we parametrize $C_{x,\omega,0}$ by $[0,2\pi]$, that is we take $[0,2\pi] \ni \theta \mapsto v(\theta)$ such that $\dot{v}(\theta) = v_\perp$ at $v=v(\theta)$. Using \eqref{equation:noyau-pi0}, this yields the following formula for the kernel of $\Pi_0$:
\[
K_0(x,y) = \int_{H_{x,y}} a \cdot \alpha/|\beta| = \int_0^{2\pi} \dfrac{a(v_{g(\omega,v(\theta),r)}) \dd\theta}{|\beta|(\omega,v(\theta),r)}.
\]
Using \eqref{equation:matrice-psi}, we find that $1/\beta = r^{-1}(1+\cO(r))$ where the $\cO(r)$ is smooth with respect to $(\omega,v)$ up to $r=0$. Furthermore, $g(\omega,v(\theta),r) = \mathcal{O}(r)$. This yields:
\[
K_0(x,y) =r^{-1}\left(\int_0^{2\pi} a(v(\theta)) \dd \theta+ \cO(r)\right) = r^{-1}\chi(x,\omega,r).
\]
This completes the proof.
\end{proof}

\subsubsection{Consequences}

Finally, we derive Theorem \ref{theorem:pdo-analytic} from Theorem \ref{theorem:wow}:

\begin{proof}[Proof of Theorem \ref{theorem:pdo-analytic}]
Straightforward consequence of Theorem \ref{theorem:wow} combined with Proposition \ref{prop:analytic-pdo}.
\end{proof}

\subsection{Extension of the normal operator and consequences} We now extend the normal operator to a neighborhood of $X$ and derive a stability estimate for the Radon transform.

\subsubsection{Extension} Recall that $(X_e, \bar g_e)$ is a proper Riemannian extension of $(X, \bar g)$, and $U \subset SX_e$ is the open set introduced in Theorem \ref{theorem:extension-foliation} containing $SX$. The Radon transform $R$ can be extended to a map $R_e : C^\infty(U) \to C^\infty(\cC_e)$ by letting for $c \in \cC_e$
\[
R_e(c) := \int_{\mathrm{L}_e(c)} f~ \dd \mathrm{L}_e(c),
\]
where $\mathrm{L}_e(c)$ denotes the Gauss lift of $\mathrm{D}_e(c)$ to $SX_e$. We then define
\begin{equation}
\label{equation:extension-operator}
P := \pi_* R_e^* R_e \pi^*,
\end{equation}
where the adjoint is computed with respect to an arbitrary smooth (resp. analytic) measure on $\cC_e$. We claim that $P$ satisfies the properties required in Theorem \ref{theorem:extension-normal-operator}:

\begin{proof}[Proof of Theorem \ref{theorem:extension-normal-operator}]
That $P$ is a pseudodifferential operator of order $-2$ on $X^\circ_e$ follows \emph{verbatim} from the arguments developed in the proof of Theorem \ref{th_pseudo_behaviour} using that the double fibration of Theorem \ref{theorem:extension-foliation} satisfies the Bolker condition. That $P$ is analytic whenever $\bar g$ is analytic also follows from the arguments developed in the proof of Theorem \ref{theorem:pdo-analytic}.

Finally, the principal symbol of $P$ is
\[
\sigma_P(x,\xi) =4\pi^2 /|\xi|^2 \left(a(x,\xi^\sharp/|\xi|) + a(x,-\xi^\sharp/|\xi|)\right), \qquad x \in X_e, \xi \neq 0,
\]
for some $a > 0$. It is elliptic on $X$, hence on an open neighborhood of $X$. In addition, by construction, $PE_0 f = 0$ if $R_0 f =0$.
\end{proof}

\subsubsection{Continuity} The following technical result will be used in the proof of injectivity of the Radon transform for generic metrics:

\begin{lemma}
\label{lemma:continuity}
Let $\bar g_0$ be a $(\Phi,k)$-simple and real analytic metric on $X$. For $\bar g$ sufficiently close to $\bar g_0$, let $P_{\bar g}$ denote the extended normal operator constructed above. Then, the map
\[
C^\infty(X, \mathrm{Sym}^2 T^*X) \to \Psi^{-2}(X_e^\circ), \qquad \bar g \to P_{\bar g}
\] 
is continuous.
\end{lemma}

Both spaces are Fréchet spaces, so continuity in the previous lemma is to be understood in the Fréchet sense (using semi-norms).

\begin{proof}
First, by the implicit function theorem applied to the $(\Phi,k)$-surface equation, together with the invertibility of the Jacobi operators, the family of $(\Phi,k)$-discs $\mathrm D_{\bar g}(c)\subset X$ depends smoothly on $(\bar g,c)$ (see \S\ref{section:plateau}); the same conclusion holds up to the compactification $\overline{\cC}$ by the analysis of small discs developed in \S\ref{section:asymptotic}. We may then fix a continuous extension operator for the graph functions defining these discs near $\partial X$, and use it, for every $\bar g$ close to $\bar g_0$, to extend the family $\mathrm D_{\bar g}(c)$ to a fixed larger manifold $X_e$ and a fixed extension $\cC_e$ of $\cC$. This can be done with a Seeley-type extension operator for instance (see the proof of Lemma \ref{lemma:interpolation} where such an extension is also used). With this choice, the resulting extended foliation depends continuously on $\bar g$ in the $C^\infty$ topology.

Then, by the description of the Schwartz kernel obtained in Theorem~\ref{theorem:wow}, near the diagonal we have
\[
K_{\bar g}(x,r,\omega)=\frac{\chi_{\bar g}(x,r,\omega)}{r},
\]
where $\chi_{\bar g}$ is smooth up to $r=0$. The construction of $\chi_{\bar g}$ in the proof of Theorem~\ref{theorem:wow} then easily shows that $\bar g\mapsto\chi_{\bar g}$ is continuous in the $C^\infty$ topology. Since, in this description, continuity with values in $\Psi^{-2}(X_e^\circ)$ is equivalent to continuity of the corresponding Schwartz kernels, it follows that $\bar g\mapsto P_{\bar g}\in\Psi^{-2}(X_e^\circ)$ is continuous. In particular, for every fixed seminorm of $\Psi^{-2}(X_e^\circ)$, only finitely many derivatives of $\bar g$ are involved, so the corresponding continuity holds in the $C^N$ topology for $N$ sufficiently large. 
\end{proof}

\subsubsection{Consequences}

The existence of an extension of the normal operator $\Pi_0$ implies the following stability estimate for the Radon transform:

\begin{lemma}
\label{lemma:soir}
Assume that $\bar g$ is $(\Phi,k)$-simple and $R_0 : C(X) \to C(\cC)$ is injective. Then there exists $C > 0$ such that for all $f \in C(X)$:
\begin{equation}
\label{equation:soir}
\|f\|_{H^{-2}(X)} \leq C \|R_0 f\|_{L^\infty(\cC)}.
\end{equation}
\end{lemma}

\begin{proof}
Let $\chi \in C^\infty_{\mathrm{comp}}(X_e^\circ)$ be a smooth compactly supported function such that $\chi \equiv 1$ on $X$, and $\chi$ is supported in a small neighborhood of $X$ contained in the elliptic set of $P$. By ellipticity, there exists $Q \in \Psi^{2}(X^\circ_e)$ and $K \in \Psi^{-\infty}(X^\circ_e)$ (with compactly supported Schwartz kernel) such that
\[
QP = \chi + K.
\]
This yields:
\[
\|\chi f \|_{H^{-2}(X_e)} \leq \|QP f\|_{H^{-2}(X_e)} + \|Kf\|_{H^{-2}(X_e)} \leq C\left(\|P f\|_{L^2(X_e)} + \|f\|_{H^{-2026}(X_e)} \right)
\]
We then specify this inequality to $E_0f$ where $E_0 : L^2(X) \to L^2(X_e)$ is the extension operator by $0$ outside of $X$, $f \in H^{-2}(X)$, and obtain
\begin{equation}
\label{equation:matin}
\|\chi E_0 f \|_{H^{-2}(X_e)} = \|f\|_{H^{-2}(X)} \leq C\left(\|P E_0 f\|_{L^2(X_e)} + \|E_0 f\|_{H^{-2026}(X)} \right).
\end{equation}
A standard argument by contradiction allows to remove the remainder term and yields
\begin{equation}
\label{equation:midi}
\|f\|_{H^{-2}(X)} \leq C\|P E_0 f\|_{L^2(X_e)}, \qquad \forall f \in H^2(X).
\end{equation}
Indeed, if \eqref{equation:midi} does not hold, then one can find a sequence $(f_n)_{n \geq 0}, f_n \in H^{-2}(X)$, such that
\[
\|f_n\|_{H^{-2}(X)} = 1 \geq n \|P E_0 f_n\|_{L^2(X_e)}.
\]
That is $P E_0 f_n \to 0$ in $L^2(X_e)$. However, applying \eqref{equation:matin} with $f_n-f_m$, and using that $H^{-2}(X) \hookrightarrow H^{-2026}(X)$ is a compact embedding, we see that, upon extracting a subsequence, $(f_n)_{n \geq 0}$ is a Cauchy sequence in $H^{-2}(X)$ which therefore converges to $f_\infty \in H^{-2}(X)$. In addition, $\|f_\infty\|_{H^{-2}(X)}=1$. By continuity, $P E_0 f_\infty = 0$, which implies that $f_\infty \in C^\infty(X)$ and $R_0 f_\infty = 0$, so $f_\infty \equiv 0$ by injectivity of $R_0$. This contradicts $\|f_\infty\|_{H^{-2}(X)}=1$ and proves \eqref{equation:midi}.

Finally, we assume that $f \in C(X)$. Using that $P = \pi_* R_e^*R_e \pi^*$ and $R_e \pi^* E_0 f$ has compact support in $\cC_e$, we obtain that
\[
\|P E_0 f\|_{L^2(X_e)} \leq C \|R_e \pi^* E_0 f \|_{L^2(\cC_e)} \leq C \|R \pi^* f \|_{L^2(\cC)} = C \|R_0 f\|_{L^2(\cC)} \leq C \|R_0 f\|_{L^\infty(\cC)},
\]
where $C > 0$ is a uniform constant. This proves the claim.
\end{proof}

Finally, we use the existence of an extension to show that the Radon transform is injective for generic metrics (Corollary \ref{corollary:kernel}):

\begin{proof}[Proof of Corollary \ref{corollary:kernel}]
\emph{(i-ii)} We apply Theorem \ref{theorem:extension-normal-operator}. First, observe that, by ellipticity of $P$ on $X$, the kernel $K := \{f \in C(X_e) ~:~ \mathrm{supp}(f) \subset X, Pf=0\}$ is finite-dimensional and contained in $C^\infty(X_e) \cap \{\mathrm{supp}(f) \subset X\}$. In particular, any function in $K$ is smooth on $X_e$, supported in $X$, and therefore vanishes to infinite order on $\partial X$.

Let $f \in C(X)$ such that $R_0 f = 0$. We extend $f$ by $0$ to a function $E_0f$ with compact support in $X$. By Theorem \ref{theorem:extension-normal-operator}, $P E_0 f = 0$ so $E_0 f \in K$ which is finite-dimensional. In addition, $E_0 f \in C^\infty(X_e)$ so $f \in C^\infty(X)$ and $f$ vanishes to infinite order on $\partial X$. \\

\emph{(iii)} If $\bar g$ is analytic, and $R_0 f = 0$, then $PE_0f = 0$ so $E_0 f$ is analytic by analytic elliptic regularity. Therefore, $E_0f \equiv 0$, that is $f \equiv 0$ so $R_0$ is injective for analytic metrics. We now deal with arbitrary generic metrics. We fix an area simple analytic metric $\bar g_0$.  Using Lemma \ref{lemma:continuity}, we deduce that 
\[
C^N(X,\mathrm{Sym}^2T^*X) \to \mathcal{L}(H^{-2}_{\mathrm{comp}}(X_e), L^2(X_e)), \qquad \bar g \mapsto P_{\bar g}
\]
is continuous for $N \gg 1$ sufficiently large. By continuity, the estimate \eqref{equation:midi} then holds for all metrics $\bar g$ which is $C^N$ close to $\bar g_0$. This yields injectivity of the corresponding surface Radon transform of $\bar g$.
\end{proof}

\section{Geometric rigidity results}

\label{section:applications}

We now discuss the applications of the framework introduced in the previous sections to the boundary area rigidity problem. Throughout this section, we consider the case $\Phi=\tr/2$, $k=0$, and assume that $\bar g$ is area simple (Definition \ref{definition:area-simple}).

\subsection{The boundary area spectrum}

Let $(X,\bar g)$ be a smooth Riemannian ball with area simple metric, and consider the foliation by minimal discs.

\begin{defi} The boundary area spectrum of $\bar g$ is the map
\[
\mathcal{A}_{\bar g} : \cC \to (0,\infty), \qquad \mathcal{A}_{\bar g}(c) :=\Area_{\bar g}(\mathrm{D}_{\bar g}(c)).
\]
\end{defi}

The next lemma asserts that the disc $D_{\bar g}(c)$ is the unique minimizer of the area functional amongst all surfaces bounded by $c$.

\begin{lemma}
\label{lemma:area-minimizer}
The disc $\mathrm{D}_{\bar g}(c)$ is the unique minimizer of the functional $S \mapsto \Area_{\bar g}(S)$, where $S$ is any embedded surface, of any topological type, bounded by $c$.
\end{lemma}

\begin{proof}
Since any minimizer is, in particular, minimal, it suffices to show that $\mathrm{D}_{\bar g}(c)$ is the unique minimal surface in $X$ bounded by $c$. Let $(c_t)_{t\in(-1,1)}$ be a smooth family of disjoint circles in $\partial X$, Hausdorff converging onto the distinct points $p_\pm$ at $t=\pm 1$, and such that $c_0=c$. For all $t$, denote $D_t:=\mathrm{D}_{\bar g}(c_t)$, and recall that the family $(D_t)_{t\in(-1,1)}$ foliates $X\setminus\{p_\pm\}$. Now let $S$ be another minimal surface in $X$ bounded by $c$. If $S$ is not equal to $\mathrm{D}_{\bar g}(c)=D_0$, then $S$ is an interior tangent to some other leaf $D_t$ of $(D_t)_{t\in(-1,1)}$ at some point. It follows by the strong geometric maximum principle that $S=D_t$, so that $c_0=\partial S=\partial D_t=c_t$. This is absurd, and uniqueness follows.
\end{proof}

Recall that the operator $R_2$ was introduced in \eqref{equation:s2}.

\begin{lemma}
\label{lemma:differential-area}
Consider a foliation of $X$ by minimal discs. Then:
\[
\partial_\eps \mathcal{A}_{{\bar g}+\eps h}(c)|_{\eps = 0} = \dfrac{1}{2} \int_{D(c)} \tr_g(h|_{T_xD(c)}(x)) \dd\Area_{D(c)}(x) = R_2h(c)/2
\]
%
\end{lemma}

In particular, given a conformal perturbation $e^{2 \eps f} \bar g$ of $\bar g$, we find
\begin{equation}
\label{equation:differential-conformal}
\partial_\eps \mathcal{A}_{e^{2\eps f}\bar g}(c)|_{\eps = 0} = 2 \int_{D(c)} f(x) \dd\Area_{D(c)}(x) = 2 R_0 f(c).
\end{equation}

\begin{proof}
Minimal surfaces are local minimizers of the area function among discs with a fixed boundary; the formula is immediate.
\end{proof}

The following result asserts that the boundary area spectrum depends in a $C^2$ fashion on the metric, provided it is smooth enough.

\begin{lemma}
The map
\[
C^{4}(X) \ni f \mapsto \mathcal{A}_{e^{2f}\bar g} \in L^\infty(\cC)
\]
is $C^2$. In particular, there exist $C, \eps_0 > 0$ such that for all $f \in C^{4}(X)$, $\|f\|_{C^4(X)} < \eps$:
\begin{equation}
\label{equation:taylor}
\|\mathcal{A}_{e^{2f}\bar g}-\mathcal{A}_{\bar g}-2R_0f\|_{L^\infty(\cC)} \leq C \|f\|^2_{C^{4}(X)}.
\end{equation}
\end{lemma}

\begin{proof}
The $C^2$ regularity of this map follows from the implicit function theorem applied in the proof of Lemma \ref{lemma:MDIsSubmanifold} with the metric $\bar g \in C^4$ as a parameter. See \cite[Theorem 2.2]{Busch-Liimatainen-Salo-Tzou-25} where this argument is written up in detail. The inequality \eqref{equation:taylor} follows from a Taylor expansion to order $2$ around $\bar g$, using \eqref{equation:differential-conformal}.
\end{proof}

\subsection{Local rigidity} \label{ssection:local-rigidity} We now establish the local rigidity of the boundary area spectrum of minimal surfaces (Theorem \ref{theorem:main}). It follows from the more general result:

\begin{theorem}
\label{theorem:estimee}
Let $(X,\bar g)$ be a three-dimensional Riemannian ball such that $\bar g$ is area simple with respect to $\cC$, and $R_0 : C(X) \to C(\cC)$ is injective. Then, there exists $C, \eps > 0$ such that for all $f \in C^{14}(X)$, $\|f\|_{C^{14}(X)} < \eps$, the following inequality holds:
\[
\|f\|_{H^{-2}(X)} \leq C  \|\mathcal{A}_{e^{2f}\bar g}- \mathcal{A}_{\bar g}\|_{L^\infty(\cC)}.
\]
In particular, if $\mathcal{A}_{e^{2f}\bar g} = \mathcal{A}_{\bar g}$, then $f = 0$.
\end{theorem}

We refer the reader to \S\ref{sssection:sobolev-boundary} for the definition of Sobolev spaces on manifolds with boundary. By Corollary \ref{corollary:kernel}, Theorem \ref{theorem:estimee} holds for an open dense subset of area simple metrics. We will use the interpolation Lemma \ref{lemma:interpolation} in the following form ($s=6,t=8$): there exists a constant $C > 0$ such that for all $f \in H^{14}(X)$
\begin{equation}
\label{equation:interpolation}
\|f\|^2_{H^6(X)} \leq C \|f\|_{H^{-2}(X)} \|f\|_{H^{14}(X)}.
\end{equation}

\begin{proof}[Proof of Theorem \ref{theorem:estimee}]
We use \eqref{equation:taylor} and Taylor expand the boundary area spectrum about $\bar g$:
\[
\mathcal{A}_{e^{2f}\bar g} = \mathcal{A}_{\bar g} + 2R_0 f + \mathcal{O}_{L^\infty(\cC)}(\|f\|^2_{C^{4}}).
\]
Note that the previous equality holds on $L^\infty(\cC)$.
Applying \eqref{equation:soir}, we obtain:
\[
\|f\|_{H^{-2}(X)} \leq C \|R_0 f\|_{L^\infty(\cC)} \leq C\left(\|\mathcal{A}_{e^{2f}g}- \mathcal{A}_g\|_{L^\infty(\cC)} +\|f\|^2_{C^{4}(X)}\right).
\]
We then use the embedding $H^{6}(X) \hookrightarrow C^{4}(X)$ and the estimate \eqref{equation:interpolation}, interpolating $H^{6}(X)$ between $H^{-2}(X)$ and $H^{14}(X)$:
\[
\begin{split}
\|f\|_{H^{-2}(X)} & \leq  C\left(\|\mathcal{A}_{e^{2f}\bar g}- \mathcal{A}_{\bar g}\|_{L^\infty(\cC)} +\|f\|_{H^{-2}(X)}\|f\|_{H^{14}(X)}\right) \\
& \leq C\left(\|\mathcal{A}_{e^{2f}\bar g}- \mathcal{A}_{\bar g}\|_{L^\infty(\cC)} +\|f\|_{H^{-2}(X)}\|f\|_{C^{14}(X)}\right).
\end{split}
\]
Assuming we have the \emph{a priori} bound $\|f\|_{C^{14}(X)} < \eps := 1/2C$, we obtain:
\[
\|f\|_{H^{-2}(X)} \leq C \|\mathcal{A}_{e^{2f}{\bar g}}- \mathcal{A}_{\bar g}\|_{L^\infty(\cC)},
\]
for some constant $C > 0$, which is the claimed estimate.
\end{proof}

\subsection{Global rigidity in the analytic case}

Finally, we prove Theorem \ref{theorem:global}.

\begin{proof}[Proof of Theorem \ref{theorem:global}]
Let $\bar g$ be an area simple metric, and let $f$ be a real analytic conformal factor such that $e^{2f} \bar g$ is area simple and
\[
\mathcal{A}_{\bar g}(c) = \mathcal{A}_{e^{2f} \bar g}(c), \qquad \forall c \in \cC.
\]
To show that $f \equiv 0$, we will prove that $f$ vanishes to infinite order on $\partial X$. We denote by $R_0^{\bar g}$ and $R_0^{e^{2f} \bar g}$ the surface Radon transforms of functions corresponding respectively to $\bar g$ and $e^{2f} \bar g$. Given a disc $\mathrm{D} \subset X$, we let $g$ denote the restriction of $\bar g$ to $\mathrm{D}$.

We begin with the following preliminary observations:
\begin{align}
\label{equation:radon1}
R_0^{\bar g}(e^{2f}-1)(c) \geq 0, & \qquad \forall c \in \cC \\
\label{equation:radon2}
R_0^{e^{2f} \bar g}(e^{-2f}-1)(c) \geq 0, & \qquad \forall c \in \cC.
\end{align}
Indeed, observe that, as the disc $\mathrm{D}_{e^{2f} \bar g}(c)$ minimizes the area within all discs with prescribed boundary $c$ (Lemma \ref{lemma:area-minimizer}), the following inequality holds:
\[
\mathcal{A}_{e^{2f}\bar g}(c) = \mathrm{Area}_{e^{2f} \bar g}(\mathrm{D}_{e^{2f} \bar g}(c)) \leq \mathrm{Area}_{e^{2f} \bar g}(\mathrm{D}_{\bar g}(c)) = \int_{\mathrm{D}_{\bar g}(c)} e^{2f} \dd \vol_{g} = R_0^{\bar g}e^{2f}(c)
\]
However, by hypotheses, $\mathcal{A}_{e^{2f}\bar g}(c) = \mathcal{A}_{\bar g}(c) = R_0^{\bar g}1(c)$. This proves \eqref{equation:radon1}. The proof of \eqref{equation:radon2} is the same upon reversing the roles of $\bar g$ and $e^{2f} \bar g$.

We will now use the following fact, derived from Corollary \ref{cor:LimitsOfSmallCircles}: if $x_0 \in \partial X$ and $(c_n)_{n \geq 0}$ is a family of circles collapsing to $x_0$, then $\mathrm{D}_{\bar g}(c_n)$ Hausdorff-convergence to $\{x_0\}$. In particular, for $n \gg 1$ large enough, $\mathrm{D}_{\bar g}(c_n)$ is contained in an arbitrarily small neighborhood of $x_0$. We will prove that the jets of $f$ vanish to infinite order on $\partial X$ by induction. First, let us establish that $f = 0$ on $\partial X$. If there exists $x_0 \in \partial X$ such that $f(x_0) > 0$, then $e^{-2f}-1 < 0$ near $x_0$, and $R_0^{e^{2f} \bar g}(e^{-2f}-1)(c_n) < 0$ for $n \gg 1$ large enough, where $(c_n)_{n \geq 0}$ is a family of circles collapsing to $x_0$. This contradicts \eqref{equation:radon2}. Similarly, $f(x_0) < 0$ is ruled out using \eqref{equation:radon1}, so $f = 0$ on $\partial X$.

Assume that in boundary normal coordinates $f(x,t) = t^m \tilde{f}(x,t)$ for some smooth function $\tilde{f}$, where $m \geq 1$, $x \in \partial X$ and $t \geq 0$. We want to show that $\tilde{f}(\cdot, 0)=0$. Assume for a contradiction that $\tilde{f}(x_0,0) > 0$ for some $x_0 \in \partial X$. Then
\[
e^{-2f}-1 = -2t^m \tilde{f}(x,t) + \mathcal{O}(t^{2m}) < 0
\]
near $(x_0,0)$. As in the previous paragraph, this then contradicts \eqref{equation:radon2} using a family of circles collapsing to $x_0$. Similarly, $\tilde{f}(x_0,0) < 0$ is also excluded using \eqref{equation:radon1} so $\tilde{f}(x_0,0)=0$ which proves the induction. This completes the proof.
\end{proof}

\appendix

\section{Microlocal analysis}

\label{appendix:pseudo}

\subsection{Pseudodifferential operators and Sobolev spaces}

\subsubsection{Standard properties} In this section, we recall standard properties of pseudodifferential operators and refer the interested reader to \cite{Hormander-65, Shubin-01, Lefeuvre-book} for details.

Let $Y$ be a smooth second-countable manifold without boundary (not necessarily compact). We denote by $\mathcal{E}'(Y)$ the space of compactly supported distributions (the topological dual of the Fréchet space of smooth densities) and $\mathcal{D}'(Y)$ the space of distributions (the dual of compactly supported smooth densities). Let $A : \mathcal{E}'(Y) \to \mathcal{D}'(Y)$ be a continuous operator. It is called \emph{smoothing} if it maps boundedly into $C^\infty(Y)$. The set of smoothing operators is denoted $\Psi^{-\infty}(Y)$.

Let $S^m(T^*Y)$ be the space of symbols of degree $m \in \R$ over $T^*Y$. This corresponds to the set of smooth functions $a \in C^\infty(T^*Y)$ such that in any local relatively compact patch of coordinates $U \subset Y$, the following estimate holds:
\[
|\partial_x^\beta \partial_\xi^\alpha a(x,\xi)| \leq C \langle\xi\rangle^{m-|\alpha|}, \qquad \forall (x,\xi) \in T^*U,
\]
where $\alpha, \beta \in \Z_{\geq 0}^n$, $C > 0$ is a uniform constant, and $\langle\xi\rangle := \sqrt{1+|\xi|^2}$, where $|\bullet|$ is the Euclidean norm in $\R^n$.

The left quantization of symbols in $\R^n$ is defined for $a \in S^m(T^*\R^n)$ and $f \in C^\infty_{\mathrm{comp}}(\R^n)$ as
\begin{equation}
\label{equation:op-rn}
\Op_{\R^n}(a)f(x) := \dfrac{1}{(2\pi)^n} \int_{\R^n_y \times \R^n_\xi} e^{i\xi(x-y)} a(x,\xi) f(y) \dd y \dd \xi.
\end{equation}
Similarly, given $a \in S^m(T^*Y)$, one can define a quantization $\Op(a)$ on $Y$ by using a partition of unity and local charts, and \eqref{equation:op-rn} in each chart, see e.g. \cite[\S5.2.3]{Lefeuvre-book}. Notice that the quantization map $\Op$ is not entirely canonical and depends on the local charts. The pseudodifferential algebra of $Y$ is then defined as
\[
\Psi^m(Y) := \{ \Op(a) + R ~:~ a \in S^m(T^*Y), R \in \Psi^{-\infty}(Y)\}.
\]
However, it can be checked that $\Psi^m(Y)$ is intrinsic to $Y$. Any operator $A \in \Psi^m(Y)$ is well-defined and bounded as a map $A : C^\infty_{\mathrm{comp}}(Y) \to C^\infty(Y)$ and $A : \mathcal{E}'(Y) \to \mathcal{D}'(Y)$.

Given $A \in \Psi^m(Y)$, its \emph{principal symbol} $[\sigma_A]$ is well-defined as an element
\[
[\sigma_A] \in S^m(T^*Y)/S^{m-1}(T^*Y),
\]
and for any representative $\sigma_A \in S^m(T^*Y)$, $A-\Op(\sigma_A) \in \Psi^{m-1}(Y)$. An operator is called \emph{elliptic} if its principal symbol (any representative of it) satisfies: for all compact subset $K \subset Y$, there exists $C > 0$ such that for all $x \in K, \xi \in T^*_xY$, and $|\xi| > C$
\[
|a(x,\xi)| \geq \langle \xi \rangle^m/C.
\]
The Japanese bracket is here computed with respect to an arbitrary background metric on $Y$.

A standard fact is that elliptic operators admit parametrices. Namely if $A \in \Psi^m(Y)$ is elliptic, then there exists $Q \in \Psi^{-m}(Y)$ and $R, R' \in \Psi^{-\infty}(Y)$ such that
\begin{equation}
\label{equation:parametrix}
A Q = \mathbf{1}- R, \qquad QA=\mathbf{1}-R'.
\end{equation}
We refer to \cite[\S5.3]{Lefeuvre-book} for details. The existence of a parametrix implies the following well-known elliptic regularity statement: if $u \in \mathcal{E}'(Y)$ is such that $A u = 0$, then $u \in C^\infty_{\mathrm{comp}}(Y)$. Indeed, $QA u = 0 = u-R'u$, that is $u = R' u \in C^\infty_{\mathrm{comp}}(Y)$ as $R'$ is smoothing. The same results hold in the analytic category, provided the pseudodifferential operators under consideration are analytic. See \cite{Treves-80} for details.

\subsubsection{A useful characterization}

It will be convenient at some point in the article to use the historical characterization of smooth ($C^\infty$) pseudodifferential due to Hörmander, see \cite[Definition 2.1]{Hormander-65}. In the following, we say that a set $A \subset C^\infty(Y)$ is \emph{bounded} if there exists a sequence $(A_k)_{k \geq 0}$ of positive real numbers and an exhaustion $Y = \cup_{k \geq 0} K_k$ by nested compact subsets $(K_k)_{k \geq 0}$ such that for all $f \in A, \|f\|_{C^k(K_k)} \leq A_k$. It is known that $A$ is compact if and only if it is closed and bounded.

An operator $P$ is a pseudodifferential operator of order $m \in \R$ if $P : C^\infty_{\mathrm{comp}}(Y) \to C^\infty(Y)$ is continuous, and there exists a sequence $s_0=0 < s_1 < ...$ of real numbers diverging to $+\infty$ such that for all $f \in C^\infty_{\mathrm{comp}}(Y),\, S \in C^\infty(Y)$ such that $\dd S \neq 0$ on $\mathrm{supp}(f)$, there is an asymptotic expansion:
\begin{equation}
\label{equation:pdo-expansion}
e^{-i S/h} P(f e^{iS/h}) \sim h^{-m} \sum_{j=0}^{+\infty} P_j(f,S) h^{s_j}.
\end{equation}
By this, we mean that for every integer $N > 0$, for every compact set $K$ of real-valued functions $S \in C^\infty(Y)$ with $\dd S \neq 0$ on the $\mathrm{supp}(f)$, for every $0 < h < 1$, the following holds: the error term
\begin{equation}
\label{equation:negligible}
h^{-s_N+m}\left(e^{-i S/h} P\left(f e^{i S/h}\right) - h^{-m}\sum_{j=0}^{N-1} P_j(f,S)h^{s_j}\right)
\end{equation}
belongs to a bounded set in $C^\infty(Y)$ with bound independent of $h$. In particular, $P$ is \emph{differential} if and only if the sum \eqref{equation:pdo-expansion} ranges over a finite number of $j$'s and $P$ is \emph{classical} if the $s_j$'s take integer values.

\subsubsection{Sobolev spaces on closed manifolds}

\label{sssection:sobolev}

We now further assume that $Y$ is a smooth closed manifold. Let $g$ be an arbitrary background metric, $\Delta_g \geq 0$ the induced Laplace operator. The Sobolev space $H^s(Y)$ is defined as the completion of $C^\infty(Y)$ with respect to the norm
\[
\|u\|^2_{H^s(Y)} := \|(\mathbf{1}+\Delta_g)^{s/2} u\|^2_{L^2(Y)},
\]
where $(\mathbf{1}+\Delta_g)^{s/2} \in \Psi^s(Y)$ is defined for $s \in \R$ thanks to the spectral theorem.

\subsubsection{Sobolev spaces on manifolds with boundary}

\label{sssection:sobolev-boundary}

In this paragraph, $Y$ is now assumed to be a smooth Riemannian manifold with boundary. Let $Y_e$ be a smooth closed Riemannian extension of $Y$, equipped with a metric $g_e$, such that $Y \hookrightarrow Y_e$ embeds isometrically. (This can be achieved by considering the double space of $Y$ obtained by gluing together two copies of $Y$ along the boundary and extending smoothly the metric $g$ to the other copy of $Y$.)

For $s \geq 0$, the Sobolev space $H^s(Y)$ is defined as the completion of $C^\infty(Y)$ with respect to the norm:
\[
\|f\|_{H^s(Y)} := \inf_{F \in H^s(Y_e), F|_Y=f} \|F\|_{H^s(Y_e)},
\]
while, for $s \leq 0$, $H^s(Y)$ is the completion of smooth functions with respect to the norm
\[
\|f\|_{H^{s}(Y)} := \|E_0 f\|_{H^s(Y_e)}, 
\]
where $E_0 : L^2(Y) \to L^2(Y_e)$ denotes the extension operator by $0$ outside of $Y$.

By construction, for $s \geq 0$, $H^{-s}(Y)$ identifies with the dual of $H^s(Y)$ \emph{via} the $L^2$ inner product on $Y$. That is, for $u \in H^s(Y), v \in H^{-s}(Y)$, letting $U \in H^s(Y_e)$ be an arbitrary extension such that $U|_Y=u$, it holds that
\[
|\langle u,v\rangle_{L^2(Y)}| = |\langle U, E_0 v \rangle_{L^2(Y_e)}| \leq \|U\|_{H^s(Y_e)} \|E_0v\|_{H^{-s}(Y_e)} \leq \|U\|_{H^s(Y_e)}\|v\|_{H^{-s}(Y)}.
\]
Taking the infimum over all such extensions $U$, we obtain
\[
|\langle u,v\rangle_{L^2(Y)}| \leq \|u\|_{H^s(Y)} \|v\|_{H^{-s}(Y)}.
\]

The following interpolation inequality is used in \S\ref{ssection:local-rigidity}:

\begin{lemma}
\label{lemma:interpolation}
Let $s \geq 0, t \geq s$. Then there exists $C := C(s,t) > 0$ such that for all $f \in H^{s+t}(Y)$:
\[
\|f\|_{H^s(Y)}^2 \leq C \|f\|_{H^{s+t}(Y)} \|f\|_{H^{s-t}(Y)}.
\]
\end{lemma}

\begin{proof}
Let $E : C^\infty(Y) \to C^\infty(Y_e)$ be a \emph{Seeley-type} extension operator (see \cite{Seeley-64}). That is, if $u \in C^\infty(Y)$, then $E u \in C^\infty(Y_e)$ and $E u|_{Y} = u$. Additionally, for all $s \in \R$, 
\begin{equation}
\label{equation:seeley}
E : H^s(Y) \to H^s(Y_e)
\end{equation}
is bounded. Note that, by duality,
\begin{equation}
\label{equation:seeley2}
E^* : H^s(Y_e) \to H^s(Y)
\end{equation}
is also bounded.

It then holds that:
\[
\begin{split}
\|f\|_{H^s(Y)}^2 & = \inf_{F \in H^s(Y_e), F|_Y=f} \|F\|^2_{H^s(Y_e)} \leq \|E f\|^2_{H^s(Y_e)} = \langle Ef, (\mathbf{1}+\Delta_{\bar g_e})^s Ef \rangle_{L^2(Y_e)} \\
& = \langle E_0 f, EE^*(\mathbf{1}+\Delta_{\bar g_e})^s Ef\rangle_{L^2(Y_e)} \leq \|E_0 f\|_{H^{s-t}(Y_e)}\|EE^*(\mathbf{1}+\Delta_{\bar g_e})^s Ef\|_{H^{-(s-t)}(Y_e)} \\
& = \|f\|_{H^{s-t}(Y)}\|EE^*(\mathbf{1}+\Delta_{\bar g_e})^s Ef\|_{H^{-(s-t)}(Y_e)}.
\end{split}
\]
Notice that it follows from \eqref{equation:seeley} and \eqref{equation:seeley2} that
\[
EE^*(\mathbf{1}+\Delta_{\bar g_e})^s E : H^{s+t}(Y) \to H^{-(s-t)}(Y_e)
\]
is bounded (by a constant $C > 0$). This provides the desired inequality:
\[
\|f\|_{H^s(Y)}^2 \leq C \|f\|_{H^{s-t}(Y)} \|f\|_{H^{s+t}(Y)}.
\]
\end{proof}

\subsection{A result on analytic pseudodifferential operators}

As above, we assume that $Y$ is a Riemannian manifold. For the sake of simplicity, we also further assume that $Y$ is $3$-dimensional but the arguments remain valid for higher dimensional manifolds up to the appropriate scaling ($1/d(x,y)$ becomes $1/d(x,y)^{n-2}$).

Let $P : C^\infty(Y) \to \cD'(Y)$ be a continuous, formally selfadjoint, operator with kernel $K \in \cD'(Y \times Y)$. That is
\[
\langle P u, v \rangle_{L^2(\vol_Y)} = \langle u, P^*v \rangle_{L^2(\vol_Y)}
\]
for all $u,v \in C^\infty(Y)$. This implies that $\iota^* K = K$ where $\iota(x,y):=(y,x)$ is the $\Z_2$ symmetry of $Y \times Y$. We let
\[
Y^{(2)} := [Y \times Y ; \Delta]
\]
be the blow up of $Y \times Y$ along the diagonal $\Delta \subset Y \times Y$. For a fixed $x \in Y$, we use polar coordinates $y = \exp_x(r\omega)$ around $x$. We will be interested in the case where, near $\Delta$, the kernel has the form
\begin{equation}
\label{equation:k}
K(x,r,\omega) = \dfrac{\chi(x,r,\omega)}{r}, \qquad \chi(x,r,\omega) = \psi(x) + \cO(r),
\end{equation}
where $\chi \in C^\infty(Y^{(2)})$ is smooth on the blown up space up to $r=0$. Note that, by Schur's test, $K \in L^2(Y \times Y, \vol_Y \times \vol_Y)$ so $P$ is compact (hence bounded) on $L^2(Y)$. Actually, $K$ is the kernel of a smooth pseudodifferential operator of order $-2$ (since $Y$ has dimension three).

\begin{prop}
\label{prop:analytic-pdo}
Assume that $Y$ is (analytically) diffeomorphic to the unit ball in $\R^3$, and that the function $\chi$ (see \eqref{equation:k}) is analytic with respect to $(x,r,\omega)$. Then $P$ is an analytic pseudodifferential operator in the sense of \cite{Treves-80}.
\end{prop}

We will use the following two facts in the proof:
\begin{equation}
\label{equation:utile}
\exp^{-1}_x(y) = A(x,y)(y-x), \qquad d_g(x,y) = \langle B(x,y)(y-x),(y-x)\rangle^{1/2},
\end{equation}
for some analytic matrix-valued map $A$ such that $A(x,x)=\mathbf{1}$, $\langle\bullet,\bullet\rangle$ is the Euclidean metric in $\R^3$, $B$ is analytic with respect to $x$ and $y$, and $B(x,x)=g(x)$ (that is the Riemannian metric $g$ seen as a symmetric $2$-tensor in Euclidean coordinates). We refer to \cite[Lemma 3.1]{Stefanov-Uhlmann-05} for a proof.

\begin{proof}
By assumption, up to an analytic diffeomorphism, $Y$ can be seen as an open subset in $\R^3$, which we will do from now on. Let
\[
v(x,y) := \exp_x^{-1}(y) = A(x,y)(y-x) = r(x,y)\omega(x,y), \qquad x,y \in Y,
\]
where $\omega(x,y) \in S_xX$. Then
\begin{equation}
\label{equation:romega}
r(x,y) := g_x(v(x,y),v(x,y))^{1/2}, \qquad \omega(x,y) = v(x,y)/r(x,y).
\end{equation}
Let $SY$ denote the unit tangent bundle of $Y$. The function $\chi$ is defined on an open subset $U'$ of the real analytic manifold $SY \times [0,\infty)$ containing $SY \times \{0\}$. By analyticity, it extends across $r=0$ and defines an analytic function on an open subset $U$ of $Z := SY \times \R$ containing $SY \times \{0\}$. Let $Z_\C$ denote the Grauert tube of $Z$, see \cite{Grauert-58}. Then $\chi$ admits an extension as a holomorphic function, still denoted by $\chi$, to an open subset $U_\C \subset Z_\C$ such that $U_\C \cap Z = U$.

By definition, the operator $P$ has Schwartz kernel $K \in \cD'(Y \times Y)$ satisfying
\[
Pf(x) = \int_Y K(x,y) \dd\vol_Y(y) = \int_Y \dfrac{\chi(x,r,\omega)}{d_g(x,y)} f(y) \sqrt{\det~g(y)} \dd y,
\]
where $g(y)$ is the symmetric $2$-tensor representing the (analytic) metric $g$ in the coordinates, and $\dd y$ is the Lebesgue measure. We can thus write
\[
Pf(x) = \int_Y k(x,x-y) f(y)\dd y,
\]
where the function $k(x,m)$ is given by (see \eqref{equation:utile} and \eqref{equation:romega})
\begin{equation}
\label{equation:kxm}
k(x,m) = \dfrac{\chi(x,r(x,x-m),\omega(x,x-m))}{\langle B(x,x-m)m,m\rangle^{1/2}} \sqrt{\det~g(x-m)},
\end{equation}
and defined on an open subset $V \subset \R^3 \times \R^3$ containing $Y \times \{0\}$. (Note that, if $x \in Y$ is fixed, then $V \supset (x,x-Y)$.)

We now use the characterization of the kernel of analytic pseudodifferential operators established in \cite{Baouendi-Goulaouic-Metivier-83}, see items (2'), (3') and (5') of that article (beware that $(x,m)$ in our conventions correspond to $(z,x)$ in \cite{Baouendi-Goulaouic-Metivier-83}). Since $P$ has order $-2$, it suffices to verify that:
\begin{enumerate}
\item[(i)] There is an open neighborhood $V_\C \subset \C^3 \times \C^3$ of $V \subset \R^3 \times \R^3$ and $h > 0$ such that $k$ extends holomorphically in
\[
V_\C \cap \{m \in \C^3 \setminus \{0\}, |\Im(m)| < h|\Re(m)|\},
\]
and for all compact subsets $K_\C \subset \pi_1(V_\C)$ (where $\pi_1 : \C^3 \times \C^3 \to \C^3$ is the projection onto the first factor), there is a $C > 0$ such that for $|\alpha|=0,1,2$
\begin{equation}
\label{equation:condition1}
|\partial^\alpha_m k(x,m)| \leq C |m|^{-(1+|\alpha|)}, \qquad \forall (x,m) \in V_\C, x \in K_\C, m \neq 0, |\Im(m)| < h|\Re(m)|.
\end{equation}
\item[(ii)] For all compact subset $K \subset Y$, there is a $C > 0$ such that for $|\alpha|=2$, and for all $\eps > 0$, for all $x \in K$:
\begin{equation}
\label{equation:condition2}
\left|\int_{x-Y, |m| > \eps} \partial^\alpha_m k(x,m) \dd m \right| \leq C.
\end{equation}
\end{enumerate}

\emph{Condition (i).} We begin by the following elementary observation: the function $m \mapsto 1/|m|=1/\sqrt{m^2}$ defined on $\R \setminus \{0\}$ extends holomorphically to $\C \cap \{|\Im(m)| < |\Re(m)|/2026\}$ after choosing the branch of the square root function $re^{i\theta} \mapsto r^{1/2} e^{i\theta/2}$. The same observation applies to the norm $m \mapsto 1/|m|=1/\sqrt{m_1^2+...+m_n^2}$ in $\R^n \setminus \{0\}$. The function $\chi$ being holomorphic in $U_\C$, and $A$ and $B$ being holomorphic on $V_\C$ (up to shrinking $V_\C$ a little if necessary), the previous fact implies that $k$ admits a holomorphic extension to $V_\C \cap  \{|\Im(m)| < h|\Re(m)|\}$, where $h > 0$ is chosen small enough (depending on $A$ and $B$). Using the expression \eqref{equation:kxm} for the kernel, the estimate \eqref{equation:condition1} is then straightforward to verify by a direct calculation. \\

\emph{Condition (ii).} Estimate \eqref{equation:condition1} shows that
\begin{equation}
\label{equation:oui}
\partial_{m_i} k(x,m) \leq C/|m|^2,
\end{equation}
for some uniform $C > 0$, for all $x \in K, m \in x-Y$. For simplicity, we assume that $\partial^\alpha_m = \partial_{m_1}^2$. Integrating by parts in \eqref{equation:condition2}, we find that for $\eps_0 > 0$ small enough:
\[
\begin{split}
\int_{x-Y, |m| > \eps} \partial^\alpha_m k(x,m) \dd m & = \int_{x-Y, |m| > \eps_0} \partial^\alpha_m k(x,m) \dd m   + \int_{|m|=\eps_0}  \nu_1(m) \partial_{m_1} k(x,m) \dd m \\
& \qquad  -\int_{|m|=\eps}  \nu_1(m) \partial_{m_1} k(x,m) \dd m,
\end{split}
\]
where $\nu(m)$ denotes the unit radial vector field in $\R^3$ and $\nu_i(m) := \langle \nu(m),\mathbf{e}_i\rangle$. In the previous equality, the first two terms on the right-hand side are independent of $\eps$ and thus uniformly bounded. As to the last term, using \eqref{equation:oui} and that the volume of the sphere $\{|m|=\eps\} \subset \R^3$ is $\leq C \eps^2$, we obtain \eqref{equation:condition2}.
\end{proof}

\section{Notation for asymptotic analysis}
\label{ss:Notation}

In order to carry out the asymptotic analysis of \S\ref{section:asymptotic}, we will use the following notation. Let $X$ be a smooth, finite-dimensional manifold, and let $U\subseteq\Bbb{R}^m$ be a neighbourhood of the origin. For $f\in C^\infty(X\times(U\setminus\{0\}))$, we denote
\begin{equation*}
f(x,y) = \opO(y^k)
\end{equation*}
whenever there exist smooth functions $f_1,\cdots,f_m\in C^\infty(X\times U)$ and homogeneous polynomials $p_1,\cdots,p_m$, all of order $k$, such that
\begin{equation*}
f(x,y) = \sum_{i=1}^m p_i(y)f_i(x,y)\ .
\end{equation*}
We underline that the usefulness of this notation stems precisely from the fact that the functions $f_1,\cdots,f_m$ are defined and smooth over the whole of $X\times U$, \emph{including at the origin} in $U$.

There are two main tools for working with this notation. The first is essentially Taylor's theorem.
\begin{lemma}
\label{lemma:TaylorsTheorem}
Let $X$ be a smooth, finite-dimensional manifold, and let $U\subseteq\mathbb{R}^m$ be a neighbourhood of the origin. For all $f\in C^\infty(U\times V)$ and for all positive integer $k$, there exist smooth functions $p_k:X\times\Bbb{R}^m\rightarrow\Bbb{R}$ and $f_{1,k+1},\cdots,f_{m,k+1}:X\times U\rightarrow\Bbb{R}$, and polynomials $q_{1,k+1},\cdots,q_{m,k+1}:\Bbb{R}^m\rightarrow\Bbb{R}$, such that
\begin{enumerate}[label=(\roman*)]
\item for all $x$, $p_k(x,\cdot)$ is a polynomial of order $k$;
\item for all $i$, $q_{i,k+1}$ is homogeneous of order $(k+1)$; and
\item
\begin{equation*}
f(x,y) = p_k(x,y) + \sum_{i=1}^m q_{i,k+1}(y)f_{i,k+1}(x,y)\ .
\end{equation*}
\end{enumerate}
\end{lemma}

The second tool is the fact that differentiation and composition by smooth functions define smooth functionals over H\"older spaces (see, for example, \cite{delaLlaveObaya1999,BourdaudLanzaDeCristoforis2002}, or \cite[Appendix A]{Smith2024AsymptoticGeometry}). This is made use of as follows. Let $X$ and $Y$ be smooth, finite-dimensional manifolds, possibly with boundary, and let $U\subseteq\Bbb{R}^m$ be a neighbourhood of the origin. Suppose that $Y$ is compact and, for all $k$, let $J^kY$ denote the bundle of $k$-jets of smooth functions over $Y$. Given a smooth function $F:X\times J^k Y\times (U\setminus\{0\})\rightarrow\Bbb{R}$, for all $(l,\alpha)$, we define the smooth functional $\mathcal{F}:X\times C^{l+k,\alpha}(Y)\times (U\setminus\{0\})\rightarrow C^{l,\alpha}(Y)$ by
\begin{equation*}
\mathcal{F}(x,f,z)(y) := F(x,J^kf(y),z)\ .
\end{equation*}
For all $m$,
\begin{equation*}
F(x,J^kf,z) = \opO(z^m)\ \Rightarrow\ \mathcal{F}(x,f,z) = \opO(z^m)\ ,
\end{equation*}
so that asymptotic formulae for smooth functionals over Banach manifolds follow from the analogous asymptotic formulae for smooth functions over finite-dimensional jet spaces. The latter are often derived using Taylor's theorem in the form stated above.

\section{Proper extensions}

\label{section:extension}

Let $N$ and $N_e$ be two smooth (resp. analytic) manifolds with boundary which are diffeomorphic as smooth (resp. analytic) manifolds. We say that an embedding $\iota : N \to N_e$ is a smooth (resp. analytic) \emph{proper extension} whenever $\iota(N) \subset N_e^\circ$. If $(N,g)$ and $(N_e,g_e)$ are both Riemannian, then we say that the embedding $\iota : (N,g) \to (N_e,g_e)$ is a \emph{proper Riemannian extension} whenever it is an isometry and a proper extension. In what follows, for simplicity, we will suppress the notation $\iota$ and view $N$ as a subset of $N_e$.

\begin{lemma}
\label{lemma:extension}
Let $(X,\bar g)$ be a smooth (resp. analytic) Riemannian ball. There exists a smooth (resp. analytic) proper Riemannian extension $(X_e,\bar g_e)$ of the smooth (resp. analytic) Riemannian manifold $(X,\bar g)$.
\end{lemma}

Note that, in the analytic case, the metric $\bar g_e$ is completely determined by $\bar g$.

\begin{proof}
The proof is immediate. Indeed, we consider first a collar neighborhood $[0,t_0)_t \times \partial X$ of the boundary $\partial X$ of $(X, \bar g)$ and we denote $X_e := X \sqcup (-t_0,0)_t \times \partial X$ for some small $t_0 > 0$. In the smooth case, the metric $\bar g$ can be smoothly extended to $X_e$ in an arbitrary fashion. In the analytic case, upon reducing $r_0$ if necessary, the jets of $\bar g$ on $\partial X$ fully determine $\bar g_e$ on $X_e \setminus X$.
\end{proof}

\bibliographystyle{alpha}
\bibliography{Biblio}

\begin{thebibliography}{GSUZ26}

\bibitem[ABN20]{Alexakis-Balehowsky-Nachman-20}
S.~Alexakis, T.~Balehowsky, and A.~Nachman.
\newblock Determining a {R}iemannian metric from minimal areas.
\newblock {\em Adv. Math.}, 366:107025, 71, 2020.

\bibitem[AG26]{Ambrozio-Guajardo-26}
L.~Ambrozio and D.~Guajardo.
\newblock Equivariant constructions of spheres with {Z}oll families of minimal
  spheres.
\newblock {\em Adv. Math.}, 496:Paper No. 110989, 57, 2026.

\bibitem[ALLS26]{Alvarez-Lefeuvre-Lowe-Smith}
S.~Alvarez, T.~Lefeuvre, B.~Lowe, and G.~Smith.
\newblock Stability and uniqueness of minimal disks in non-constant curvature.
\newblock {\em Preprint}, 2026.

\bibitem[ALS25a]{Alvarez-Lowe-Smith-25-1}
S.~Alvarez, B.~Lowe, and G.~Smith.
\newblock Foliated plateau problems and asymptotic counting of surface
  subgroups.
\newblock {\em Ann. Sci. \'Ec. Norm. Sup\'er. (4)}, 58(3):607--663, 2025.

\bibitem[ALS25b]{Alvarez-Lowe-Smith-25-2}
S.~Alvarez, B.~Lowe, and G.~Smith.
\newblock Rigidity of the hyperbolic marked energy spectrum and entropy for
  {$k$}-surfaces.
\newblock {\em J. \'Ec. polytech. Math.}, 12:1197--1227, 2025.

\bibitem[Alv25]{Alvarez-25}
S.~Alvarez.
\newblock Foliated {P}lateau problems, geometric rigidity and equidistribution
  of closed $k$-surfaces.
\newblock {\em to appear in {A}ctes du {S}\'eminaire de {T}h\'eorie {S}pectrale
  et {G}éométrie}, 2025.

\bibitem[AMN25]{Ambrozio-Marques-Neves-25}
L.~Ambrozio, F.~Marques, and A.~Neves.
\newblock Riemannian metrics on the sphere with {Z}oll families of minimal
  hypersurfaces.
\newblock {\em J. Differential Geom.}, 130(2):269--341, 2025.

\bibitem[And83]{Anderson-83}
M.~Anderson.
\newblock Complete minimal hypersurfaces in hyperbolic {$n$}-manifolds.
\newblock {\em Comment. Math. Helv.}, 58(2):264--290, 1983.

\bibitem[Aro57]{Aronszajn-57}
N.~Aronszajn.
\newblock A unique continuation theorem for solutions of elliptic partial
  differential equations or inequalities of second order.
\newblock {\em J. Math. Pures Appl. (9)}, 36:235--249, 1957.

\bibitem[BCG95]{Besson-Courtois-Gallot-95}
G.~Besson, G.~Courtois, and S.~Gallot.
\newblock Entropies et rigidit\'{e}s des espaces localement sym\'{e}triques de
  courbure strictement n\'{e}gative.
\newblock {\em Geom. Funct. Anal.}, 5(5):731--799, 1995.

\bibitem[BdC80]{Barbosa-doCarmo-80}
J.~L. Barbosa and M.~do~Carmo.
\newblock Stability of minimal surfaces and eigenvalues of the {L}aplacian.
\newblock {\em Math. Z.}, 173(1):13--28, 1980.

\bibitem[Ber69]{Berg}
C.~Berg.
\newblock Corps convexes et potentiels sph\'eriques.
\newblock {\em Mat.-Fys. Medd. Danske Vid. Selsk.}, 37(6):64, 1969.

\bibitem[BGM83]{Baouendi-Goulaouic-Metivier-83}
M.~S. Baouendi, C.~Goulaouic, and G.~M{\'e}tivier.
\newblock Kernels and symbols of analytic pseudodifferential operators.
\newblock {\em J. Differ. Equations}, 48:227--240, 1983.

\bibitem[BI10]{Burago-Ivanov-10}
D.~Burago and S.~Ivanov.
\newblock Boundary rigidity and filling volume minimality of metrics close to a
  flat one.
\newblock {\em Ann. of Math. (2)}, 171(2):1183--1211, 2010.

\bibitem[BI13]{Burago-Ivanov-13}
D.~Burago and S.~Ivanov.
\newblock Area minimizers and boundary rigidity of almost hyperbolic metrics.
\newblock {\em Duke Math. J.}, 162(7):1205--1248, 2013.

\bibitem[BLdC02]{BourdaudLanzaDeCristoforis2002}
G.~Bourdaud and M.~Lanza~de Cristoforis.
\newblock Functional calculus in h\"older-zygmund spaces.
\newblock {\em Trans. Amer. Math. Soc.}, 354(10):4109--4129, 2002.

\bibitem[BLP24]{Bohr-Lefeuvre-Paternain-24}
J.~Bohr, T.~Lefeuvre, and G.~Paternain.
\newblock Invariant distributions and the transport twistor space of closed
  surfaces.
\newblock {\em J. Lond. Math. Soc., II. Ser.}, 109(5):39, 2024.
\newblock Id/No e12894.

\bibitem[BLST25]{Busch-Liimatainen-Salo-Tzou-25}
L.~Busch, T.~Liimatainen, M.~Salo, and L.~Tzou.
\newblock Generalized boundary rigidity and minimal surface transform.
\newblock {\em Preprint, [arXiv:2510.23366]}, 2025.

\bibitem[CGL24]{Cekic-Guillarmou-Lefeuvre-24}
M.~Ceki\'c, C.~Guillarmou, and T.~Lefeuvre.
\newblock Local lens rigidity for manifolds of {A}nosov type.
\newblock {\em Anal. PDE}, 17(8):2737--2795, 2024.

\bibitem[Chr65]{Christoffel}
E.~B. Christoffel.
\newblock Ueber die {B}estimmung der {G}estalt einer krummen {O}berfl\"ache
  durch lokale {M}essungen auf derselben.
\newblock {\em J. Reine Angew. Math.}, 64:193--209, 1865.

\bibitem[CLT24]{Carstea-Liimatainen-Tzou-24}
L.~C{\^a}rstea, T.~Liimatainen, and L.~Tzou.
\newblock The {C}alder\'on problem on riemannian surfaces and of minimal
  surfaces.
\newblock {\em Preprint, [arXiv:2406.16944]}, 2024.

\bibitem[CM11]{Colding-Minicozzi-11}
T.~H. Colding and W.~P. Minicozzi, II.
\newblock {\em A course in minimal surfaces}, volume 121 of {\em Graduate
  Studies in Mathematics}.
\newblock American Mathematical Society, Providence, RI, 2011.

\bibitem[CMN22]{Calegari-Marques-Neves-22}
D.~Calegari, F.~Marques, and A.~Neves.
\newblock Counting minimal surfaces in negatively curved 3-manifolds.
\newblock {\em Duke Math. J.}, 171(8):1615--1648, 2022.

\bibitem[CNS88]{Caffarelli-Nirenberg-Spruck-88}
L.~Caffarelli, L.~Nirenberg, and J.~Spruck.
\newblock Nonlinear second-order elliptic equations. {V}. {T}he {D}irichlet
  problem for {W}eingarten hypersurfaces.
\newblock {\em Comm. Pure Appl. Math.}, 41(1):47--70, 1988.

\bibitem[Cro90]{Croke-90}
C.~B. Croke.
\newblock Rigidity for surfaces of nonpositive curvature.
\newblock {\em Comment. Math. Helv.}, 65(1):150--169, 1990.

\bibitem[Cro91]{Croke-91}
C.~B. Croke.
\newblock Rigidity and the distance between boundary points.
\newblock {\em J. Differential Geom.}, 33(2):445--464, 1991.

\bibitem[Cro04]{Croke-04}
C.~B. Croke.
\newblock Rigidity theorems in {R}iemannian geometry.
\newblock In {\em Geometric methods in inverse problems and {PDE} control},
  volume 137 of {\em IMA Vol. Math. Appl.}, pages 47--72. Springer, New York,
  2004.

\bibitem[dlLO99]{delaLlaveObaya1999}
R.~de~la Llave and R.~Obaya.
\newblock Regularity of the composition operator in spaces of h\"older
  functions.
\newblock {\em Discrete Contin. Dynam. Systems}, 5(1):157--184, 1999.

\bibitem[DZ16]{Dyatlov-Zworski-16}
S.~Dyatlov and M.~Zworski.
\newblock Dynamical zeta functions for {A}nosov flows via microlocal analysis.
\newblock {\em Ann. Sci. \'{E}c. Norm. Sup\'{e}r. (4)}, 49(3):543--577, 2016.

\bibitem[EGM09]{EspinarGalvezMira}
J.~Espinar, J.~G\'alvez, and P.~Mira.
\newblock Hypersurfaces in {$\Bbb H^{n+1}$} and conformally invariant
  equations: the generalized {C}hristoffel and {N}irenberg problems.
\newblock {\em J. Eur. Math. Soc. (JEMS)}, 11(4):903--939, 2009.

\bibitem[Fir68]{Firey}
W.~Firey.
\newblock Christoffel's problem for general convex bodies.
\newblock {\em Mathematika}, 15:7--21, 1968.

\bibitem[FS11]{Faure-Sjostrand-11}
F.~Faure and J.~Sj\"{o}strand.
\newblock Upper bound on the density of {R}uelle resonances for {A}nosov flows.
\newblock {\em Comm. Math. Phys.}, 308(2):325--364, 2011.

\bibitem[FV16]{FillastreVeronelli}
F.~Fillastre and G.~Veronelli.
\newblock Lorentzian area measures and the {C}hristoffel problem.
\newblock {\em Ann. Sc. Norm. Super. Pisa Cl. Sci. (5)}, 16(2):383--467, 2016.

\bibitem[G\.59]{Garding_1959}
Lars G\.arding.
\newblock An inequality for hyperbolic polynomials.
\newblock {\em J. Math. Mech.}, 8:957--965, 1959.

\bibitem[Ger03]{Gerhardt2003Hypersurfaces}
C.~Gerhardt.
\newblock Hypersurfaces of prescribed scalar curvature in {L}orentzian
  manifolds.
\newblock {\em J. Reine Angew. Math.}, 554:157--199, 2003.

\bibitem[Gra58]{Grauert-58}
H.~Grauert.
\newblock On {Levi}'s problem and the imbedding of real-analytic manifolds.
\newblock {\em Ann. Math. (2)}, 68:460--472, 1958.

\bibitem[Gro83]{Gromov-83}
M.~Gromov.
\newblock Filling {R}iemannian manifolds.
\newblock {\em J. Differential Geom.}, 18(1):1--147, 1983.

\bibitem[Gro91a]{Gromov-91-1}
M.~Gromov.
\newblock Foliated {P}lateau problem. {I}. {M}inimal varieties.
\newblock {\em Geom. Funct. Anal.}, 1(1):14--79, 1991.

\bibitem[Gro91b]{Gromov-91-2}
M.~Gromov.
\newblock Foliated {P}lateau problem. {II}. {H}armonic maps of foliations.
\newblock {\em Geom. Funct. Anal.}, 1(3):253--320, 1991.

\bibitem[GS04]{GuanSpruck2004Locally}
B.~Guan and J.~Spruck.
\newblock Locally convex hypersurfaces of constant curvature with boundary.
\newblock {\em Comm. Pure Appl. Math.}, 57(10):1311--1331, 2004.

\bibitem[GSUZ26]{Grebnev-Stefanov-Uhlmann-Zhou-26}
H.~{Grebnev}, P.~{Stefanov}, G.~{Uhlmann}, and H.~{Zhou}.
\newblock {The linearized minimal surfaces problem}.
\newblock {\em arXiv e-prints}, page arXiv:2606.24682, June 2026.

\bibitem[GT01]{Gilbarg-Trudinger-01}
D.~Gilbarg and N.~Trudinger.
\newblock {\em Elliptic partial differential equations of second order}.
\newblock Classics in Mathematics. Springer-Verlag, Berlin, 2001.
\newblock Reprint of the 1998 edition.

\bibitem[GU89]{Greenleaf-Uhlmann-89}
A.~Greenleaf and G.~Uhlmann.
\newblock Nonlocal inversion formulas for the {X}-ray transform.
\newblock {\em Duke Math. J.}, 58(1):205--240, 1989.

\bibitem[GU91]{Greenleaf-Uhlmann-91}
A.~Greenleaf and G.~Uhlmann.
\newblock Microlocal techniques in integral geometry.
\newblock Integral geometry and tomography, {Proc}. {AMS}-{IMS}-{SIAM} {Jt}.
  {Summer} {Res}. {Conf}., {Arcata}/{CA} ({USA}) 1989, {Contemp}. {Math}. 113,
  121-135 (1991)., 1991.

\bibitem[Gui76]{Guillemin-76}
V.~Guillemin.
\newblock The {Radon} transform on {Zoll} surfaces.
\newblock {\em Adv. Math.}, 22:85--119, 1976.

\bibitem[Gui85]{Guillemin-84}
V.~Guillemin.
\newblock On some results of {G}el'fand in integral geometry.
\newblock In {\em Pseudodifferential operators and applications ({N}otre
  {D}ame, {I}nd., 1984)}, volume~43 of {\em Proc. Sympos. Pure Math.}, pages
  149--155. Amer. Math. Soc., Providence, RI, 1985.

\bibitem[Gui17a]{Guillarmou-17-1}
C.~Guillarmou.
\newblock Invariant distributions and {X}-ray transform for {A}nosov flows.
\newblock {\em J. Differential Geom.}, 105(2):177--208, 2017.

\bibitem[Gui17b]{Guillarmou-17-2}
C.~Guillarmou.
\newblock Lens rigidity for manifolds with hyperbolic trapped sets.
\newblock {\em J. Amer. Math. Soc.}, 30(2):561--599, 2017.

\bibitem[Hel59]{Helgason-59}
S.~Helgason.
\newblock Differential operators on homogeneous spaces.
\newblock {\em Acta Math.}, 102:239--299, 1959.

\bibitem[HL13]{HarveyLawson_2013}
F.~R. Harvey and H.~B. Lawson, Jr.
\newblock G\aa rding's theory of hyperbolic polynomials.
\newblock {\em Comm. Pure Appl. Math.}, 66(7):1102--1128, 2013.

\bibitem[HL21]{HarveyLawson2021Pseudoconvexity}
F.~R. Harvey and H.~B. Lawson.
\newblock Pseudoconvexity for the special {L}agrangian potential equation.
\newblock {\em Calculus of Variations and Partial Differential Equations},
  60:1--37, 2021.

\bibitem[HLS26]{Huang-Lowe-Seppi-26}
Z.~Huang, B.~Lowe, and A.~Seppi.
\newblock Uniqueness and non-uniqueness for the asymptotic {P}lateau problem in
  hyperbolic space.
\newblock {\em Proc. Lond. Math. Soc. (3)}, 132(1):Paper No. e70121, 31, 2026.

\bibitem[H{\"{o}}r65]{Hormander-65}
L.~H{\"{o}}rmander.
\newblock Pseudo-differential operators.
\newblock {\em Comm. Pure Appl. Math.}, 18:501--517, 1965.

\bibitem[H{\"o}r07]{Hormander-07}
Lars H{\"o}rmander.
\newblock {\em The analysis of linear partial differential operators. {III}:
  {Pseudo}-differential operators}.
\newblock Class. Math. Berlin: Springer, reprint of the 1994 ed. edition, 2007.

\bibitem[KM12a]{Kahn-Markovic-12-2}
J.~Kahn and V.~Markovi\'c.
\newblock Counting essential surfaces in a closed hyperbolic three-manifold.
\newblock {\em Geom. Topol.}, 16(1):601--624, 2012.

\bibitem[KM12b]{Kahn-Markovic-12-1}
J.~Kahn and V.~Markovic.
\newblock Immersing almost geodesic surfaces in a closed hyperbolic three
  manifold.
\newblock {\em Ann. of Math. (2)}, 175(3):1127--1190, 2012.

\bibitem[Kur91]{Kurusa-91}
{\'A}.~Kurusa.
\newblock The {Radon} transform on hyperbolic space.
\newblock {\em Geom. Dedicata}, 40(3):325--339, 1991.

\bibitem[Lab97]{Labourie-97}
F.~Labourie.
\newblock Probl\`emes de {M}onge-{A}mp\`ere, courbes holomorphes et
  laminations.
\newblock {\em Geom. Funct. Anal.}, 7(3):496--534, 1997.

\bibitem[Lab00]{Labourie-00}
F.~Labourie.
\newblock Un lemme de {M}orse pour les surfaces convexes.
\newblock {\em Invent. Math.}, 141(2):239--297, 2000.

\bibitem[Lab02]{Labourie-02}
F.~Labourie.
\newblock The phase space of {$k$}-surfaces.
\newblock In {\em Rigidity in dynamics and geometry ({C}ambridge, 2000)}, pages
  295--307. Springer, Berlin, 2002.

\bibitem[Lab05]{Labourie-05}
F.~Labourie.
\newblock Random {$k$}-surfaces.
\newblock {\em Ann. of Math. (2)}, 161(1):105--140, 2005.

\bibitem[Lab21]{Labourie-21}
F.~Labourie.
\newblock Asymptotic counting of minimal surfaces and of surface groups in
  hyperbolic 3-manifolds.
\newblock {\em S\'eminaire Bourbaki, expos{\'e} 1179 Ast\'{e}risque},
  (430):Exp. No. 1179, 425--457, 2021.

\bibitem[Lef19]{Lefeuvre-19-1}
T.~Lefeuvre.
\newblock On the s-injectivity of the {X}-ray transform on manifolds with
  hyperbolic trapped set.
\newblock {\em Nonlinearity}, 32(4):1275--1295, 2019.

\bibitem[Lef20]{Lefeuvre-19-2}
T.~Lefeuvre.
\newblock Local marked boundary rigidity under hyperbolic trapping assumptions.
\newblock {\em J. Geom. Anal.}, 30(1):448--465, 2020.

\bibitem[Lef25]{Lefeuvre-book}
T.~Lefeuvre.
\newblock {\em Microlocal analysis in hyperbolic dynamics and geometry. {With}
  a contributed chapter by {Yann} {Chaubet}}, volume~32 of {\em Cours Sp{\'e}c.
  (Paris)}.
\newblock Paris: Soci{\'e}t{\'e} Math{\'e}matique de France (SMF), 2025.

\bibitem[L{\'o}p13]{Lopez-13}
R.~L{\'o}pez.
\newblock {\em Constant mean curvature surfaces with boundary}.
\newblock Springer Monogr. Math. Berlin: Springer, 2013.

\bibitem[Low21]{Lowe-21}
B.~Lowe.
\newblock Deformations of totally geodesic foliations and minimal surfaces in
  negatively curved 3-manifolds.
\newblock {\em Geom. Funct. Anal.}, 31(4):895--929, 2021.

\bibitem[Mic82]{Michel-81}
R.~Michel.
\newblock Sur la rigidit\'{e} impos\'{e}e par la longueur des
  g\'{e}od\'{e}siques.
\newblock {\em Invent. Math.}, 65(1):71--83, 1981/82.

\bibitem[Mor58a]{Morrey-58-1}
C.~B. Morrey.
\newblock On the analyticity of the solutions of analytic non-linear elliptic
  systems of partial differential equations. {I}. {A}nalyticity in the
  interior.
\newblock {\em Amer. J. Math.}, 80:198--218, 1958.

\bibitem[Mor58b]{Morrey-58-2}
C.~B. Morrey.
\newblock On the analyticity of the solutions of analytic non-linear elliptic
  systems of partial differential equations. {II}. {A}nalyticity at the
  boundary.
\newblock {\em Amer. J. Math.}, 80:219--237, 1958.

\bibitem[MST23]{Mazzuchelli-Salo-Tzou-23}
M.~Mazzuchelli, M.~Salo, and L.~Tzou.
\newblock A general support theorem for analytic double fibration transforms.
\newblock {\em Preprint, [arXiv:2306.05906]}, 2023.

\bibitem[Muk77]{Mukhometov-77}
R.~G. Mukhometov.
\newblock The reconstruction problem of a two-dimensional {R}iemannian metric,
  and integral geometry.
\newblock {\em Dokl. Akad. Nauk SSSR}, 232(1):32--35, 1977.

\bibitem[Muk81]{Mukhometov-81}
R.~G. Mukhometov.
\newblock On a problem of reconstructing {R}iemannian metrics.
\newblock {\em Sibirsk. Mat. Zh.}, 22(3):119--135, 237, 1981.

\bibitem[Ota90]{Otal-90-2}
J.-P. Otal.
\newblock Sur les longueurs des g\'{e}od\'{e}siques d'une m\'{e}trique \`a
  courbure n\'{e}gative dans le disque.
\newblock {\em Comment. Math. Helv.}, 65(2):334--347, 1990.

\bibitem[Pat99]{Paternain-99}
G.~Paternain.
\newblock {\em Geodesic flows}, volume 180 of {\em Progress in Mathematics}.
\newblock Birkh\"{a}user Boston, Inc., Boston, MA, 1999.

\bibitem[PSU13]{Paternain-Salo-Uhlmann-13}
G.~Paternain, M.~Salo, and G.~Uhlmann.
\newblock Tensor tomography on simple surfaces.
\newblock {\em Invent. Math.}, 193(1):229--247, 2013.

\bibitem[PSU23]{Paternain-Salo-Uhlmann-23}
G.~Paternain, M.~Salo, and G.~Uhlmann.
\newblock {\em Geometric inverse problems---with emphasis on two dimensions},
  volume 204 of {\em Cambridge Studies in Advanced Mathematics}.
\newblock Cambridge University Press, Cambridge, 2023.
\newblock With a foreword by Andr\'as Vasy.

\bibitem[PU05]{Pestov-Uhlmann-05}
L.~Pestov and G.~Uhlmann.
\newblock Two dimensional compact simple {R}iemannian manifolds are boundary
  distance rigid.
\newblock {\em Ann. of Math. (2)}, 161(2):1093--1110, 2005.

\bibitem[RS94]{Rosenberg-Spruck-94}
H.~Rosenberg and J.~Spruck.
\newblock On the existence of convex hypersurfaces of constant {G}auss
  curvature in hyperbolic space.
\newblock {\em J. Differential Geom.}, 40(2):379--409, 1994.

\bibitem[San52]{Santalo-52}
L.~A. Santal\'{o}.
\newblock Measure of sets of geodesics in a {R}iemannian space and applications
  to integral formulas in elliptic and hyperbolic spaces.
\newblock {\em Summa Brasil. Math.}, 3:1--11, 1952.

\bibitem[See64]{Seeley-64}
R.~Seeley.
\newblock Extension of {{\(C^ \infty\)}} functions defined in a half space.
\newblock {\em Proc. Am. Math. Soc.}, 15:625--626, 1964.

\bibitem[Sem61]{Semyanistyj-61}
V.~I. Semyanistyj.
\newblock Homogeneous functions and some problems of integral geometry in
  spaces of constant curvature.
\newblock {\em Sov. Math., Dokl.}, 2:59--62, 1961.

\bibitem[Sha94]{Sharafutdinov-94}
V.~A. Sharafutdinov.
\newblock {\em Integral geometry of tensor fields}.
\newblock Inverse and Ill-posed Problems Series. VSP, Utrecht, 1994.

\bibitem[Shu01]{Shubin-01}
M.~A. Shubin.
\newblock {\em Pseudodifferential operators and spectral theory}.
\newblock Springer-Verlag, Berlin, second edition, 2001.
\newblock Translated from the 1978 Russian original by Stig I. Andersson.

\bibitem[Smi13]{Smith2013SpecialLagrangian}
G.~Smith.
\newblock Special {L}agrangian curvature.
\newblock {\em Math. Ann.}, 355(1):57--95, 2013.

\bibitem[Smi20]{Smith-20}
G.~Smith.
\newblock The {P}lateau problem for convex curvature functions.
\newblock {\em Ann. Inst. Fourier (Grenoble)}, 70(1):1--66, 2020.

\bibitem[Smi21]{Smith-21}
G.~Smith.
\newblock On the asymptotic {P}lateau problem in {C}artan--{H}adamard
  manifolds.
\newblock {\em Preprint, [arXiv:2107.14670]}, 2021.

\bibitem[Smi24]{Smith2024AsymptoticGeometry}
G.~Smith.
\newblock On the asymptotic geometry of finite-type $k$-surfaces in
  three-dimensional hyperbolic space.
\newblock {\em J. Eur. Math. Soc. (JEMS)}, 26(2):407--467, 2024.

\bibitem[SU04]{Stefanov-Uhlmann-04}
P.~Stefanov and G.~Uhlmann.
\newblock Stability estimates for the {X}-ray transform of tensor fields and
  boundary rigidity.
\newblock {\em Duke Math. J.}, 123(3):445--467, 2004.

\bibitem[SU05]{Stefanov-Uhlmann-05}
P.~Stefanov and G.~Uhlmann.
\newblock Boundary rigidity and stability for generic simple metrics.
\newblock {\em J. Amer. Math. Soc.}, 18(4):975--1003, 2005.

\bibitem[SUV21]{Stefanov-Uhlmann-Vasy-21}
P.~Stefanov, G.~Uhlmann, and A.~Vasy.
\newblock Local and global boundary rigidity and the geodesic {X}-ray transform
  in the normal gauge.
\newblock {\em Ann. of Math. (2)}, 194(1):1--95, 2021.

\bibitem[Tr{\`e}80]{Treves-80}
F.~Tr{\`e}ves.
\newblock Introduction to pseudodifferential and {Fourier} integral operators.
  {Vol}. 1: {Pseudodifferential} operators. {Vol}. 2: {Fourier} integral
  operators.
\newblock The {University} {Series} in {Mathematics}. {New} {York}, {London}:
  {Plenum} {Press}. {XXXIX}, {XXV}, 649 p. (1980)., 1980.

\bibitem[Uhl83]{Uhlenbeck-83}
K.~Uhlenbeck.
\newblock Closed minimal surfaces in hyperbolic {$3$}-manifolds.
\newblock In {\em Seminar on minimal submanifolds}, volume 103 of {\em Ann. of
  Math. Stud.}, pages 147--168. Princeton Univ. Press, Princeton, NJ, 1983.

\bibitem[UV16]{Uhlmann-Vasy-16}
G.~Uhlmann and A.~Vasy.
\newblock The inverse problem for the local geodesic ray transform.
\newblock {\em Invent. Math.}, 205(1):83--120, 2016.

\bibitem[Wan06]{Wang-06}
R.~Wang.
\newblock Analyticity of solutions of analytic non-linear general elliptic
  boundary value problems, and some results about linear problems.
\newblock {\em Front. Math. China}, 1(3):382--429, 2006.

\bibitem[Whi87]{White-87}
B.~White.
\newblock The space of {$m$}-dimensional surfaces that are stationary for a
  parametric elliptic functional.
\newblock {\em Indiana Univ. Math. J.}, 36(3):567--602, 1987.

\end{thebibliography}

\end{document}